\documentclass[11pt,reqno]{amsart}

\usepackage{amsmath, amsfonts, amsthm, amssymb,mathrsfs, amscd, graphicx, color}

\newtheorem{thm}{Theorem}[section]

\newtheorem{Pros}{Proposition}[section]
\newtheorem{lemma}{Lemma}[section]

\theoremstyle{definition}
\newtheorem{define}{Definition}[section]

\theoremstyle{remark}
\newtheorem{rem}{Remark}[section]

\numberwithin{equation}{section}

\allowdisplaybreaks

\newcommand{\na}{\nabla}

\def\f{\frac}

\def\hf1{^\f{1}{1-\xi^2}}

\def\be{\begin{equation}}
\def\en{\end{equation}}
\def\bs{\begin{split}}
\def\es{\end{split}}

\newcommand{\beq}{\begin{equation}}
\newcommand{\eeq}{\end{equation}}
\newcommand{\ben}{\begin{eqnarray}}
\newcommand{\een}{\end{eqnarray}}
\newcommand{\beno}{\begin{eqnarray*}}
	\newcommand{\eeno}{\end{eqnarray*}}

\allowdisplaybreaks \voffset=-0.2in \numberwithin{equation}{section}

\title[Global regularity of micropolar equations]
{Global regularity of the three-dimensional fractional micropolar equations}

\author{Dehua Wang, Jiahong Wu  and Zhuan Ye }
\address{Department of Mathematics, University of Pittsburgh,
                           Pittsburgh, PA 15260, USA}
\email{dwang@math.pitt.edu}

\address{Department of Mathematics, Oklahoma State University, 401 Mathematical
Sciences, Stillwater, OK 74078, USA}
\email{jiahong.wu@okstate.edu}

\address{Department of Mathematics and Statistics, Jiangsu Normal University,  Xuzhou, Jiangsu 221116, P. R. China (Zhuan Ye is the corresponding author)}
\email{yezhuan815@126.com}

\date{\today}

\begin{document}

\begin{abstract}
The global well-posedness of the smooth solution to the three-dimensional (3D) incompressible micropolar equations is a difficult open problem.
This paper focuses on the 3D incompressible micropolar equations with fractional dissipations $(-\Delta)^{\alpha}u$ and $(-\Delta)^{\beta}w$.
Our objective is to establish the global regularity of the fractional micropolar equations with the minimal amount of dissipations. We prove that, if  $\alpha\geq \frac{5}{4}$, $\beta\geq 0$ and $\alpha+\beta\geq\frac{7}{4}$,   the fractional 3D micropolar equations always possess a unique global classical solution for any sufficiently smooth data. In addition, we also obtain the global regularity of the 3D micropolar equations with the dissipations given by Fourier multipliers that are logarithmically weaker than the fractional Laplacian.
\end{abstract}
\maketitle

\section{Introduction}

The classical micropolar equations,  first derived by Eringen \cite{Eringen1} in 1966, govern the motion of micropolar fluids that are fluids with microstructures and   nonsymmetric stress tensors.
Micropolar fluids usually contain rigid and randomly oriented or spherical particles that have their own spins and microrotations  in a viscous medium.
Special examples of micropolar fluids are  ferrofluids, blood flows, bubbly liquids, liquid crystals, and so on. More background information can be found in \cite{Cowin, Erdo, Eringen1, Er2, Lukaszewicz2}. Mathematically the micropolar equations consist of the forced Navier-Stokes equations for the fluid velocity $u$ and the evolution equation for the microrotation $w$ representing the angular velocity of rotation of particles of the fluid.
More precisely, the three-dimensional (3D) incompressible micropolar equations are 
\begin{equation}\label{C3DMP}
\left\{\begin{array}{l}
\partial_t u+(u\cdot\nabla) u
       -(\nu +\kappa)\Delta u +\nabla p = 2\kappa\nabla \times w ,\vspace{2mm} \qquad x\in \mathbb{R}^{3},\,t>0,\\
\partial_tw + (u\cdot\nabla)  w +4\kappa w
            -\gamma\Delta w =2\kappa\nabla\times u+\mu \nabla \nabla\cdot w,
             \vspace{2mm}\\
\nabla\cdot u=0,\vspace{2mm}\\
u(x,0)=u_{0}(x),\quad w(x,0)=w_{0}(x),
\end{array}\right.
\end{equation}
where $u=u(x,t)=(u_{1}(x,t),u_{2}(x,t),u_{3}(x,t))$ represents the fluid velocity,  $w=w(x,t)=(w_{1}(x,t),w_{2}(x,t),w_{3}(x,t))$ the microrotation, $p(x,t)$ the pressure, and the
parameter $\nu$ denotes the kinematic viscosity, $\kappa$ the
microrotation viscosity,   $\gamma$ and $\mu$ the angular
viscosities, and  $u_{0}$ and $w_{0}$
are the prescribed initial data for the velocity and microrotation with the property $\nabla\cdot u_{0}=0$.

Whether or not the smooth solutions of the classical 3D micropolar equations \eqref{C3DMP} with general initial data develop finite time singularities is a difficult problem and remains open. The key difficulty is that the Laplacian dissipation is not sufficient to control the nonlinearity when applying the standard techniques to derive global a priori bounds. It is very natural to ask the question of how much dissipation one really needs in order to ensure the global
regularity. Replacing the Laplacian operators by general fractional Laplacian operators exactly serves this purpose. We are naturally led to the  following 3D fractional micropolar equations
\begin{equation}\label{3DMP}
\left\{\begin{array}{l}
\partial_t u+(u\cdot\nabla) u
       +(\nu +\kappa)(-\Delta)^{\alpha} u +\nabla p = 2\kappa\nabla \times w ,\vspace{2mm} \qquad x\in \mathbb{R}^{3},\,t>0,\\
\partial_tw + (u\cdot\nabla)  w +4\kappa w
            +\gamma(-\Delta)^{\beta} w =2\kappa\nabla\times u+\mu \nabla \nabla\cdot w,
             \vspace{2mm}\\
\nabla\cdot u=0,\vspace{2mm}\\
u(x,0)=u_{0}(x),\quad w(x,0)=w_{0}(x),
\end{array}\right.
\end{equation}
where the fractional powers $\alpha$ and $\beta$ are nonnegative and the fractional Laplacian operator $(-\Delta)^{\rho}$ is defined via the Fourier transform
$$\widehat{(-\Delta)^{\rho} f} (\xi) = |\xi|^{2\rho}\, \widehat{f}(\xi).$$
We will adopt the convention that $\alpha=0$ really means that there is no dissipation in $\eqref{3DMP}_{1}$, and  $\beta=0$ indicates that there is no dissipation in $\eqref{3DMP}_{2}$. For simplicity, we will frequently use the notation $\Lambda:=(-\Delta)^{\frac{1}{2}}$.

The system (\ref{3DMP}) with the fractional Laplacian operator is physically relevant. The fractional Laplacian operator can model various anomalous diffusion. Especially, (\ref{3DMP}) allows us to study long-range diffusive interactions. In addition,  (\ref{3DMP}) with hyperviscosity can be used in turbulence modeling
to control the effective range of the non-local dissipation and to make numerical resolutions more efficient (\cite{UU}). Studies on models with fractional Laplacian have recently gained a lot of momentum. Two recent monographs, entitled ``Nonlocal Diffusion and Applications"
\cite{BValdinoci} and ``The Fractional Laplacian" \cite{Pozrikidis}, have documented a list of
applications. Closely related to (\ref{3DMP}) are several fractionally dissipated equations including the surface quasi-geostrophic equation, the porous medium equation and the
Boussinesq equations. The surface quasi-geostrophic equation with fractional dissipation
arises in geophysical studies of strongly rotating fluid flows (see, e.g., \cite{P00}).
The porous medium equation with fractional dissipation models various anomalous diffusion
process (see, e.g., \cite{VVa}).
The fractional Boussinesq equations are used to model the anomalous attenuation to take account of the effect of kinematic and thermal
diffusion in the study of viscous flows in the thinning of atmosphere (see \cite{Gill}).
The fractional Laplacian has now found applications far beyond fluid mechanics. It is used in fractional quantum mechanics \cite{Laskin}, probability \cite{Applebaum,Bertoinl}, overdriven detonations in gases \cite{Clavind}, anomalous diffusion in semiconductor growth \cite{Woyczyh}, physics and chemistry \cite{Metzlerk}, optimization and finance \cite{ConTankov}.

In addition to their applications in engineering and physics, the micropolar equations are also mathematically significant and have attracted considerable
attention in the community of mathematical fluids. Fundamental mathematical issues such as the global regularity of   solutions have been investigated extensively  
with many interesting results 
(see, e.g.,  \cite{Boldrinidr,EGaldir,chenmm,Lukaszewicz1,
Lukaszewicz2,Lukaszewicz3,Rojaszammmc,Nowakowski,Yamaguchi}). We mention some of
the results on the 3D micropolar equations \eqref{C3DMP} that are relevant to our study in this paper. The  weak solutions
of the initial boundary-value problem for \eqref{C3DMP} was first studied by Galdi and Rionero \cite{EGaldir}. Lukaszewicz in \cite{Lukaszewicz1} used linearization and an almost fixed point thereom to establish the global existence of weak solutions for sufficiently regular initial data. By the same approach, Lukaszewicz \cite{Lukaszewicz2} established the local and global well-posedness results under asymmetric conditions. Boldrini, Dur$\rm\acute{a}$n and Rojas-Medar \cite{Boldrinidr} proved, by the Galerkin method, the local existence and uniqueness of strong solutions to the initial and boundary-value problem for bounded or unbounded domains. Yamaguchi \cite{Yamaguchi} proved, via the $L^p$-semigroup approach,  the global existence of small classical solutions in bounded domains.

There is also an array of exciting results on the 2D  micropolar equations. We  mention some of them. The 2D micropolar equations are obtained by setting
$$
u=(u_1(x_1,x_2,t),u_2(x_1,x_2,t),0), \quad w=(0,0,w_3(x_1,x_2,t)), \quad p=p(x_1,x_2,t)
$$
in \eqref{C3DMP}. The corresponding 2D fractional micropolar equations can be written as
\begin{equation}\label{2DMP}
\left\{\begin{array}{l}
\partial_t u+(u\cdot\nabla) u
       +(\nu +\kappa)(-\Delta)^{\alpha} u +\nabla p = 2\kappa\nabla^{\perp} w ,\vspace{2mm} \qquad x\in \mathbb{R}^{2},\,t>0,\\
\partial_tw + (u\cdot\nabla)  w +4\kappa w
            +\gamma(-\Delta)^{\beta} w =2\kappa\nabla^{\perp}\cdot u,
             \vspace{2mm}\\
\nabla\cdot u=0,\vspace{2mm}\\
u(x,0)=u_{0}(x),\quad w(x,0)=w_{0}(x),
\end{array}\right.
\end{equation}
where $\nabla^{\perp}:=(-\partial_{x_{2}},\,\partial_{x_{1}})^{\mathbb{T}}$.
When $\alpha=\beta=1$, any sufficiently smooth initial data leads to a unique
global solution (see, e.g., \cite{Lukaszewicz3,ConstantinF,MB2011}). In
the case of inviscid micropolar equations, namely $\alpha=\beta=0$,
the global regularity
problem turns out to be extremely difficult and remains outstandingly open.
In the last few years the 2D micropolar equations with partial dissipation have
attracted considerable interests. In \cite{DZ} Dong and Zhang obtained the global
regularity for \eqref{2DMP} with $\alpha=1$ and $\beta=0$. Xue \cite{Xue}
reexamined the case $\alpha=1,\,\beta=0$, and established the global
well-posedness in the frame work of Besov spaces. The case with  $\alpha=0$
and $\beta=1$ was settled by Dong, Li and Wu \cite{DLW}, who  proved
the global well-posedness of \eqref{2DMP} in $\mathbb R^2$, and was also able to
determine the precise larg-time behavior of the solutions to this only partially dissipated system via the process of diagonalizing the linearized system.  The corresponding initial and boundary-value problem  for (\ref{2DMP}) with $\alpha=0$
and $\beta=1$  was examined by Jiu, Liu, Wu and Yu \cite{JiuLiu}, who established the global
well-posedness in the functional setting $u_0\in H^1$ and $w_0\in H^2$. We also mention a recent preprint of Liu and Wang on the initial and boundary-value problem
for the case $\alpha=1,\,\beta=0$ \cite{LiuWang}.  Very recently, Dong, Wu, Xu and Ye \cite{Dwxye} investigated the 2D fractional micropolar equations (\ref{2DMP}) and obtained the global well-posedness for the fractional
powers $\alpha$ and $\beta$ in suitable ranges. These global regularity results are not trivial and  were achieved by
fully exploiting the fine structures of the corresponding vorticity equation and by working with carefully selected combined quantities.

To put the results of this paper in a proper broad perspective, we give a brief
description of closely related global regularity results for several set of associated systems. The first one is the incompressible Navier-Stokes equations with
hyperviscosity
\begin{equation}\label{NS}
\left\{\begin{array}{l}
\partial_t u+(u\cdot\nabla) u
+(-\Delta)^{\alpha} u +\nabla p = 0,\vspace{2mm} \qquad x\in \mathbb{R}^{3},\,t>0,\\
\nabla\cdot u=0.
\end{array}\right.
\end{equation}
Lions' book \cite{Lions} establishes the global well-posedness of (\ref{NS}) for $\alpha>\frac54$. Katz and
Pavlovi\'{c} \cite{Katz} has a nice proof for the case when $\alpha>\frac54$. An elementary and simple proof for the global regularity of (\ref{NS}) with $\alpha\ge \frac54$ is given in \cite{wujde3}.  T. Tao \cite{ttTao} was  able to reduce the hyperdissipation by a logarithm and still establish the global regularity. Tao's work was further improved by \cite{Barbatomr}.  The case $\alpha=\frac54$ is a critical index from the scaling point of view and
it is not possible at this stage to obtain the global regularity for any $\alpha<\frac54$. Very recently Yang, Jiu and Wu \cite{Yang1} sharpened this classical
result from a different perspective and established the global regularity even with one-directional hyperviscosity removed from each velocity component equation. Another closely related system is the system of Boussinesq equations.
The Boussinesq equations are among the most frequently used models for geophysical fluids. The fractional dissipated version is given by
\begin{equation}\label{3Bous}
\left\{\begin{array}{l}
\partial_t u+(u\cdot\nabla) u
+(-\Delta)^{\alpha} u +\nabla p =\theta e_{3},\vspace{2mm} \qquad x\in \mathbb{R}^{3},\,t>0,\\
\partial_t \theta + (u\cdot\nabla)\theta +(-\Delta)^{\beta}\theta=0,
\vspace{2mm}\\
\nabla\cdot u=0,
\end{array}\right.
\end{equation}
where $u(x,\,t)$ is the velocity, $p$ the pressure and $\theta$
the temperature,  and $e_{3}$ is the unit vector $(0,\,0,\,1)$. It is shown in \cite{jiuYu,Xiang,Yamazaki,Yeat15} that any smooth initial data of \eqref{3Bous} leads to a unique global
 solution if
\be \label{BouCon}
\alpha\geq \frac{5}{4},\quad \beta\geq0.
\en
A logarithmic improvement was obtained in  \cite{ye2017good}. Very recently \eqref{3Bous} with (\ref{BouCon}) is shown to possess a unique global weak solution
when the initial data $(u_0, \theta_0)$ is only in $L^2(\mathbb R^3)$ (see \cite{Nicki,Dai}).  These results do not require any thermal diffusion. Another closely related system
is the 3D incompressible fractional magnetohydrodynamic (MHD) equations, which are given by
\begin{equation}\label{3DMHD}
\left\{\begin{array}{l}
\partial_t u+(u\cdot\nabla) u
+(-\Delta)^{\alpha} u +\nabla p = (b\cdot\nabla)b,\vspace{2mm} \qquad x\in \mathbb{R}^{3},\,t>0,\\
\partial_t b + (u\cdot\nabla)b +(-\Delta)^{\beta}b=(b\cdot\nabla)u,
\vspace{2mm}\\
\nabla\cdot u=\nabla\cdot b=0,
\end{array}\right.
\end{equation}
where $u$ represents the velocity, $p$ the pressure and $b$ the magnetic field.
The MHD equations govern the dynamics of electrically conducting fluids. When
the fractional powers $\alpha$ and $\beta$ satisfy
\begin{align*}
\alpha\geq
\frac{5}{4},\quad \beta\geq0,\quad\alpha+\beta\geq\frac{5}{2},
\end{align*}
any sufficiently smooth initial data leads to a unique global solution \cite{wujde3}. Highly nontrivial logarithmic improvements were obtained in  \cite{wujmfm,TYZ,Yamazaki2,Yama}.

\vskip .1in
The micropolar equations \eqref{3DMP} share some similarities with  the Boussinesq equations \eqref{3Bous} and the MHD equations \eqref{3DMHD}. They all contain the Navier-Stokes equations as the centerpiece and all current results
need $\alpha\ge \frac54$ as a condition for global regularity, but there are significant differences. There are difficulties that are unique for the  micropolar equations. The MHD equations involve nice canceling properties due to the presence of the four nonlinear terms, but the micropolar equations do not have
similar canceling properties. Due to the presence of the term $\nabla \nabla\cdot w$, we cannot directly estimate $\|w\|_{L^{q}}$ with $q>2$. In addition, for small $\beta>0$, it is really hard to derive any Sobolev estimates for $w$. Finally, some of the approaches for the 2D micropolar equations fail for the 3D micropolar equations due to the presence of the vortex stretching term.  This paper overcomes
these difficulties and establish three main global regularity results for the
 micropolar equations \eqref{3DMP}. Our first result can be stated as follows.

 \begin{thm}\label{Th1}
 	Assume $(u_0,w_0)  \in H^s(\mathbb{R}^3)$ with $s>\frac{5}{2}$ and $\nabla\cdot u_0=0$.
 	If $\alpha$ and $\beta$ satisfy
 	\begin{align*} \alpha\geq
 	\frac{5}{4},\quad \beta\geq0,\quad\alpha+\beta\geq\frac{7}{4},
 	\end{align*}
 	then \eqref{3DMP} admits a unique global regular
 	solution $(u,w)$ satisfying for any given $T>0$,
 		$$ u\in  L^{\infty}([0,T];H^s(\mathbb{R}^{3}))\cap
 		L^2([0,T];H^{s+\alpha}(\mathbb{R}^{3})),$$
		and
 		$$ w\in  L^{\infty}([0,T];H^s(\mathbb{R}^{3}))\cap L^2([0,T];H^{s+\beta}(\mathbb{R}^{3})).$$
 \end{thm}
\begin{rem}\rm
As a special consequence, Theorem \ref{Th1} states that one-derivative dissipation in the equation of the  microrotation is sufficient for global regularity if the velocity equation has $(-\Delta)^{\frac54} u$. As a comparison with the MHD equations, the micropolar equations require less dissipation (by the order of $\frac34$ derivatives) to achieve the global regularity.\vskip .1in
\end{rem}

\begin{rem}\rm
It appears that the condition on $\beta$ with $\alpha \ge \frac54$ and
$\alpha + \beta \ge \frac{7}{4}$ is sharp. This means that when $\alpha =\frac54$, we need $\beta\ge \frac12$, or at least one derivative dissipation in the equation of $w$. The sharpness is implied by the technical estimates. As we know, in order to obtain the global regularity, one needs to control the time integral of the velocity Lipschitz norm, namely
$$
\int_0^T \|\nabla u\|_{L^\infty(\mathbb R^3)}\,dt.
$$
In the three-dimensional space, more or less, one needs to bound
$$
\int_0^T \|\Lambda^{\frac52} u\|_{L^2(\mathbb R^3)} dt.
$$
Due to the presence of $\na\times w$ in the velocity equation, this requires the estimate
on the time integral of $\|\nabla w\|_{L^2}$. We need at least one-derivative dissipation in the $w$-equation in order to bound
$\|\nabla w\|_{L^2}$ suitably. More details can be found in the proof of Theorem \ref{Th1}.
\end{rem}

Our next goal is to consider the 3D micropolar equations with the logarithmical dissipations of the form
\begin{equation}\label{logorga11}
\left\{\begin{array}{l}
\partial_t u+(u\cdot\nabla) u
       +\mathcal{L}^{2} u +\nabla p =  \nabla \times w ,\vspace{2mm} \qquad x\in \mathbb{R}^{3},\,t>0,\\
\partial_tw + (u\cdot\nabla) w +2 w+(-\Delta)^{\beta}w
             = \nabla\times u+  \nabla \nabla\cdot w,
             \vspace{2mm}\\
\nabla\cdot u=0,\vspace{2mm}\\
u(x,0)=u_{0}(x),\quad w(x,0)=w_{0}(x),
\end{array}\right.
\end{equation}
where the operator $\mathcal{L}$ is defined by
$$\widehat{\mathcal{L}u}(\xi)=\frac{|\xi|^{\alpha}}{g(\xi)}\widehat{u}(\xi)$$
for some non-decreasing radial function $g(\tau)\geq 1$ defined on $\tau\geq 0$.

\vskip .1in
Theorem \ref{Th1} for the borderline case $\beta=0$ can be improved logarithmically.  More precisely, the following improved global regularity result for $\beta=0$ holds. For the sake of simplicity, we may assume $\nu =\kappa=\frac{1}{2}$ and $\mu=\gamma=1$. The sizes of these coefficients do not play a role in our results.

\begin{thm}\label{addTh2}
Consider \eqref{logorga11} with zero angular viscosity, namely,
\begin{equation}\label{log3DMP}
\left\{\begin{array}{l}
\partial_t u+(u\cdot\nabla) u
       +\mathcal{L}^{2} u +\nabla p =  \nabla \times w ,\vspace{2mm} \qquad x\in \mathbb{R}^{3},\,t>0,\\
\partial_tw + (u\cdot\nabla)  w +2 w
             = \nabla\times u+  \nabla \nabla\cdot w,
             \vspace{2mm}\\
\nabla\cdot u=0,\vspace{2mm}\\
u(x,0)=u_{0}(x),\quad w(x,0)=w_{0}(x).
\end{array}\right.
\end{equation}
Let $(u_{0},\,w_{0})\in H^{s}(\mathbb{R}^{3})$ with $s>\frac{5}{2}$ and $\nabla\cdot u_{0}=0$. If $\alpha\geq\frac{7}{4}$ and $g$ further satisfies
\begin{equation}\label{logcobd}
\int_{e}^{\infty}\frac{d\tau}{\tau\sqrt{\ln \tau} g^{2}(\tau)}=\infty,\end{equation}
then \eqref{log3DMP} admits a unique global solution $(u,\,w)$ such that, for any given $T>0$,
$$(u,\,w)\in L^{\infty}([0, T]; H^{s}(\mathbb{R}^{3})),\quad \mathcal{L} u\in L^{2}([0, T]; H^{s}(\mathbb{R}^{3})).
$$
\end{thm}

\begin{rem}\rm
We remark that the typical examples satisfying the condition (\ref{logcobd}) are
\begin{equation}
\begin{split}
& g(\xi)=\big[\ln(e+|\xi|)\big]^{\frac{1}{4}}; \\
& g(\xi)=\big[\ln(e+|\xi|)\big]^{\frac{1}{4}}\big[\ln (e+\ln(e+|\xi|)) \big]^{\frac{1}{2}};\\
& g(\xi)=\big[\ln(e+|\xi|)\big]^{\frac{1}{4}}\big[\ln (e+\ln(e+|\xi|)) \ln(e+\ln(e+\ln(e+|\xi|)))\big]^{\frac{1}{2}}.\nonumber
\end{split}
\end{equation}
\end{rem}

\vskip .2in
For the case $\alpha\ge \frac54$ and $\beta>0$, Theorem \ref{Th1} can also be improved logarithmically.

\begin{thm}\label{addTh3}
Consider \eqref{logorga11} with angular viscosity, namely,
\begin{equation}\label{loggnDMP}
\left\{\begin{array}{l}
\partial_t u+(u\cdot\nabla) u
       +\mathcal{L}^{2} u +\nabla p =  \nabla \times w ,\vspace{2mm} \qquad x\in \mathbb{R}^{3},\,t>0,\\
\partial_tw + (u\cdot\nabla) w +2 w+(-\Delta)^{\beta}w
             = \nabla\times u+  \nabla \nabla\cdot w,
             \vspace{2mm}\\
\nabla\cdot u=0,\vspace{2mm}\\
u(x,0)=u_{0}(x),\quad w(x,0)=w_{0}(x).
\end{array}\right.
\end{equation}
Let $(u_{0},\,w_{0})\in H^{s}(\mathbb{R}^{3})$ with $s>\frac{5}{2}$ and $\nabla\cdot u_{0}=0$. If $\alpha$ and $\beta$ satisfy
\begin{align*}
\alpha\geq
\frac{5}{4},\quad \beta>0,\quad\alpha+\beta\geq \frac{7}{4},
\end{align*}
and $g$ further satisfies
\begin{equation}\label{lognver56}
\int_{e}^{\infty}\frac{d\tau}{\tau g^{4}(\tau)}=\infty,\end{equation}
then \eqref{loggnDMP} admits a unique global solution $(u,\,w)$ such that, for any given $T>0$,
$$(u,\,w)\in L^{\infty}([0, T]; H^{s}(\mathbb{R}^{3})),\quad (\mathcal{L} u,\, \Lambda^{\beta}w)\in L^{2}([0, T]; H^{s}(\mathbb{R}^{3})).
$$
\end{thm}

\vskip .1in
\begin{rem}\rm
The typical examples of $g$ satisfying \eqref{lognver56} are as follows
\begin{equation}
\begin{split}
& g(\xi)=\big[\ln(e+|\xi|)\big]^{\frac{1}{4}}; \\
& g(\xi)=\big[\ln(e+|\xi|) \ln (e+\ln(e+|\xi|)) \big]^{\frac{1}{4}};\\
& g(\xi)=\big[\ln(e+|\xi|)\ln (e+\ln(e+|\xi|)) \ln(e+\ln(e+\ln(e+|\xi|)))\big]^{\frac{1}{4}}.\nonumber
\end{split}
\end{equation}
\end{rem}

\begin{rem}\rm
	Theorem \ref{Th1}, Theorem \ref{addTh2} and  Theorem \ref{addTh3} are also true for other fluid equations with similar structure in the general $n$-dimension case when the powers $\alpha$ and $\beta$ satisfy
	$$\alpha\geq
	\frac{1}{2}+\frac{n}{4},\quad \beta\geq0,\quad\alpha+\beta\geq 1+\frac{n}{4}.$$
\end{rem}

\vskip .1in
We now briefly summarize the main challenges and outline the main idea in the proof of our results.
Since the existence and uniqueness of local smooth solutions in the functional setting $H^{s}(\mathbb{R}^{3})$ with $s>\frac{5}{2}$ can be derived via a standard approach, our efforts are devoted to obtaining global {\it a
priori} bounds for $u$ and $w$ in the Sobolev space $H^s$ with $s>\frac52$. For the sake of completeness, the local well-posedness part is presented in Appendix \ref{apset3}. The proof of Theorem \ref{Th1} starts with the basic $L^2$-estimate
\begin{eqnarray}\label{xcxfsd4}
\|u(t)\|^2_{L^2}+\|w(t)\|_{L^2}^2
+ \int_0^t(\|\Lambda^\alpha u(\tau)\|_{L^2}^2 +\|\Lambda^\beta w(\tau)\|_{L^2}^2) d\tau\leq
C(t,\,u_{0},\,w_{0})
\end{eqnarray}
and a regularity bound for $u$ based on $\alpha\geq
\frac{5}{4}$ and \eqref{xcxfsd4}
\begin{eqnarray}\label{kcvre}
\|\Lambda^{\alpha+\beta-1}u(t)\|_{L^{2}}^{2}+
\int_{0}^{t}{\|\Lambda^{2\alpha+\beta-1}u(\tau)\|_{L^{2}}^{2}\,d\tau}<\infty.
\end{eqnarray}
Unfortunately, \eqref{kcvre} is insufficient to complete the proof of Theorem \ref{Th1}. When $\beta<1$, it is not possible to directly improve (\ref{kcvre})
to the critical level
\begin{eqnarray}\label{kn1}
\int_{0}^{t}{\|\Lambda^{\frac{5}{2}}u(\tau)
	\|_{L^{2}}^{2}\,d\tau}<\infty
\end{eqnarray}
due to the presence of the term  $\nabla \times w$. In order to overcome this difficulty, we combine the estimates of $u$ and $w$ to derive the regularity estimate of $w$,
\begin{eqnarray}\label{kn2}
\|\Lambda^{\varrho}w(t)\|_{L^{2}}^{2}+ \int_{0}^{t}{\|\Lambda^{\varrho+\beta}w(\tau)\|_{L^{2}}^{2}\,d\tau}<\infty
\end{eqnarray}
for any $\varrho<1+\beta$ with $\beta>0$. The regularity bound for $w$ in (\ref{kn2}) is sufficient for (\ref{kn1}). In order to obtain \eqref{kn2}, we need a suitable bound on the following commutator
$$\int_{\mathbb{R}^{3}}
[\Lambda^{\rho}, u \cdot
\nabla] w\Lambda^{\rho}w\,dx.$$
This is exactly where the condition $\alpha+\beta\geq \frac{7}{4}$ is needed (see \eqref{sdsfew09} for details). To prepare for the proof of the global $H^s$-bound for $(u, w)$, we need a global bound on $\|w(t)\|_{L^q}$ with $2\le q\le \infty$,
which is not a consequence  of (\ref{kn2}) in the 3D case. Due to the presence of
$\nabla \nabla\cdot w$, it is not possible to obtain this bound from the
equation of $w$ directly. Instead we circumvent this difficulty by estimating
the norm $\|\Lambda^{\sigma}w(t)\|_{L^{2}}$. After working out the estimates on
several commutators, we find
that, for any $\sigma\leq \frac{3}{2}$,
\begin{eqnarray}\label{kn3}
\|\Lambda^{\sigma}w(t)\|_{L^{2}}^{2}+\int_{0}^{t}{\|\Lambda^{\sigma+\beta}w(\tau)\|_{L^{2}}^{2}\,d\tau}<\infty.
\end{eqnarray}
Once \eqref{kn1} and \eqref{kn3} are at our disposal, the global
regularity of $u$ and $w$ can then be obtained.

We now explain the main idea in the proof of Theorem \ref{addTh2}. The proof is different from that for Theorem \ref{Th1}. One reason is that there is no dissipation in the $w$-equation. Another reason is that the dissipation is given by a
general Fourier multiplier operator here. The key idea here is still to successively
establish more and more regular bounds. The starting step is the basic energy estimate
\begin{eqnarray}
\|u(t)\|^2_{L^2}+\|w(t)\|_{L^2}^2
   + \int_{0}^{t}{(\|\Lambda^{r}u(\tau)\|_{L^{2}}^{2}+
\|\mathcal{L}u(\tau)\|_{L^{2}}^{2})\,d\tau} \leq
C(t,\,u_{0},\,w_{0})\nonumber
\end{eqnarray}
for any $r\in [0,\,\frac{7}{4})$. This bound allows us to establish an estimate
of the form
\begin{eqnarray}
\Big\|\frac{\Lambda^{\frac{3}{4}}}{g(\Lambda)}u(t)\Big\|_{L^{2}}^{2} +\int_{0}^{t}{
\Big\|\frac{\Lambda^{\frac{5}{2}}}{g^{2}(\Lambda)}u(\tau)\Big\|_{L^{2}}^{2}\,d\tau}\leq C(t,\,u_{0},\,w_{0}),\nonumber
\end{eqnarray}
where $\frac{\Lambda^{\gamma}}{g^{k}(\Lambda)}$ denotes the Fourier multiplier operator given by the symbol $\frac{|\xi|^{\gamma}}{g^{k}(\xi)}$ for $\gamma>0,k>0$.
The next step aims at the estimate, for any $\widetilde{\sigma}\in (0,\frac{5}{2})$ and $k_{1}<\widetilde{\sigma}+\frac{3}{4}$,
\begin{eqnarray}
\|\Lambda^{k_{1}}u(t)\|_{L^{2}}^{2}
+\|\Lambda^{\widetilde{\sigma}}w(t)\|_{L^{2}}^{2}
+\int_{0}^{t}{
 \|\Lambda^{k_{1}+\frac{7}{4}-\epsilon_{2}}u
(\tau)\|_{L^{2}}^{2} \,d\tau}\leq
C(t,\,u_{0},\,w_{0}).\nonumber
\end{eqnarray}
The proof of this estimate is highly technical. We use the high-low frequency splitting techniques and several logarithmic interpolation inequalities. Once this is obtained, we are then able to bound the global $H^s$-norm of $(u, w)$. The proof of Theorem \ref{addTh3} combines the techniques from the proofs of Theorem \ref{Th1} and of Theorem \ref{addTh2}.

The rest of this paper is organized as follows.  Section \ref{sec2} is devoted to the proof of Theorem \ref{Th1}. Section \ref{sec3} proves Theorem \ref{addTh2} while
Section \ref{sec4} proves Theorem \ref{addTh3}. The Appendix \ref{apset1} provides the definitions of the Littlewood-Paley decomposition and Besov spaces,  and related tools. The Appendix \ref{apset2} provides a global regularity result for a micropolar system
without $\na\na\cdot w$ in the equation of $w$, namely (\ref{3DMP}) with $\mu=0$. The purpose of this appendix is to show that, without this bad term, the requirement for $\alpha$ and $\beta$ can be reduced to $\alpha\ge \frac54$ and $\beta=0$. The Appendix \ref{apset3} provides the local well-posedness result.

\vskip .3in
\section{The proof of Theorem \ref{Th1}}
\label{sec2}
\setcounter{equation}{0}

This section is devoted to the proof of Theorem \ref{Th1}. Generic constant will be denoted by $C$. Sometimes we write  $C(\gamma_{1},\gamma_{2},\cdot\cdot\cdot,\gamma_{k})$ to emphasize the dependence of $C$ on the quantities $\gamma_{1},\gamma_{2},\cdot\cdot\cdot,\gamma_{k}$. We write $A\thickapprox B$ if there exist two positive constants $C_{1}\leq C_{2}$ such that $C_{1}B\leq A\leq C_{2}B$.

\vskip .1in
The proof focuses on the case $\alpha+\beta = \frac{7}{4}$ as the case $\alpha+\beta>\frac{7}{4}$ is simpler and can
be done with a slight modification. The global {\it a priori} bounds start with the
following basic energy estimate.

\begin{lemma}
\label{L301}
Assume $(u_{0},w_{0})$ satisfies the assumptions stated in Theorem \ref{Th1}.
If $\alpha+\beta>1$, then the corresponding solution $(u, w)$
of \eqref{3DMP} admits the following bound for any $t>0$
\begin{eqnarray}\label{t301}
\|u(t)\|^2_{L^2}+\|w(t)\|_{L^2}^2
   + \int_0^t(\|\Lambda^\alpha u(\tau)\|_{L^2}^2 +\|\Lambda^\beta w(\tau)\|_{L^2}^2) d\tau\leq
C(t,\,u_{0},\,w_{0}).
\end{eqnarray}
\end{lemma}

\begin{proof}
Multiplying the first two equations in \eqref{3DMP} by $(u, w)$ and integrating by parts, we have
\begin{align*}
 &  \frac{1}{2}\frac{d}{dt}
  ( \|u(t)\|^2_{L^2}+\|w(t)\|_{L^2}^2)
  +  \| \Lambda^\alpha u\|_{L^2}^2+  \|\Lambda^\beta w\|_{L^2}^2+2\|w\|^2_{L^2}+\|\nabla\cdot w\|^2_{L^2}\nonumber\\
  &=\int_{\mathbb{R}^{3}} \Big( (\nabla \times
          w)\cdot u+
      (\nabla \times u)\cdot w  \Big) dx\\
   &\leq 2 \|\Lambda^{\kappa}u\|_{L^2}\|\Lambda^{1-\kappa}w\|_{L^2}
 \\
   &\leq 2 (\|u\|_{L^2}^{1-\frac{\kappa}{\alpha}}\|\Lambda^{\alpha}u\|_{L^2}
 ^{\frac{\kappa}{\alpha}})(\|w\|_{L^2}^{1-\frac{1-\kappa}{\beta}}
 \|\Lambda^{\beta}w\|_{L^2}
 ^{\frac{1-\kappa}{\beta}})
 \\
 &\leq  \frac{1}{2}\|\Lambda^\alpha u \|_{L^2}^2+\frac{1}{2}\|\Lambda^{\beta}w\|_{L^2}
 ^{2}
 +C(\|u\|^2_{L^2}+\|  w \|^2_{L^2}),
\end{align*}
where $1-\beta<\kappa< \alpha$ (such $\kappa$ is possible due to the condition $\alpha+\beta>1$). Here we have also used the equality, due to $\nabla\cdot u=0$,
$$
\int_{\mathbb{R}^{3}} ( u\cdot\nabla ) u \cdot u \ dx=0 \ \ \  \mbox{and} \
\  \int_{\mathbb{R}^{3}}(
u\cdot \nabla)   w\cdot w \ dx=0,
$$
and the following interpolation inequality
$$\|f\|_{\dot{H}^{s_{0}}}\leq C\|f\|_{\dot{H}^{s_{1}}}^{1-\theta}\|f\|_{\dot{H}^{s_{2}}}^{\theta},$$
where $s_{0}=(1-\theta)s_{1}+\theta s_{2}$ with $0\leq\theta\leq1$. By the Gronwall inequality, one has
\begin{align*}
  \|u(t)\|^2_{L^2}+\|w(t)\|_{L^2}^2
   + \int_0^t(\|\Lambda^\alpha u(\tau)\|_{L^2}^2 +\|\Lambda^\beta w(\tau)\|_{L^2}^2) d\tau \leq C(t,\,u_{0},\,w_{0}).
\end{align*}
This proves the lemma.
\end{proof}

\vskip .1in
Next we proceed with the following estimate in the Sobolev spaces.

\begin{lemma}\label{addt01}
Assume $(u_{0},w_{0})$ satisfies the assumptions stated in Theorem \ref{Th1}.
If $\alpha\geq\frac{5}{4}$ and $\beta\geq0$, then the corresponding solution $(u, w)$ of \eqref{3DMP} admits the following bound for any $t>0$
\begin{eqnarray}\label{addtd201}
\|\Lambda^{\alpha+\beta-1}u(t)\|_{L^{2}}^{2}
+\int_{0}^{t}{\|\Lambda^{2\alpha+\beta-1}u(\tau)\|_{L^{2}}^{2}\,d\tau}\leq
C(t,\,u_{0},\,w_{0}).
\end{eqnarray}
\end{lemma}
\begin{proof}
Applying $\Lambda^{\alpha+\beta-1}$ to $\eqref{3DMP}_{1}$
and taking the $L^{2}$ inner product with $\Lambda^{\alpha+\beta-1}u$, we obtain
\begin{align}
 &\frac{1}{2}\frac{d}{dt}\|\Lambda^{\alpha+\beta-1}u(t)\|_{L^{2}}^{2}
+\|\Lambda^{2\alpha+\beta-1}
u\|_{L^{2}}^{2}\nonumber\\
&=  -\int_{\mathbb{R}^{3}}{[\Lambda^{\alpha+\beta-1},
u\cdot\nabla]u\cdot \Lambda^{\alpha+\beta-1}u\,dx}+\int_{\mathbb{R}^{3}}
{\Lambda^{\alpha+\beta-1}(\nabla\times w)
\Lambda^{\alpha+\beta-1}u\,dx}\nonumber\\
&:=  M_{1}+M_{2},\nonumber
\end{align}
where $[\Lambda^{s}, f]g:=\Lambda^{s}(fg)-f(\Lambda^{s}g)$ denotes the standard commutator. We recall the Kato-Ponce inequality (see, e.g.,  \cite{KPPonce})
\begin{eqnarray*}
 \|[\Lambda^s,f ]g \|_{L^r}
  \leq C\left(\|\nabla f \|_{L^{p_1}}\|\Lambda^{s-1} g \|_{L^{q_1}}
        +     \|\Lambda^{s} f \|_{L^{p_2}} \|  g \|_{L^{q_2}}\right),
\end{eqnarray*}
where $s>0$, $1<r<\infty$ and $\frac{1}{r}=\frac{1}{p_1}+\frac{1}{q_1}=\frac{1}{p_2}+\frac{1}{q_2}$ with $q_1, p_2\in(1,\infty)$ and $p_1, q_2\in[1,\infty]$.
We will make use of the following variant of the Kato-Ponce inequality,
$$\|[\Lambda^{s-1}\partial_{x_{i}},f ]g \|_{L^r}
  \leq C\left(\|\nabla f \|_{L^{p_1}}\|\Lambda^{s-1} g \|_{L^{q_1}}
        +     \|\Lambda^{s} f \|_{L^{p_2}} \|  g \|_{L^{q_2}}\right),\quad i=1,\,2,\,3.$$
As a consequence, we have
\begin{align}
M_{1}
 &\leq  C\|[\Lambda^{\alpha+\beta-1},
u\cdot\nabla]u\|_{L^{2}}\|\Lambda^{\alpha+\beta-1}u\|_{L^{2}}
 \nonumber\\
 &\leq C
 \Big(\|\nabla u\|_{L^{p_{0}}}\|\Lambda^{\alpha+\beta-2}
 \nabla u\|_{L^{\frac{2p_{0}}{p_{0}-2}}}+\|\nabla u\|_{L^{p_{0}}}
 \|\Lambda^{\alpha+\beta-1}
 u\|_{L^{\frac{2p_{0}}{p_{0}-2}}}\Big)\|\Lambda^{\alpha+\beta-1}u\|_{L^{2}}
 \nonumber\\
 &\leq C \|\nabla u\|_{L^{p_{0}}}
 \|\Lambda^{\alpha+\beta-1}
 u\|_{L^{\frac{2p_{0}}{p_{0}-2}}}\|\Lambda^{\alpha+\beta-1}u\|_{L^{2}}
  \nonumber\\
 &\leq C \|u\|_{H^{\alpha}}
 \|\Lambda^{\alpha+\beta-1}u\|_{L^{2}}^{1-\frac{3}{\alpha p_{0}}}\|\Lambda^{2\alpha+\beta-1}
 u\|_{L^{2}}^{\frac{3}{\alpha p_{0}}}
 \|\Lambda^{\alpha+\beta-1}u\|_{L^{2}}
\nonumber\\
 &\leq
 \frac{1}{4}\|\Lambda^{2\alpha+\beta-1}u\|_{L^{2}}^{2}+
 C(1+\|u\|_{H^{\alpha}}^{2})
 \|\Lambda^{\alpha+\beta-1}u\|_{L^{2}}^{2},\nonumber
\end{align}
where, due to  $\alpha\geq
\frac{5}{4}$, we have selected $p_{0}>2$ satisfying
$$\frac{5-2\alpha}{6}\leq\frac{1}{p_{0}}\leq \frac{\alpha}{3}.$$
The term $M_{2}$ can be bounded by
\begin{eqnarray}
M_{2}\leq C\|\Lambda^{\beta}w\|_{L^{2}}\|\Lambda^{2\alpha+\beta-1}u\|_{L^{2}}
 \leq \frac{1}{4}\|\Lambda^{2\alpha+\beta-1}u\|_{L^{2}}^{2}
 +C\|\Lambda^{\beta}w\|_{L^{2}}^{2}.\nonumber
\end{eqnarray}
Collecting the estimates above yields
$$\frac{d}{dt}\|\Lambda^{\alpha+\beta-1}u(t)\|_{L^{2}}^{2}
+\|\Lambda^{2\alpha+\beta-1}
u\|_{L^{2}}^{2}\leq C(1+\|u\|_{H^{\alpha}}^{2})
 \|\Lambda^{\alpha+\beta-1}u\|_{L^{2}}^{2}+C\|w\|_{H^{\beta}}^{2}.$$
Thanks to the estimate \eqref{t301}, we have
\begin{eqnarray}\label{sdsfew16}
 \int_0^t(\| u(\tau)\|_{H^{\alpha}}^2+\|w(\tau)\|_{H^{\beta}}^{2})d\tau\leq
C(t,\,u_{0},\,w_{0}).
\end{eqnarray}
By the Gronwall inequality,
$$\|\Lambda^{\alpha+\beta-1}u(t)\|_{L^{2}}^{2}
+\int_{0}^{t}{\|\Lambda^{2\alpha+\beta-1}u(\tau)\|_{L^{2}}^{2}\,d\tau}\leq
C(t,\,u_{0},\,w_{0}).$$
Therefore, we complete the proof of Lemma \ref{addt01}.
\end{proof}

We are now in the position to derive the following key estimates.
\begin{lemma}\label{AZL302}
Assume $(u_{0},w_{0})$ satisfies the assumptions stated in Theorem \ref{Th1}.
If $\alpha\geq\frac{5}{4}$, $\beta>0$ and $\alpha+\beta\geq\frac{7}{4}$, then the corresponding solution $(u, w)$
of  \eqref{3DMP} admits the following bounds for any $t>0$,
\begin{gather}
\|\Lambda^{\frac{5}{4}}u(t)\|_{L^{2}}^{2}
+\int_{0}^{t}{\|\Lambda^{\alpha+\frac{5}{4}}u(\tau)\|_{L^{2}}^{2}\,d\tau}\leq
C(t,\,u_{0},\,w_{0}),\label{asdAZf}\\
\|\Lambda^{\rho}w(t)\|_{L^{2}}^{2}
+\int_{0}^{t}{\|\Lambda^{\rho+\beta}w(\tau)\|_{L^{2}}^{2}\,d\tau}\leq
C(t,\,u_{0},\,w_{0}),\label{dsfsdae}
\end{gather}
where $\rho>0$ satisfies
 \begin{eqnarray}\label{sgy78fv3s}
 \frac{9}{4}-(\alpha+\beta)< \rho<1+\beta.
  \end{eqnarray}
In particular, due to $\alpha\geq
\frac{5}{4}$, it follows from \eqref{asdAZf} that
\begin{eqnarray}\label{AZ001}
\|\Lambda^{\frac{5}{4}}u(t)\|_{L^{2}}^{2}
+\int_{0}^{t}{\|\Lambda^{\frac{5}{2}}u(\tau)\|_{L^{2}}^{2}\,d\tau}\leq
C(t,\,u_{0},\,w_{0}).
\end{eqnarray}
\end{lemma}

\begin{rem}\rm
	Combining with  \eqref{t301}, we see that \eqref{dsfsdae} is true  for any
	$0\leq \rho<1+\beta$.  This fact will be used repeatedly later.
\end{rem}

\begin{proof}
Applying $\Lambda^{\frac{5}{4}}$ to $\eqref{3DMP}_{1}$
and taking the $L^{2}$ inner product with $\Lambda^{\frac{5}{4}}u$ yield
\begin{align}\label{sdsfew01}
\frac{1}{2}\frac{d}{dt}\|\Lambda^{\frac{5}{4}}u(t)\|_{L^{2}}^{2}
+\|\Lambda^{\alpha+\frac{5}{4}}
u\|_{L^{2}}^{2}
&=  -\int_{\mathbb{R}^{3}}{[\Lambda^{\frac{5}{4}},
u\cdot\nabla]u\cdot \Lambda^{\frac{5}{4}}u\,dx}
+\int_{\mathbb{R}^{3}}{\Lambda^{\frac{5}{4}}(\nabla\times w)
\Lambda^{\frac{5}{4}}u\,dx}\nonumber\\
&:=  N_{1}+N_{2}.
\end{align}
Applying $\Lambda^{\rho}$ to $\eqref{3DMP}_{2}$ and taking the
inner product with $\Lambda^{\rho}w$, we obtain
\begin{align}\label{sdsfew02}&
\frac{1}{2}\frac{d}{dt}\|\Lambda^{\rho}w(t)\|_{L^{2}}^{2}
+\|\Lambda^{\rho+\beta}w \|_{L^{2}}^{2}+2\|\Lambda^{\rho}w\|^2_{L^2}+\|\Lambda^{\rho}\nabla\cdot w\|_{L^{2}}^{2}\nonumber\\&= \int_{\mathbb{R}^{3}}
\Lambda^{\rho}\big(\nabla\times u\big)\Lambda^{\rho}w\,dx -\int_{\mathbb{R}^{3}}
[\Lambda^{\rho}, u \cdot
\nabla] w\Lambda^{\rho}w\,dx\nonumber\\
&:=  N_{3}+N_{4}.
\end{align}
Combing \eqref{sdsfew01} and \eqref{sdsfew02} leads to
 \begin{eqnarray}\label{sdsfew03}
 \frac{1}{2}\frac{d}{dt}(\|\Lambda^{\frac{5}{4}}u(t)\|_{L^{2}}^{2}
+\|\Lambda^{\rho}w(t)\|_{L^{2}}^{2})
+\|\Lambda^{\alpha+\frac{5}{4}}
u\|_{L^{2}}^{2}+\|\Lambda^{\rho+\beta}w \|_{L^{2}}^{2}
\leq \sum_{k=1}^{4}N_{k}.
 \end{eqnarray}
Thanks to the Sobolev embedding inequality and the Kato-Ponce inequality, we obtain
\begin{align}\label{sdsfew04}
N_{1}
 &\leq  \|[\Lambda^{\frac{5}{4}},
u\cdot\nabla]u\|_{L^{2}}\|\Lambda^{\frac{5}{4}}u\|_{L^{2}}
 \nonumber\\
 &\leq C
 \Big(\|\nabla u\|_{L^{\frac{12}{5}}}\|\Lambda^{\frac{1}{4}}
 \nabla u\|_{L^{12}}+\|\nabla u\|_{L^{\frac{12}{5}}}
 \|\Lambda^{\frac{5}{4}}
 u\|_{L^{12}}\Big)\|\Lambda^{\frac{5}{4}}u\|_{L^{2}}
 \nonumber\\
 &\leq C \|\nabla u\|_{L^{\frac{12}{5}}}
 \|\Lambda^{\frac{5}{4}}
 u\|_{L^{12}}\|\Lambda^{\frac{5}{4}}u\|_{L^{2}}
  \nonumber\\
 &\leq C \|\Lambda^{\frac{5}{4}}u\|_{L^{2}}
 \|\Lambda^{\frac{5}{2}}u\|_{L^{2}}
 \|\Lambda^{\frac{5}{4}}u\|_{L^{2}}
  \nonumber\\
 &\leq C \|u\|_{H^{\alpha}}
 \|u\|_{H^{\alpha+\frac{5}{4}}}
 \|\Lambda^{\frac{5}{4}}u\|_{L^{2}}
\nonumber\\
 &\leq
 \frac{1}{16}\|\Lambda^{\alpha+\frac{5}{4}}u\|_{L^{2}}^{2}+
 C(1+\|u\|_{H^{\alpha}}^{2})(1+
 \|\Lambda^{\frac{5}{4}}u\|_{L^{2}}^{2}).
\end{align}
The equivalence of the norms in the following three spaces has been used here and will be used frequently,
$$\|f\|_{B_{2,2}^{s}}\approx \|f\|_{H^{s}}\approx \|f\|_{L^{2}}+\|\Lambda^{s}f\|_{L^{2}},\ \ \forall\,s\geq0,$$
where $B_{p,q}^{s}$ denotes the nonhomogeneous Besov space (see Appendix for its definition).
The next term $N_{2}$ can be bounded by
\begin{align}\label{sdsfew05}
N_{2}
 &\leq  \|w\|_{H^{\frac{9}{4}-\alpha}}\|\Lambda^{\alpha+\frac{5}{4}}u\|_{L^{2}}\nonumber\\
  &\leq  C\|w \|_{H^{\rho+\beta}} \|\Lambda^{\alpha+\frac{5}{4}}u\|_{L^{2}}
  \nonumber\\
  &\leq \frac{1}{16}\|\Lambda^{\alpha+\frac{5}{4}}u\|_{L^{2}}^{2}
  +\frac{1}{16}\|\Lambda^{\rho+\beta}w \|_{L^{2}}^{2} +C\|w\|_{L^{2}}^{2},
\end{align}
where $\rho$ satisfies
\begin{eqnarray}\label{sdsfew06}
\rho> \frac{9}{4}-(\alpha+\beta).
\end{eqnarray}
Similarly,
\begin{align}\label{sdsfew07}
N_{3}
 &\leq \|\Lambda^{\rho-\beta}\big(\nabla\times u\big)\|_{L^{2}}\|\Lambda^{\rho+\beta}w\|_{L^{2}}\nonumber\\
 &\leq C\|u\|_{H^{\alpha+\frac{5}{4}}}
 \|\Lambda^{\rho+\beta}w\|_{L^{2}}
 \nonumber\\
  &\leq  \frac{1}{16}\|\Lambda^{\alpha+\frac{5}{4}}u\|_{L^{2}}^{2}+
  \frac{1}{16} \|\Lambda^{\rho+\beta}w\|_{L^{2}}^{2}
  +C\|u\|_{L^{2}}^{2},
\end{align}
where $\rho$ satisfies
\begin{eqnarray}\label{sdsfew08}
\rho<\frac{1}{4}+\alpha+\beta.
\end{eqnarray}
To deal with $N_{4}$, we need the commutator estimate (see, e.g., \cite[Lemma 2.6]{Yeacap18})
\begin{align}\label{qqqqqqtNCE12}
\|[\Lambda^{\delta},f\cdot\nabla]g\|_{B_{p,r}^{s}}\leq
C\big(\|\nabla f\|_{L^{p_{1}}}\|g\|_{
{B}_{p_{2},r}^{s+\delta}}+\|f\|_{L^{2}}\|g\|_{L^{2}}\big),
\end{align}
where $\nabla\cdot f=0$  and $\frac{1}{p}=\frac{1}{p_{1}}+\frac{1}{p_{2}}$ with $p\in[2,
\infty)$, $p_{1},\,p_{2}\in[2, \infty]$, $r\in[1, \infty]$ and $s\in(-1, 1-\delta)$ for $\delta\in(0,2)$. By the Sobolev embedding inequality, the H$\rm\ddot{o}$lder inequality and \eqref{qqqqqqtNCE12}, one has
\begin{align}\label{sdsfew09}
N_{4}
\leq&  C\|[\Lambda^{\rho}, u\cdot\nabla]w\|_{B_{2,2}^{-\beta}}\|\Lambda^{\rho}w\|_{B_{2,2}^{\beta}}\nonumber\\
\leq & C(\|\nabla
u\|_{L^{m_{0}}}\|w\|_{B_{\frac{2m_{0}}{m_{0}-2},2}^{\rho-\beta}}
+\|u\|_{L^{2}}\|w\|_{L^{2}})
(\|\Lambda^{\rho}w\|_{L^{2}}+\|\Lambda^{\rho+\beta}w\|_{L^{2}}) \nonumber\\
\leq &
C(\|u\|_{H^{2\alpha+\beta-1}}
\|w\|_{B_{2,2}^{\rho-\beta+\frac{3}{m_{0}}}}+\|u\|_{L^{2}}\|w\|_{L^{2}})
(\|\Lambda^{\rho}w\|_{L^{2}}+\|\Lambda^{\rho+\beta}w\|_{L^{2}})
\nonumber\\
\leq&
C(\|u\|_{H^{2\alpha+\beta-1}}
\|w\|_{H^{\rho}}+\|u\|_{L^{2}}\|w\|_{L^{2}})
(\|\Lambda^{\rho}w\|_{L^{2}}+\|\Lambda^{\rho+\beta}w\|_{L^{2}})
\nonumber\\
\leq&
  \frac{1}{16}\|\Lambda^{\rho+\beta}w\|_{L^{2}}^{2}
  +C(1+\|u\|_{H^{2\alpha+\beta-1}}^{2})
  \|\Lambda^{\rho}w\|_{L^{2}}^{2}\nonumber\\& +C(1+\|u\|_{L^{2}}^{2})(1+
  \|w\|_{L^{2}}^{2}),
\end{align}
where $\rho$ and  $m_{0}>2$ satisfy
\begin{eqnarray}\label{aasdsfew10}
 \rho<1+\beta,\qquad \frac{7-4\alpha-2\beta}{6}\leq \frac{1}{m_{0}}\leq\frac{\beta}{3}.
\end{eqnarray}
Such $m_{0}$ exists due to $\alpha+\beta\geq \frac{7}{4}$. It is worth noting that this is the only place
where we use the assumption $\alpha+\beta\geq \frac{7}{4}$.
We also remark that in order to obtain \eqref{sdsfew09}, we need the restriction $\beta<1$. This is due to the use of \eqref{qqqqqqtNCE12}.
When $\beta\geq1$,
it is easy to show \eqref{AZ001} without exploiting the dedicate estimate (\ref{sdsfew02}). Therefore,  \eqref{AZ001} holds for both $\beta<1$ and $\beta \ge 1$.
Inserting the estimates \eqref{sdsfew04}, \eqref{sdsfew05}, \eqref{sdsfew07} and \eqref{sdsfew09} in \eqref{sdsfew03}, we conclude
 \begin{align*}
 &  \frac{d}{dt}(\|\Lambda^{\frac{5}{4}}u(t)\|_{L^{2}}^{2}
+\|\Lambda^{\rho}w(t)\|_{L^{2}}^{2})
+\|\Lambda^{\alpha+\frac{5}{4}}
u\|_{L^{2}}^{2}+\|\Lambda^{\rho+\beta}w \|_{L^{2}}^{2}\nonumber\\&\leq
C(1+\|u\|_{H^{\alpha}}^{2}+\|u\|_{H^{2\alpha+\beta-1}}^{2})
 (\|\Lambda^{\frac{5}{4}}u\|_{L^{2}}^{2}
 +\|\Lambda^{\rho}w\|_{L^{2}}^{2}) \nonumber\\& \quad+C(1+\|u\|_{L^{2}}^{2})(1+
  \|w\|_{L^{2}}^{2}),
 \end{align*}
where $\rho>0$ satisfies \eqref{sgy78fv3s}
by combining \eqref{sdsfew06}, \eqref{sdsfew08} and \eqref{aasdsfew10}.
By \eqref{addtd201}, \eqref{sdsfew16} and the Gronwall
inequality, we obtain
$$\|\Lambda^{\frac{5}{4}}u(t)\|_{L^{2}}^{2}+\|\Lambda^{\rho}w(t)\|_{L^{2}}^{2}
+\int_{0}^{t}{(\|\Lambda^{\alpha+\frac{5}{4}}u(\tau)\|_{L^{2}}^{2}+\|\Lambda^{\rho+\beta}w (\tau)\|_{L^{2}}^{2})\,d\tau}\leq
C(t,\,u_{0},\,w_{0}).$$
This completes the proof of Lemma
\ref{AZL302}.
\end{proof}

\vskip .1in
With the help of \eqref{AZ001}, we can now establish the higher regularity estimate for $w$, which can be stated as the following lemma.
\begin{lemma}\label{AZL3dsf8}
Assume $(u_{0},w_{0})$ satisfies the assumptions stated in Theorem \ref{Th1}.
If $\alpha\geq
\frac{5}{4}$, $\beta>0$ and $\alpha+\beta\geq \frac{7}{4}$, then the corresponding solution $(u, w)$
of  \eqref{3DMP} admits the following bound for any $t>0$ and for any $\sigma\leq \frac{3}{2}$,
\begin{eqnarray*}
\|\Lambda^{\sigma}w(t)\|_{L^{2}}^{2}
+\int_{0}^{t}{\|\Lambda^{\sigma+\beta}w(\tau)\|_{L^{2}}^{2}\,d\tau}\leq
C(t,\,u_{0},\,w_{0}).
\end{eqnarray*}
In particular, due to $\beta>0$, it holds
\begin{eqnarray}\label{sfdserndf}
\int_{0}^{t}{\|w(\tau)\|_{L^{\infty}}^{2}\,d\tau}\leq
C(t,\,u_{0},\,w_{0}).
\end{eqnarray}
\end{lemma}
\begin{proof}
Applying $\Lambda^{\sigma}$ to $ \eqref{3DMP}_{2}$ and taking the
inner product with $\Lambda^{\sigma}w$, one gets
\begin{align}\label{sdsfew20}&
\frac{1}{2}\frac{d}{dt}\|\Lambda^{\sigma}w(t)\|_{L^{2}}^{2}
+\|\Lambda^{\sigma+\beta}w \|_{L^{2}}^{2}+2\|\Lambda^{\sigma}w\|^2_{L^2}+\|\Lambda^{\sigma}\nabla\cdot w\|_{L^{2}}^{2}\nonumber\\&= \int_{\mathbb{R}^{3}}
\Lambda^{\sigma}\big(\nabla\times u\big)\Lambda^{\sigma}w\,dx -\int_{\mathbb{R}^{3}}
[\Lambda^{\sigma}, u \cdot
\nabla] w\Lambda^{\sigma}w\,dx\nonumber\\
&:= L_{1}+L_{2}.
\end{align}
We deduce from the Gagliardo-Nirenberg inequality that
\begin{align}\label{sdsfew21}
L_{1}&\leq C\|\Lambda^{\sigma+1-\beta}u\|_{L^{2}}\|\Lambda^{\sigma+\beta}w \|_{L^{2}}\nonumber\\
&\leq \frac{1}{16}\|\Lambda^{\sigma+\beta}w \|_{L^{2}}^{2}+C\|\Lambda^{\sigma+1-\beta}u\|_{L^{2}}^{2}
\nonumber\\
&\leq \frac{1}{16}\|\Lambda^{\sigma+\beta}w \|_{L^{2}}^{2}+C\|u\|_{L^{2}}^{2}+C\|\Lambda^{\frac{5}{2}}u\|_{L^{2}}^{2},
\end{align}
where $\sigma\leq \beta+\frac{3}{2}$. Making use of $\nabla\cdot u=0$ and the Kato-Ponce inequality, we achieve
\begin{align}\label{sdsfew22}
L_{2}&= -\int_{\mathbb{R}^{3}}
[\Lambda^{\sigma}\partial_{x_{i}}, u_{i}
] w\Lambda^{\sigma}w\,dx
\nonumber\\
&\leq \|[\Lambda^{\sigma}\partial_{x_{i}}, u_{i}
] w \|_{L^{2}}\|\Lambda^{\sigma}w\|_{L^{2}}
\nonumber\\
&\leq C(\|\nabla u \|_{L^{m_{1}}}\|\Lambda^{\sigma}w\|_{L^{\frac{2m_{1}}{m_{1}-2}}}+
\|\Lambda^{\sigma+1}u\|_{L^{2}}\|w\|_{L^{\infty}})\|\Lambda^{\sigma}w\|_{L^{2}}
\nonumber\\
&\leq C(\|u \|_{H^{\frac{5}{2}}}\|\Lambda^{\sigma}w\|_{L^{2}}^{1-\frac{3}{\beta m_{1}}}\|\Lambda^{\sigma+\beta}w\|_{L^{2}}
^{\frac{3}{\beta m_{1}}}+
\|u \|_{H^{\frac{5}{2}}}\| w \|_{H^{\sigma+\beta}})\|\Lambda^{\sigma}w\|_{L^{2}}
\nonumber\\
&\leq \frac{1}{16}\|\Lambda^{\sigma+\beta}w \|_{L^{2}}^{2}+C(1+\|u\|_{L^{2}}^{2}+\|\Lambda^{\frac{5}{2}}u\|_{L^{2}}^{2})
\|\Lambda^{\sigma}w\|_{L^{2}}^{2},
\end{align}
where $\max\{\frac{3}{\beta},\,2\}<m_{1}<\infty$ and $\frac{3}{2}-\beta<\sigma\leq \frac{3}{2}$.
Putting \eqref{sdsfew21}-\eqref{sdsfew22} into \eqref{sdsfew20} yields
\begin{align*}
 \frac{d}{dt}\|\Lambda^{\sigma}w(t)\|_{L^{2}}^{2}
+\|\Lambda^{\sigma+\beta}w \|_{L^{2}}^{2}\leq C(1+\|u\|_{L^{2}}^{2}+\|\Lambda^{\frac{5}{2}}u\|_{L^{2}}^{2})
(1+\|\Lambda^{\sigma}w\|_{L^{2}}^{2}).
\end{align*}
It follows from \eqref{AZ001} and the Gronwall inequality that
$$\|\Lambda^{\sigma}w(t)\|_{L^{2}}^{2}
+\int_{0}^{t}{\|\Lambda^{\sigma+\beta}w(\tau)\|_{L^{2}}^{2}\,d\tau}\leq
C(t,\,u_{0},\,w_{0}).$$
We thus complete the proof of Lemma \ref{AZL3dsf8}.
\end{proof}

\vskip .1in
Finally, with \eqref{AZ001} and \eqref{sfdserndf} at our disposal, we are ready to show the global $H^s$-bound.
\begin{proof}[{Proof of Theorem \ref{Th1}}]
Applying $\Lambda^{s}$ with $s>\frac{5}{2}$ to  \eqref{3DMP} and
taking the $L^{2}$ inner product with $(\Lambda^{s}u, \Lambda^{s} w)$, we have
\begin{align}\label{tsdfwef}&
\frac{1}{2}\frac{d}{dt}(\|\Lambda^{s} u\|_{L^{2}}^{2}+\|\Lambda^{s}
w\|_{L^{2}}^{2})+\|\Lambda^{s+\alpha}
  u\|_{L^{2}}^{2}+\|\Lambda^{s+\beta}  w\|_{L^{2}}^{2}+2\|\Lambda^{s}  w\|_{L^{2}}^{2}+\|\Lambda^{s}\nabla\cdot w\|_{L^{2}}^{2}\nonumber\\
&=\int_{\mathbb{R}^{3}}\Big(
\Lambda^{s}(\nabla\times u)\cdot\Lambda^{s} w+\Lambda^{s}
(\nabla\times w)\cdot \,\,\Lambda^{s} u\Big)\,dx
-\int_{\mathbb{R}^{3}}
[\Lambda^{s}, u \cdot \nabla] u \,\Lambda^{s} u\,dx\nonumber\\&\quad-\int_{\mathbb{R}^{3}}
[\Lambda^{s}, u \cdot \nabla]w\, \Lambda^{s} w\,dx
\nonumber\\&:= J_{1}+J_{2}+J_{3}.
\end{align}
Due to $\alpha\geq1$, it is easy to check that
\begin{align*}
J_{1}&\leq  2\|\Lambda^{s+1} u\|_{L^{2}}\|\Lambda^{s} w\|_{L^{2}}\nonumber\\
&\leq  2\|\Lambda^{s} u\|_{L^{2}}^{\frac{\alpha-1}{\alpha}}\|\Lambda^{s+\alpha} u\|_{L^{2}}^{\frac{1}{\alpha}}\|\Lambda^{s} w\|_{L^{2}}
\nonumber\\
&\leq
\frac{1}{16}\|\Lambda^{s+\alpha}
  u\|_{L^{2}}^{2}+C\|\Lambda^{s}  w\|_{L^{2}}^{2}+C\|\Lambda^{s} u\|_{L^{2}}^{2}.
\end{align*}
By the Kato-Ponce inequality, we directly get
\begin{align*}
J_{2}&\leq C\|[\Lambda^{s}, u \cdot \nabla] u \|_{L^{2}}\|\Lambda^{s} u\|_{L^{2}}\nonumber\\
&\leq C\|\nabla u \|_{L^{m_{2}}}\|\Lambda^{s} u\|_{L^{\frac{2m_{2}}{m_{2}-2}}}\|\Lambda^{s} u\|_{L^{2}}
\nonumber\\
&\leq C\|u \|_{H^{\frac{5}{2}}}\|\Lambda^{s} u\|_{L^{2}}^{1-\frac{3}{\alpha m_{2}}}\|\Lambda^{s+\alpha} u\|_{L^{2}}^{\frac{3}{\alpha m_{2}}}\|\Lambda^{s} u\|_{L^{2}}
\nonumber\\
&\leq
\frac{1}{16}\|\Lambda^{s+\alpha}
  u\|_{L^{2}}^{2}+C(\|u\|_{L^{2}}^{2}+\|\Lambda^{\frac{5}{2}}u\|_{L^{2}}^{2})
\|\Lambda^{s}u\|_{L^{2}}^{2},
\end{align*}
where $\max\{\frac{3}{\alpha},\,2\}<m_{2}<\infty$.
Similarly to \eqref{sdsfew22}, one may conclude
\begin{align}
J_{3}=& -\int_{\mathbb{R}^{3}}
[\Lambda^{s}\partial_{x_{i}}, u_{i}
] w\Lambda^{s}w\,dx
\nonumber\\
\leq& \|[\Lambda^{s}\partial_{x_{i}}, u_{i}
] w \|_{L^{2}}\|\Lambda^{s}w\|_{L^{2}}
\nonumber\\
\leq& C(\|\nabla u \|_{L^{m_{1}}}\|\Lambda^{s}w\|_{L^{\frac{2m_{1}}{m_{1}-2}}}+
\|\Lambda^{s+1}u\|_{L^{2}}\|w\|_{L^{\infty}})\|\Lambda^{s}w\|_{L^{2}}
\nonumber\\
\leq& C(\|u \|_{H^{\frac{5}{2}}}\|\Lambda^{s}w\|_{L^{2}}^{1-\frac{3}{\beta m_{1}}}\|\Lambda^{s+\beta}w\|_{L^{2}}
^{\frac{3}{\beta m_{1}}}+
\|u \|_{H^{s+\alpha}}\| w \|_{L^{\infty}})\|\Lambda^{s}w\|_{L^{2}}
\nonumber\\
\leq& \frac{1}{16}\|\Lambda^{s+\alpha}
  u\|_{L^{2}}^{2}+\frac{1}{16}\|\Lambda^{s+\beta}w \|_{L^{2}}^{2}+C(1+\|u\|_{L^{2}}^{2}+\|\Lambda^{\frac{5}{2}}u\|_{L^{2}}^{2}+\| w \|_{L^{\infty}}^{2})\nonumber\\& \times
(\|\Lambda^{s}u\|_{L^{2}}^{2}+\|\Lambda^{s}w\|_{L^{2}}^{2}).\nonumber
\end{align}
Combining the above estimates yields
\begin{align} \label{dfddfr88h}&
 \frac{d}{dt}(\|\Lambda^{s} u\|_{L^{2}}^{2}+\|\Lambda^{s}
w\|_{L^{2}}^{2})+\|\Lambda^{s+\alpha}
  u\|_{L^{2}}^{2}+\|\Lambda^{s+\beta}  w\|_{L^{2}}^{2}\nonumber\\&\leq C(1+\|u\|_{L^{2}}^{2}
  +\|\Lambda^{\frac{5}{2}}u\|_{L^{2}}^{2}+\| w \|_{L^{\infty}}^{2})
(\|\Lambda^{s}u\|_{L^{2}}^{2}+\|\Lambda^{s}w\|_{L^{2}}^{2}).
\end{align}
By \eqref{AZ001} and \eqref{sfdserndf},
$$\int_{0}^{t}{(1+\|u(\tau)\|_{L^{2}}^{2}+\|\Lambda^{\frac{5}{2}}u(\tau)\|_{L^{2}}^{2}+\| w(\tau)\|_{L^{\infty}}^{2})\,d\tau}\leq
C(t,\,u_{0},\,w_{0}).$$
Applying the standard Gronwall inequality to \eqref{dfddfr88h} implies
\begin{align}  \|\Lambda^{s} u(t)\|_{L^{2}}^{2}+\|\Lambda^{s}
w(t)\|_{L^{2}}^{2}+\int_{0}^{t}{(\|\Lambda^{s+\alpha}
  u(\tau)\|_{L^{2}}^{2}+\|\Lambda^{s+\beta}  w(\tau)\|_{L^{2}}^{2})\,d\tau}\leq C(t,\,u_{0},\,w_{0}),\nonumber
  \end{align}
which along with \eqref{t301} gives
  \begin{align} \label{vbdghrtyrt21}\|u(t)\|_{H^{s}}^{2}+\|
w(t)\|_{H^{s}}^{2}+\int_{0}^{t}{(\|\Lambda^{\alpha}
  u(\tau)\|_{H^{s}}^{2}+\|\Lambda^{\beta}  w(\tau)\|_{H^{s}}^{2})\,d\tau}\leq C(t,\,u_{0},\,w_{0}).\end{align}
With \eqref{vbdghrtyrt21} in hand, the uniqueness follows directly due to $s>\frac52$. In fact, let $(u, w, p)$ and
$(\overline{u}, \overline{w}, \overline{p})$ be two solutions to (\ref{3DMP}) with the same initial data. Letting $$\delta u=u-\overline{u},\ \ \ \delta w=w-\overline{w},\ \ \ \delta p=p-\overline{p},$$
we thus have
\begin{equation}\label{difeDMP}
\left\{\begin{array}{l}
\partial_t\delta u+(u\cdot\nabla) \delta u
       +(-\Delta)^{\alpha} \delta u +\nabla \delta p = \nabla \times \delta w-(\delta u\cdot\nabla) \overline{u},\\
\partial_t\delta w + (u\cdot\nabla) \delta w +\delta w
            +(-\Delta)^{\beta} \delta w =\nabla\times \delta u+ \nabla \nabla\cdot \delta w-(\delta u\cdot\nabla) \overline{w},
             \vspace{2mm}\\
\delta u(x,0)=0,\quad \delta w(x,0)=0,
\end{array}\right.
\end{equation}
Multiplying $(\eqref{difeDMP}_{1},\eqref{difeDMP}_{2})$ by $(\delta u,\delta w)$ and integrating over the whole space, one obtains
\begin{align}
&\frac{1}{2}\frac{d}{dt}(\|\delta u(t)\|_{L^{2}}^{2}+\|\delta w(t)\|_{L^{2}}^{2})
+\|\Lambda^{\alpha}\delta u\|_{L^{2}}^{2}
+ \|\Lambda^{\beta}\delta w\|_{L^{2}}^{2} + \|\nabla\cdot\delta w\|_{L^{2}}^{2}\nonumber\\
&\leq  \int_{\mathbb{R}^{3}}{\nabla\times \delta w\cdot \delta u\,dx}+\int_{\mathbb{R}^{3}}{\nabla\times \delta u\cdot \delta w\,dx}- \int_{\mathbb{R}^{3}}{(\delta u\cdot\nabla) \overline{u}\cdot  \delta u\,dx}\nonumber\\&\quad-\int_{\mathbb{R}^{3}}{(\delta u\cdot\nabla) \overline{w}\cdot  \delta w\,dx}
\nonumber\\
&\leq 2 \|\Lambda^{\kappa}\delta u\|_{L^2}\|\Lambda^{1-\kappa}\delta w\|_{L^2}+C\|\nabla  \overline{u}\|_{L^{\infty}}\|\delta u\|_{L^{2}}^{2}+C\|\nabla  \overline{w}\|_{L^{\infty}}\|\delta u\|_{L^{2}}\|\delta w\|_{L^{2}}
\nonumber\\
&\leq
2 (\|\delta u\|_{L^2}^{1-\frac{\kappa}{\alpha}}\|\Lambda^{\alpha}\delta u\|_{L^2}
 ^{\frac{\kappa}{\alpha}})(\|\delta w\|_{L^2}^{1-\frac{1-\kappa}{\beta}}
 \|\Lambda^{\beta}\delta w\|_{L^2}
 ^{\frac{1-\kappa}{\beta}})\nonumber\\
&\quad+C(\|\nabla  \overline{u}\|_{L^{\infty}}+\|\nabla  \overline{w}\|_{L^{\infty}})(\|\delta u\|_{L^{2}}^{2}+\|\delta w\|_{L^{2}}^{2})
\nonumber\\
&\leq
\frac{1}{2}\|\Lambda^\alpha \delta u \|_{L^2}^2+\frac{1}{2}\|\Lambda^{\beta}\delta w\|_{L^2}
 ^{2}
 +C(\|\delta u\|^2_{L^2}+\| \delta w \|^2_{L^2})\nonumber\\
&\quad+C(\|\overline{u}\|_{H^{s}}+\| \overline{w}\|_{_{H^{s}}})(\|\delta u\|_{L^{2}}^{2}+\|\delta w\|_{L^{2}}^{2}),\nonumber
\end{align}
where we have used $s>\frac52$ and have selected $1-\beta<\kappa< \alpha$. This yields
$$\frac{d}{dt}(\|\delta u(t)\|_{L^{2}}^{2}+\|\delta w(t)\|_{L^{2}}^{2})\leq C(1+\|\overline{u}\|_{H^{s}}+\| \overline{w}\|_{_{H^{s}}})(\|\delta u\|_{L^{2}}^{2}+\|\delta w\|_{L^{2}}^{2}),$$
which together with $\eqref{difeDMP}_{3}$ and the Gronwall inequality implies
$$\delta u(t)=\delta w(t)=0.$$
We thus obtain the uniqueness. Therefore, this completes the proof of Theorem \ref{Th1}.
\end{proof}

\vskip .3in
\section{The proof of Theorem \ref{addTh2}}
\label{sec3}
\setcounter{equation}{0}

In this section, we are going to prove Theorem \ref{addTh2}. Our attention is focused on the case when  $\alpha=\frac{7}{4}$ since $\alpha>\frac{7}{4}$ is even simpler to handle.
We begin with the basic $L^2$-estimate.

\begin{lemma}\label{loglema1}
Assume $(u_{0},w_{0})$ satisfies the assumptions stated in Theorem \ref{addTh2}.
Then the corresponding solution $(u, w)$
of  \eqref{log3DMP}  admits the following bound for any $t>0$
\begin{eqnarray} \label{logtt001}
\|u(t)\|^2_{L^2}+\|w(t)\|_{L^2}^2
   + \int_{0}^{t}{(\|\Lambda^{r}u(\tau)\|_{L^{2}}^{2}+
\|\mathcal{L}u(\tau)\|_{L^{2}}^{2})\,d\tau} \leq
C(t,\,u_{0},\,w_{0})
\end{eqnarray}
for any $r \in [0,\,\frac{7}{4})$.
\end{lemma}

\begin{proof}
It follows from the first equation of (\ref{log3DMP}) that
\begin{align*} &
\frac{1}{2}\frac{d}{dt}
  ( \|u(t)\|^2_{L^2}+\|w(t)\|_{L^2}^2)+\int_{\mathbb{R}^{3}}{  \mathcal{L}^{2}u\cdot u\,dx}+2\|w\|^2_{L^2}+\|\nabla\cdot w\|^2_{L^2}\nonumber\\
 &= \int_{\mathbb{R}^{3}} \Big( (\nabla \times w)\cdot u+ (\nabla \times u)\cdot w  \Big) dx.
\end{align*}
By Plancherel's theorem,
\begin{eqnarray}
\int_{\mathbb{R}^{3}}{  \mathcal{L}^{2}u\cdot u\,dx}=
\int_{\mathbb{R}^{3}}{ \frac{|\xi|^{\frac{7}{2}}}{g^{2}(|\xi|)}|\widehat{u}(\xi)|^{2}\,d\xi}=\|\mathcal{L}u\|_{L^{2}}^{2}.
\nonumber
\end{eqnarray}
Based on the assumptions on $g$, $g$ grows logarithmically and we conclude that for any fixed $\sigma>0$, there exists $N=N(\sigma)$ satisfying
$$g(r)\leq \widetilde{C}r^{\sigma},\quad \forall\,r\geq N,$$
with some constant $\widetilde{C}=\widetilde{C}(\sigma)$.
Therefore, we have
\begin{align}\label{logtt002}
\|\mathcal{L}u\|_{L^{2}}^{2}&=
\int_{|\xi|< N(\sigma)}{ \frac{|\xi|^{\frac{7}{2}}}{g^{2}(|\xi|)}|\widehat{u}(\xi)|^{2}\,d\xi}
+\int_{|\xi|\geq N(\sigma)}{ \frac{|\xi|^{\frac{7}{2}}}{g^{2}(|\xi|)}|\widehat{u}(\xi)|^{2}\,d\xi}\nonumber\\
&\geq
\int_{|\xi|\geq N(\sigma)}{ \frac{|\xi|^{\frac{7}{2}}}{\big[\widetilde{C}|\xi|^{\sigma}
\big]^{2}}|\widehat{u}(\xi)|^{2}\,d\xi}\nonumber\\
&=
\int_{\mathbb{R}^{3}}{ \frac{|\xi|^{\frac{7}{2}}}{\big[\widetilde{C}|\xi|^{\sigma}
\big]^{2}}|\widehat{u}(\xi)|^{2}\,d\xi}-\int_{|\xi|<N(\sigma)}{ \frac{|\xi|^{\frac{7}{2}}}{\big[\widetilde{C}|\xi|^{\sigma}
\big]^{2}}|\widehat{u}(\xi)|^{2}\,d\xi}
\nonumber\\
&\geq  C_{0}\|\Lambda^{\frac{7-4\sigma}{4}}u\|_{L^{2}}^{2}-
\widetilde{C_{0}}\|u\|_{L^{2}}^{2},
\end{align}
where $C_{0}$ and $\widetilde{C_{0}}$ depend only on $\sigma$. Now if we further fix $\sigma$ satisfying $\sigma\in (0,\frac{3}{4})$, then we get by combining all the  estimates above,
\begin{align*}
 &  \frac{1}{2}\frac{d}{dt}
  ( \|u(t)\|^2_{L^2}+\|w(t)\|_{L^2}^2)
  +\frac{1}{2}\|\mathcal{L}u\|_{L^{2}}^{2}+
  \frac{C_{0}}{2}\|\Lambda^{\frac{7-4\sigma}{4}}u\|_{L^{2}}^{2}+2\|w\|^2_{L^2}+\|\nabla\cdot w\|^2_{L^2}\nonumber\\
 &\leq \widetilde{C_{0}}\|u\|_{L^{2}}^{2}+\int_{\mathbb{R}^{3}} \Big( (\nabla \times w)\cdot u+ (\nabla \times u)\cdot w  \Big) dx\\
 & \leq \widetilde{C_{0}}\|u\|_{L^{2}}^{2}+2 \|\nabla u\|_{L^2}\|w\|_{L^2}
 \\
 & \leq   \widetilde{C_{0}}\|u\|_{L^{2}}^{2}+2 (\|u\|_{L^2}^{\frac{3- 4\sigma}{7-4\sigma}}\|\Lambda^{\frac{7-4\sigma}{4}}u\|_{L^2}
 ^{\frac{4}{7-4\sigma}})\|w\|_{L^2}
 \\
 &\leq \frac{C_{0}}{4}\|\Lambda^{\frac{7-4\sigma}{4}}u\|_{L^{2}}^{2}
 +C(\|u\|^2_{L^2}+\|w\|^2_{L^2}).
\end{align*}
Therefore, for any $\sigma\in (0,\frac{3}{4})$,
\begin{eqnarray}
\frac{d}{dt} (\|u(t)\|^2_{L^2}+\|w(t)\|_{L^2}^2)+ \|\Lambda^{\frac{7-4\sigma}{4}}u\|_{L^{2}}^{2}+\|\mathcal{L}u\|_{L^{2}}^{2} \leq C(\|u\|^2_{L^2}+\|w\|^2_{L^2}).\nonumber
\end{eqnarray}
Using the Gronwall inequality yields
$$\|u(t)\|^2_{L^2}+\|w(t)\|_{L^2}^2
   + \int_{0}^{t}{
(\|\Lambda^{\frac{7-4\sigma}{4}}u(\tau)\|_{L^{2}}^{2}+
\|\mathcal{L}u(\tau)\|_{L^{2}}^{2})\,d\tau} \leq
C(t,\,u_{0},\,w_{0}).$$
The estimate \eqref{logtt001} then follows.
\end{proof}

\vskip .1in
The following estimate plays an important role in proving the main result.
\begin{lemma}\label{loglema2}
Assume $(u_{0},w_{0})$ satisfies the assumptions stated in Theorem \ref{addTh2}.
Then the corresponding solution $(u, w)$
of  \eqref{log3DMP}  admits the following bound for any $t>0$
\begin{align*}
\Big\|\frac{\Lambda^{\frac{3}{4}}}{g(\Lambda)}u(t)\Big\|_{L^{2}}^{2} +\int_{0}^{t}{
\Big\|\frac{\Lambda^{\frac{5}{2}}}{g^{2}(\Lambda)}u(\tau)\Big\|_{L^{2}}^{2}\,d\tau}\leq C(t,\,u_{0},\,w_{0}).
\end{align*}
In particular, there holds for any $\epsilon_{1}\in(0,\,\frac{5}{2})$
\begin{eqnarray}\label{xsxrt8}
\int_{0}^{t}{
\|u(\tau)\|_{H^{\frac{5}{2}-\epsilon_{1}}}^{2}\,d\tau}\leq C(t,\,u_{0},\,w_{0}).
\end{eqnarray}
\end{lemma}
 \begin{proof}
Taking the inner product of  $\eqref{log3DMP}_{1}$ with $\frac{\Lambda^{\frac{3}{2}}}{g^{2}(\Lambda)}u$, we have
\begin{align}
 \frac{1}{2}\frac{d}{dt} \Big\|\frac{\Lambda^{\frac{3}{4}}}{g(\Lambda)}u(t)\Big\|_{L^{2}}^{2}+
\Big\|\frac{\Lambda^{\frac{5}{2}}}{g^{2}(\Lambda)}u\Big\|_{L^{2}}^{2} =&
-\int_{\mathbb{R}^{3}}{ (u \cdot \nabla u)\cdot \frac{\Lambda^{\frac{3}{2}}}{g^{2}(\Lambda)}u\,dx}\nonumber\\& +\int_{\mathbb{R}^{3}}{ \nabla\times w\cdot \frac{\Lambda^{\frac{3}{2}}}{g^{2}(\Lambda)}u\,dx}.\nonumber
\end{align}
The Young inequality ensures
\begin{align}
\Big|\int_{\mathbb{R}^{3}}{ \nabla\times w\cdot \frac{\Lambda^{\frac{3}{2}}}{g^{2}(\Lambda)}u\,dx}\Big|&\leq  \|w\|_{L^{2}}\Big\|\nabla \frac{\Lambda^{\frac{3}{2}}}{g^{2}(\Lambda)}u\Big\|_{L^{2}}\nonumber\\
&\leq \frac{1}{16}\Big\|\frac{\Lambda^{\frac{5}{2}}}{g^{2}(\Lambda)}u\Big\|_{L^{2}}^{2}+C\|w\|_{L^{2}}^{2}.\nonumber
\end{align}
Following the proof of \eqref{logtt002}, it is not difficult to check that for any $\epsilon\in (0,\,\frac{3}{4})$,
\begin{eqnarray}\label{loglogtt004}
\Big\|\frac{\Lambda^{\frac{3}{4}}}{g(\Lambda)}u\Big\|_{L^{2}}^{2}\geq C_{1}\|\Lambda^{\frac{3}{4}-\epsilon}u\|_{L^{2}}^{2}
-C_{2}\|u\|_{L^{2}}^{2}.
\end{eqnarray}
In view of the fact $\nabla\cdot u=0$ and the Gagliardo-Nirenberg inequality, we infer that
\begin{align}
\Big|\int_{\mathbb{R}^{3}}{ (u \cdot \nabla u)\cdot \frac{\Lambda^{\frac{3}{2}}}{g^{2}(\Lambda)}u\,dx}\Big|\leq& C\|uu\|_{L^{2}}\Big\|\nabla \frac{\Lambda^{\frac{3}{2}}}{g^{2}(\Lambda)}u\Big\|_{L^{2}}\nonumber\\
\leq&  C\|u\|_{L^{4}}^{2}\Big\| \frac{\Lambda^{\frac{5}{2}}}{g^{2}(\Lambda)}u\Big\|_{L^{2}}
\nonumber\\
\leq & C\|\Lambda^{\frac{1}{2}}u\|_{L^{2}}\|\Lambda u\|_{L^{2}}\Big\| \frac{\Lambda^{\frac{5}{2}}}{g^{2}(\Lambda)}u\Big\|_{L^{2}}\nonumber\\
\leq &\frac{1}{16}\Big\| \frac{\Lambda^{\frac{5}{2}}}{g^{2}(\Lambda)}u\Big\|_{L^{2}}^{2}+C
\|\Lambda^{\frac{1}{2}}u\|_{L^{2}}^{2}\|\Lambda u\|_{L^{2}}^{2}
\nonumber\\
\leq& \frac{1}{16}\Big\| \frac{\Lambda^{\frac{5}{2}}}{g^{2}(\Lambda)}u\Big\|_{L^{2}}^{2}+C
\|u\|_{L^{2}}^{2}\|\Lambda u\|_{L^{2}}^{2}
\nonumber\\& +C\Big\|\frac{\Lambda^{\frac{3}{4}}}{g(\Lambda)}u\Big\|_{L^{2}}^{2}
\|\Lambda u\|_{L^{2}}^{2},\nonumber
\end{align}
where in the last line we have applied \eqref{loglogtt004} with $\epsilon=\frac{1}{4}$.
We thus conclude
\begin{align}
 \frac{d}{dt} \Big\|\frac{\Lambda^{\frac{3}{4}}}{g(\Lambda)}u(t)\Big\|_{L^{2}}^{2} +
\Big\|\frac{\Lambda^{\frac{5}{2}}}{g^{2}(\Lambda)}u\Big\|_{L^{2}}^{2} \leq&
C\|w\|_{L^{2}}^{2}+C
\|u\|_{L^{2}}^{2}\|\Lambda u\|_{L^{2}}^{2}\nonumber\\& +
C\|\Lambda u\|_{L^{2}}^{2}\Big\|\frac{\Lambda^{\frac{3}{4}}}{g(\Lambda)}u\Big\|_{L^{2}}^{2}.\nonumber
\end{align}
The estimate \eqref{logtt001} with $r=1$ gives
\begin{eqnarray}
\|u(t)\|^2_{L^2}+\|w(t)\|_{L^2}^2
   + \int_{0}^{t}{\|\Lambda u(\tau)\|_{L^{2}}^{2}\,d\tau} \leq
C(t,\,u_{0},\,w_{0}).\nonumber
\end{eqnarray}
By the Gronwall inequality,
$$ \Big\|\frac{\Lambda^{\frac{3}{4}}}{g(\Lambda)}u(t)\Big\|_{L^{2}}^{2} +\int_{0}^{t}{
\Big\|\frac{\Lambda^{\frac{5}{2}}}{g^{2}(\Lambda)}u(\tau)\Big\|_{L^{2}}^{2}\,d\tau}\leq C(t,\,u_{0},\,w_{0}).$$
The desired bound \eqref{xsxrt8} can be deduced following the proof of \eqref{logtt002}.
This completes the proof of Lemma \ref{loglema2}.
\end{proof}

The following is our main lemma in the proof of  Theorem \ref{addTh2}.

\begin{lemma}\label{dvefvb21}
Assume $(u_{0},w_{0})$ satisfies the assumptions stated in Theorem \ref{Th1}.
Then the corresponding solution $(u, w)$
of  \eqref{log3DMP} admits the following bound for any $t>0$
\begin{eqnarray}\label{cber75b9xc}
\|\Lambda^{k_{1}}u(t)\|_{L^{2}}^{2}
+\|\Lambda^{\widetilde{\sigma}}w(t)\|_{L^{2}}^{2}
+\int_{0}^{t}{\|\mathcal{L}\Lambda^{k_{1}}
u(\tau)\|_{L^{2}}^{2} \,d\tau}\leq
C(t,\,u_{0},\,w_{0}),
\end{eqnarray}
where $k_{1}< \widetilde{\sigma}+\frac{3}{4}$ for any $\widetilde{\sigma}\in (0,\frac{5}{2})$.
In particular, we have, by taking $k_{1}>\frac{3}{4}$,
 \begin{eqnarray}\label{csdfnmy}
\int_{0}^{t}{
 \|\nabla u
(\tau)\|_{L^{\infty}}^{2} \,d\tau}\leq
C(t,\,u_{0},\,w_{0}).
\end{eqnarray}
Moreover, if one takes $\widetilde{\sigma}\in (\frac{3}{2},\frac{5}{2})$, then
 \begin{eqnarray}\label{cxdf78nmycv}
\|w(t)\|_{L^{\infty}}\leq
C(t,\,u_{0},\,w_{0}).
\end{eqnarray}
\end{lemma}

\begin{proof}
Applying $\Lambda^{\widetilde{\sigma}}$ to $ \eqref{log3DMP}_{2}$ and taking the
inner product with $\Lambda^{\widetilde{\sigma}}w$ yield
\begin{align} &
\frac{1}{2}\frac{d}{dt}\|\Lambda^{\widetilde{\sigma}}w(t)\|_{L^{2}}^{2}
+2\|\Lambda^{\widetilde{\sigma}}w\|^2_{L^2}+\|\Lambda^{\widetilde{\sigma}}\nabla\cdot w\|_{L^{2}}^{2}\nonumber\\&= \int_{\mathbb{R}^{3}}
\Lambda^{\widetilde{\sigma}}\big(\nabla\times u\big)\Lambda^{\widetilde{\sigma}}w\,dx -\int_{\mathbb{R}^{3}}
[\Lambda^{\widetilde{\sigma}}, u \cdot
\nabla] w\Lambda^{\widetilde{\sigma}}w\,dx\nonumber\\
&:= K_{1}+K_{2}.\nonumber
\end{align}
Applying $\Lambda^{k_{1}}$ to the system $ \eqref{log3DMP}_{1}$
and taking the $L^{2}$ inner product with $\Lambda^{k_{1}}u$, we have
\begin{align}
& \frac{1}{2}\frac{d}{dt}\|\Lambda^{k_{1}}u(t)\|_{L^{2}}^{2}
+\|\mathcal{L}\Lambda^{k_{1}}
u\|_{L^{2}}^{2}\nonumber\\
&=  -\int_{\mathbb{R}^{3}}{[\Lambda^{k_{1}},
u\cdot\nabla]u\cdot \Lambda^{k_{1}}u\,dx}+\int_{\mathbb{R}^{3}}{\Lambda^{k_{1}}(\nabla\times w)
\Lambda^{k_{1}}u\,dx}\nonumber\\
&:=  K_{3}+K_{4}.\nonumber
\end{align}
First, we deduce from the proof of \eqref{logtt002} that for any $\epsilon_{2}\in (0,\,k_{1}+\frac{7}{4})$,
\begin{align*}
\|\mathcal{L}\Lambda^{k_{1}}
u\|_{L^{2}}^{2}\geq C_{3}\|\Lambda^{k_{1}+\frac{7}{4}-\epsilon_{2}}
u\|_{L^{2}}^{2}
-C_{4}\|u\|_{L^{2}}^{2}.
\end{align*}
It thus follows from the Gagliardo-Nirenberg inequality that
\begin{align}
K_{1}&\leq C\|\Lambda^{\widetilde{\sigma}+1}u\|_{L^{2}}
\|\Lambda^{\widetilde{\sigma}}w \|_{L^{2}}\nonumber\\
&\leq C(\|u\|_{L^{2}}+\|\mathcal{L}\Lambda^{k_{1}}
u\|_{L^{2}})
\|\Lambda^{\widetilde{\sigma}}w \|_{L^{2}}\nonumber\\
&\leq \frac{1}{16}\|\mathcal{L}\Lambda^{k_{1}}
u\|_{L^{2}}^{2}+C
\|\Lambda^{\widetilde{\sigma}}w \|_{L^{2}}^{2}+C\|u\|_{L^{2}}^{2},\nonumber
\end{align}
where $\widetilde{\sigma}-\frac{3}{4}+\epsilon_{2}\leq k_{1}$.
According to $\nabla\cdot u=0$ and the Kato-Ponce inequality, we obtain, for any $\widetilde{\sigma}<\frac{3}{2}$,
\begin{align}
K_{2} =&-\int_{\mathbb{R}^{3}}
[\Lambda^{\widetilde{\sigma}}\partial_{x_{i}}, u_{i}
] w\Lambda^{\widetilde{\sigma}}w\,dx
\nonumber\\
 \leq&\|[\Lambda^{\widetilde{\sigma}}\partial_{x_{i}}, u_{i}
] w \|_{L^{2}}\|\Lambda^{\widetilde{\sigma}}w\|_{L^{2}}
\nonumber\\
 \leq&C(\|\nabla u \|_{L^{\infty}}\|\Lambda^{\widetilde{\sigma}}w\|_{L^{2}}+
\|\Lambda^{\widetilde{\sigma}+1}u\|_{L^{\frac{3}{\widetilde{\sigma}}}}
\|w\|_{L^{\frac{6}{3-2\widetilde{\sigma}}}})\|\Lambda^{\widetilde{\sigma}}w\|_{L^{2}}
\nonumber\\
 \leq&
C(\|\nabla u \|_{L^{\infty}}\|\Lambda^{\widetilde{\sigma}}w\|_{L^{2}}+
\|\Lambda^{\frac{5}{2}}u\|_{L^{2}}
\|\Lambda^{\widetilde{\sigma}}w\|_{L^{2}})\|\Lambda^{\widetilde{\sigma}}w\|_{L^{2}}
\nonumber\\
 \leq&C(\|\nabla u \|_{L^{\infty}}+\|\Lambda^{\frac{5}{2}}u\|_{L^{2}})
\|\Lambda^{\widetilde{\sigma}}w\|_{L^{2}}^{2}.\nonumber
\end{align}
For any $\widetilde{\sigma}\in(1,\,\frac{5}{2})$, $K_2$ can  be bounded by
\begin{align}
K_{2} \leq&\|[\Lambda^{\widetilde{\sigma}}, u\cdot\nabla
] w \|_{L^{2}}\|\Lambda^{\widetilde{\sigma}}w\|_{L^{2}}
\nonumber\\
 \leq&C(\|\nabla u \|_{L^{\infty}}\|\Lambda^{\widetilde{\sigma}}w\|_{L^{2}}+
\|\Lambda^{\widetilde{\sigma}}u\|_{L^{\frac{3}{\widetilde{\sigma}-1}}}
\|\nabla w\|_{L^{\frac{6}{5-2\widetilde{\sigma}}}})\|\Lambda^{\widetilde{\sigma}}w\|_{L^{2}}
\nonumber\\
 \leq&
C(\|\nabla u \|_{L^{\infty}}\|\Lambda^{\widetilde{\sigma}}w\|_{L^{2}}+
\|\Lambda^{\frac{5}{2}}u\|_{L^{2}}
\|\Lambda^{\widetilde{\sigma}}w\|_{L^{2}})\|\Lambda^{\widetilde{\sigma}}w\|_{L^{2}}
\nonumber\\
 \leq&C(\|\nabla u \|_{L^{\infty}}+\|\Lambda^{\frac{5}{2}}u\|_{L^{2}})
\|\Lambda^{\widetilde{\sigma}}w\|_{L^{2}}^{2}.\nonumber
\end{align}
There,  for any $\widetilde{\sigma}\in(0,\,\frac{5}{2})$,
\begin{eqnarray}
K_{2}\leq
C(\|\nabla u \|_{L^{\infty}}+\|\Lambda^{\frac{5}{2}}u\|_{L^{2}})
\|\Lambda^{\widetilde{\sigma}}w\|_{L^{2}}^{2}.\nonumber
\end{eqnarray}
By the Kato-Ponce inequality, we obtain
\begin{align}\label{opmdwzv78}
K_{3} \leq& C\|[\Lambda^{k_{1}},
u\cdot\nabla]u\|_{L^{2}}\|\Lambda^{k_{1}}u\|_{L^{2}}
\nonumber\\ \leq&
C\|\nabla u\|_{L^{\infty}}
\|\Lambda^{k_{1}}u\|_{L^{2}}^{2}.
\end{align}
By the Gagliardo-Nirenberg inequality, one has
\begin{align}
K_{4} \leq&C\|\Lambda^{1+2k_{1}-\widetilde{\sigma}}u\|_{L^{2}}
\|\Lambda^{\widetilde{\sigma}}w \|_{L^{2}}\nonumber\\
 \leq&
C\|u\|_{H^{k_{1}+\frac{7}{4}-\epsilon_{2}}}
\|\Lambda^{\widetilde{\sigma}}w \|_{L^{2}}
\nonumber\\
 \leq&
C(\|\mathcal{L}\Lambda^{k_{1}}
u\|_{L^{2}}+\|u\|_{L^2})
\|\Lambda^{\widetilde{\sigma}}w \|_{L^{2}}
\nonumber\\
 \leq& \frac{1}{16}\|\mathcal{L}\Lambda^{k_{1}}
u\|_{L^{2}}^{2}+C
\|\Lambda^{\widetilde{\sigma}}w \|_{L^{2}}^{2}+C\|u\|_{L^{2}}^{2},\nonumber
\end{align}
where $k_{1}<\widetilde{\sigma}+\frac{3}{4}-\epsilon_{2}$.
Collecting all the estimates above yields
\begin{align} &
 \frac{d}{dt}(\|\Lambda^{\widetilde{\sigma}}w(t)\|_{L^{2}}^{2}
 +\|\Lambda^{k_{1}}u(t)\|_{L^{2}}^{2})+\|\Lambda^{k_{1}+\frac{7}{4}-\epsilon_{2}}
u\|_{L^{2}}^{2}
\nonumber\\&\leq  C(1+\|\nabla u \|_{L^{\infty}}+\|\Lambda^{\frac{5}{2}}u\|_{L^{2}})
(\|\Lambda^{\widetilde{\sigma}}w\|_{L^{2}}^{2}+\|\Lambda^{k_{1}}u\|_{L^{2}}^{2})
+C\|u\|_{L^{2}}^{2}.\nonumber
\end{align}
If we set
$$A(t):=\|\Lambda^{\widetilde{\sigma}}w(t)\|_{L^{2}}^{2}
 +\|\Lambda^{k_{1}}u(t)\|_{L^{2}}^{2},\quad B(t):=\|\Lambda^{k_{1}+\frac{7}{4}-\epsilon_{2}}
u(t)\|_{L^{2}}^{2}+\|\mathcal{L}\Lambda^{k_{1}}
u(t)\|_{L^{2}}^{2},$$
$$ f(t):=C+C\|u\|_{L^{2}}^{2},$$
then
\begin{eqnarray} \label{ccxdwkl2}
 \frac{d}{dt}A(t)+B(t)
 \leq    C(1+\|\nabla u \|_{L^{\infty}}+\|\Lambda^{\frac{5}{2}}u\|_{L^{2}})A(t)
+f(t).
\end{eqnarray}
We now bound $\|\na u\|_{L^\infty}$ via the high-low frequency technique (see, e.g.,
\cite{wujmfm}).  By the Littlewood-Paley decomposition  (see Appendix for details),
\begin{eqnarray}
\|\nabla u\|_{L^{\infty}}\leq \|\Delta_{-1}\nabla u\|_{L^{\infty}}+\sum_{l=0}^{N-1}\|\Delta_{l}\nabla u\|_{L^{\infty}}+\sum_{l=N}^{\infty}\|\Delta_{l}\nabla u\|_{L^{\infty}}.\nonumber
\end{eqnarray}
By the Bernstein inequality (see Lemma \ref{vfgty8xc}), one has
$$\|\Delta_{-1}\nabla u\|_{L^{\infty}}\leq C\|u\|_{L^{2}}$$
and
\begin{align}\label{ctrflkkm2}
\sum_{l=N}^{\infty}\|\Delta_{l}\nabla u\|_{L^{\infty}} \leq&
 C \sum_{l=N}^{\infty}2^{\frac{5}{2}l}\|\Delta_{l}u\|_{L^{2}}\nonumber\\ =&
C \sum_{l=N}^{\infty}2^{l(\epsilon_{2}+\frac{3}{4}-k_{1})}\|\Delta_{l}\Lambda^{k_{1}
+\frac{7}{4}-\epsilon_{2}}u\|_{L^{2}}\nonumber\\ \leq& C2^{N(\epsilon_{2}+\frac{3}{4}-k_{1})}\|\Lambda^{k_{1}
+\frac{7}{4}-\epsilon_{2}}u\|_{L^{2}},
\end{align}
where we have used $k_{1}>\epsilon_{2}+\frac{3}{4}$.
According to the Bernstein inequality again and the Plancherel theorem, the middle term can be handled as follows,
\begin{align}
\sum_{l=0}^{N-1}\|\Delta_{l}\nabla u\|_{L^{\infty}} \leq C &\sum_{l=0}^{N-1}2^{\frac{5}{2}l}\|\Delta_{l} u\|_{L^{2}}\leq
C \sum_{l=0}^{N-1} \|\Delta_{l} \Lambda^{\frac{5}{2}}u\|_{L^{2}}\nonumber\\
 \leq &
C \sum_{l=0}^{N-1} \|\varphi(2^{-l}\xi) |\xi|^{\frac{5}{2}}\widehat{u}(\xi)\|_{L^{2}}\nonumber\\
 =&C\sum_{l=0}^{N-1} \Big\|\varphi(2^{-l}\xi)
g^{2}(|\xi|)
\frac{|\xi|^{\frac{5}{2}}}{ g^{2}(|\xi|) }\widehat{u}(\xi)\Big\|_{L^{2}}
\nonumber\\
 \leq &C\sum_{l=0}^{N-1} g^{2}(2^{l}) \Big\|\frac{|\xi|^{\frac{5}{2}}}{g^{2}(|\xi|) }\widehat{\Delta_{l}u}(\xi)\Big\|_{L^{2}}
\nonumber\\
 \leq &C\Big(\sum_{l=0}^{N-1} g^{4}(2^{l})\Big)^{\frac{1}{2}}\left(\sum_{l=0}^{N-1}
\Big\|\frac{|\xi|^{\frac{5}{2}}}{ g^{2}(|\xi|)}\widehat{\Delta_{l}u}(\xi)\Big\|_{L^{2}}^{2}\right)^{\frac{1}{2}}
\nonumber\\
\leq &Cg^{2}(2^{N})\Big(\sum_{l=1}^{N-1} 1\Big)^{\frac{1}{2}} \Big\|\frac{\Lambda^{\frac{5}{2}}}{ g^{2}(\Lambda)}u \Big\|_{L^{2}}
\nonumber\\
\leq &Cg^{2}(2^{N})\sqrt{N} \Big\|\frac{\Lambda^{\frac{5}{2}}}{ g^{2}(\Lambda)}u \Big\|_{L^{2}},\nonumber
\end{align}
where we used the fact that $g$ is a non-decreasing function and $\phi(2^{-l} \xi)$ denotes the symbol of the operator $\Delta_l$ (see Appendix).
Summarizing the above estimates implies
\begin{eqnarray}
\|\nabla u\|_{L^{\infty}}\leq C\|u\|_{L^{2}}+Cg^{2}(2^{N})\sqrt{N} \Big\|\frac{\Lambda^{\frac{5}{2}}}{ g^{2}(\Lambda)}u \Big\|_{L^{2}}+C2^{N(\epsilon_{2}+\frac{3}{4}-k_{1})}\|\Lambda^{k_{1}
+\frac{7}{4}-\epsilon_{2}}u\|_{L^{2}}.\nonumber
\end{eqnarray}
By the same argument, we have
\begin{eqnarray}
\|\Lambda^{\frac{5}{2}}u\|_{L^{2}}\leq C\|u\|_{L^{2}}+Cg^{2}(2^{N})\sqrt{N} \Big\|\frac{\Lambda^{\frac{5}{2}}}{ g^{2}(\Lambda)}u \Big\|_{L^{2}}+C2^{N(\epsilon_{2}+\frac{3}{4}-k_{1})}\|\Lambda^{k_{1}
+\frac{7}{4}-\epsilon_{2}}u\|_{L^{2}}.\nonumber
\end{eqnarray}
The estimate for $\|\Lambda^{\frac{5}{2}}u\|_{L^{2}}$ can be refined to
\begin{eqnarray}\label{qazxdk90hj}
\|\Lambda^{\frac{5}{2}}u\|_{L^{2}}\leq Cg^{2}(2^{N})  \Big\|\frac{\Lambda^{\frac{5}{2}}}{ g^{2}(\Lambda)}u \Big\|_{L^{2}}+C2^{N(\epsilon_{2}+\frac{3}{4}-k_{1})}\|\Lambda^{k_{1}
+\frac{7}{4}-\epsilon_{2}}u\|_{L^{2}}.
\end{eqnarray}
The above estimate \eqref{qazxdk90hj} can be established by invoking  \eqref{ctrflkkm2},
\begin{align}
\|\Lambda^{\frac{5}{2}}u\|_{L^{2}} \leq& \|S_{N}\Lambda^{\frac{5}{2}}u\|_{L^{2}}+\sum_{l=0}^{N-1}
\|\Delta_{l}\Lambda^{\frac{5}{2}}u\|_{L^{2}}
\nonumber\\ \leq& C\left\|\chi(2^{-N}\xi)|\xi|^{\frac{5}{2}}\widehat{u}(\xi)\right\|_{L^{2}}
+C\sum_{l=0}^{N-1}
2^{\frac{5}{2}l}
\|\Delta_{l}u\|_{L^{2}}
\nonumber\\ \leq& C\left\|\chi(2^{-N}\xi)g^{2}(|\xi|)\frac{|\xi|^{\frac{5}{2}}\widehat{u}(\xi)}{g^{2}(|\xi|)}\right\|_{L^{2}}
+C2^{N(\epsilon_{2}+\frac{3}{4}-k_{1})}\|\Lambda^{k_{1}
+\frac{7}{4}-\epsilon_{2}}u\|_{L^{2}}
\nonumber\\ \leq& C g^{2}(2^{N})\left\|\frac{|\xi|^{\frac{5}{2}}\widehat{u}(\xi)}{g^{2}(|\xi|)}
\right\|_{L^{2}}
+C2^{N(\epsilon_{2}+\frac{3}{4}-k_{1})}\|\Lambda^{k_{1}
+\frac{7}{4}-\epsilon_{2}}u\|_{L^{2}}
\nonumber\\ \leq&
Cg^{2}(2^{N})  \Big\|\frac{\Lambda^{\frac{5}{2}}}{ g^{2}(\Lambda)}u \Big\|_{L^{2}}+C2^{N(\epsilon_{2}+\frac{3}{4}-k_{1})}\|\Lambda^{k_{1}
+\frac{7}{4}-\epsilon_{2}}u\|_{L^{2}},\nonumber
\end{align}
where the low frequency cutoff operator $S_{N}$ is given in the first appendix.
By \eqref{ccxdwkl2},
\begin{align}
 \frac{d}{dt}A(t)+B(t)
  \leq&   CA(t)+Cg^{2}(2^{N})\sqrt{N} \Big\|\frac{\Lambda^{\frac{5}{2}}}{ g^{2}(\Lambda)}u \Big\|_{L^{2}}A(t)\nonumber\\
&   +\, C2^{N(\epsilon_{2}+\frac{3}{4}-k_{1})}B^{\frac{1}{2}}(t)A(t)
+f(t).\nonumber
\end{align}
For simplicity, we denote
$$a:=k_{1}-\epsilon_{2}-\frac{3}{4}>0.$$
Choosing $N$ satisfying
$$2^{N}\approx \left(e+A(t)\right)^{\frac{1}{2a}},$$
we obtain
\begin{align}
 \frac{d}{dt}A(t)+B(t)
  \leq&    C\left(1+\Big\|\frac{\Lambda^{\frac{5}{2}}}{ g^{2}(\Lambda)}u \Big\|_{L^{2}}\right)g^{2}\big[\left(e+A(t)\right)^{\frac{1}{2a}}\big]\sqrt{\ln\big(e+A(t)\big)} \big(e+A(t)\big)\nonumber\\& +CB^{\frac{1}{2}}(t)\big(e+A(t)\big)^{\frac{1}{2}}
+f(t)\nonumber\\
 \leq&    C\left(1+\Big\|\frac{\Lambda^{\frac{5}{2}}}{ g^{2}(\Lambda)}u \Big\|_{L^{2}}\right)g^{2}\big[\left(e+A(t)\right)^{\frac{1}{2a}}\big]\sqrt{\ln\big(e+A(t)\big)} \big(e+A(t)\big)\nonumber\\ &+\frac{1}{2}B(t)+C\big(e+A(t)\big)
+f(t).\nonumber
\end{align}
In particular,
\begin{align}\label{crkop8}
 \frac{d}{dt}A(t)\leq& C\left(1+\Big\|\frac{\Lambda^{\frac{5}{2}}}{ g^{2}(\Lambda)}u \Big\|_{L^{2}}\right)g^{2}\big[\left(e+A(t)\right)^{\frac{1}{2a}}\big]\sqrt{\ln\big(e+A(t)\big)} \big(e+A(t)\big)
\nonumber\\ &+Cf(t).
\end{align}
Thanks to
$$g^{2}\big[\left(e+A(t)\right)^{\frac{1}{2a}}\big]\sqrt{\ln\big(e+A(t)\big)} \big(e+A(t)\big)\geq1,$$
we divide \eqref{crkop8} by $g^{2}\big[\left(e+A(t)\right)^{\frac{1}{2a}}\big]\sqrt{\ln\big(e+A(t)\big)} \big(e+A(t)\big)$ and integrating in time to obtain
$$\int_{e+A(0)}^{e+A(t)}\frac{d\tau}{\tau\sqrt{\ln \tau} g^{2}(\tau^{\frac{1}{2a}})}\leq C
\int_{0}^{t}{
\left(1+f(\tau)+\Big\|\frac{\Lambda^{\frac{5}{2}}}{ g^{2}(\Lambda)}u (\tau) \Big\|_{L^{2}}\right)\,d\tau}.$$
Recalling the condition (\ref{logcobd}) on $g$,
$$\int_{e}^{\infty}\frac{d\tau}{\tau\sqrt{\ln \tau} g^{2}(\tau^{\frac{1}{2a}})}=\sqrt{2a}\int_{e^{\frac{1}{2a}}}^{\infty}\frac{d\tau}{\tau\sqrt{\ln \tau} g^{2}(\tau)}=\infty$$
and the bound in Lemma \ref{loglema2},
$$\int_{0}^{t}{
\left(1+f(\tau)+\Big\|\frac{\Lambda^{\frac{5}{2}}}{ g^{2}(\Lambda)}u (\tau) \Big\|_{L^{2}}\right)\,d\tau}\leq C(t,\,u_{0},\,w_{0}),$$
we deduce that
$$A(t)\leq C(t,\,u_{0},\,w_{0}).$$
Returning to \eqref{crkop8}, we also get
$$\int_{0}^{t}{B(\tau)\,d\tau}\leq C(t,\,u_{0},\,w_{0}).$$
Thus, we have
\begin{eqnarray*}
\|\Lambda^{k_{1}}u(t)\|_{L^{2}}^{2}
+\|\Lambda^{\widetilde{\sigma}}w(t)\|_{L^{2}}^{2}
+\int_{0}^{t}{(
 \|\Lambda^{k_{1}+\frac{7}{4}-\epsilon_{2}}u
(\tau)\|_{L^{2}}^{2}
 +\|\mathcal{L}\Lambda^{k_{1}}
u(\tau)\|_{L^{2}}^{2}) \,d\tau}\leq
C(t,\,u_{0},\,w_{0}),
\end{eqnarray*}
which is \eqref{cber75b9xc}. This finishes the proof of Lemma \ref{dvefvb21}.
\end{proof}

\vskip .1in
We are now ready to prove the global $H^s$-bound and thus Theorem \ref{addTh2}.

\begin{proof}[{Proof of Theorem \ref{addTh2}}]
To this end, we recall \eqref{tsdfwef},
\begin{align*} &
\frac{1}{2}\frac{d}{dt}(\|\Lambda^{s} u\|_{L^{2}}^{2}+\|\Lambda^{s}
w\|_{L^{2}}^{2})+\|\mathcal{L}\Lambda^{s}
  u\|_{L^{2}}^{2}+2\|\Lambda^{s}  w\|_{L^{2}}^{2}+\|\Lambda^{s}\nabla\cdot w\|_{L^{2}}^{2}\nonumber\\
&= \int_{\mathbb{R}^{3}}\Big(
\Lambda^{s}(\nabla\times u)\cdot\Lambda^{s} w+\Lambda^{s}
(\nabla\times w)\cdot \,\,\Lambda^{s} u\Big)\,dx
-\int_{\mathbb{R}^{3}}
[\Lambda^{s}, u \cdot \nabla] u \,\Lambda^{s} u\,dx\nonumber\\&\quad-\int_{\mathbb{R}^{3}}
[\Lambda^{s}, u \cdot \nabla]w\, \Lambda^{s} w\,dx\nonumber\\&:= J_{1}+J_{2}+J_{3}.
\end{align*}
According to the proof of \eqref{logtt002}, we have, for any $\epsilon_{3}\in (0,\,\frac{7}{4})$
\begin{eqnarray}\label{dcdfdwwsz34}
\|\mathcal{L}\Lambda^{s}
u\|_{L^{2}}^{2}\geq C_{5}\|\Lambda^{s+\frac{7}{4}-\epsilon_{3}}
u\|_{L^{2}}^{2}
-C_{6}\|\Lambda^{s} u\|_{L^{2}}^{2}.
\end{eqnarray}
By \eqref{dcdfdwwsz34} with $\epsilon_{3}=\frac{3}{4}$, it leads to
\begin{align}
J_{1} \leq&C\|\Lambda^{s+1}u\|_{L^{2}}
\|\Lambda^{s}w \|_{L^{2}}\nonumber\\
 \leq&C
(\|\mathcal{L}\Lambda^{s}
u\|_{L^{2}}+\|\Lambda^{s} u\|_{L^{2}})
\|\Lambda^{s}w \|_{L^{2}}
\nonumber\\
 \leq&
\frac{1}{16}\|\mathcal{L}\Lambda^{s}
  u\|_{L^{2}}^{2}+C(\|\Lambda^{s} u\|_{L^{2}}^{2}+\|\Lambda^{s}
w\|_{L^{2}}^{2}).\nonumber
\end{align}
According to \eqref{opmdwzv78}, one has
\begin{align}\label{tycvdf789}
J_{2}\leq
C\|\nabla u\|_{L^{\infty}}
\|\Lambda^{s}u\|_{L^{2}}^{2}.
\end{align}
The last term $J_{3}$ can be bounded by
\begin{align}\label{tyb3178ht4}
J_{3} =&-\int_{\mathbb{R}^{3}}
[\Lambda^{s}\partial_{x_{i}}, u_{i}
] w\Lambda^{s}w\,dx
\nonumber\\
 \leq&\|[\Lambda^{s}\partial_{x_{i}}, u_{i}] w \|_{L^{2}}\|\Lambda^{s}w\|_{L^{2}}
\nonumber\\
 \leq&C(\|\nabla u \|_{L^{\infty}}\|\Lambda^{s}w\|_{L^{2}}+
\|\Lambda^{s+1}u\|_{L^{2}}
\|w\|_{L^{\infty}})\|\Lambda^{s}w\|_{L^{2}}
\nonumber\\
 \leq&C\left(\|\nabla u \|_{L^{\infty}}\|\Lambda^{s}w\|_{L^{2}}+
(\|\mathcal{L}\Lambda^{s}
u\|_{L^{2}}+\|\Lambda^{s} u\|_{L^{2}})\|w\|_{L^{\infty}}\right)\|\Lambda^{s}w\|_{L^{2}}
\nonumber\\ \leq&
\frac{1}{16}\|\mathcal{L}\Lambda^{s}
  u\|_{L^{2}}^{2}+C
  (\|\nabla u \|_{L^{\infty}}+\|w\|_{L^{\infty}}^{2}) (\|\Lambda^{s}u\|_{L^{2}}^{2}+\|\Lambda^{s}w\|_{L^{2}}^{2}).
\end{align}
Combining all the above estimates, we obtain
\begin{align} \label{dfgdfgytr16} &
 \frac{d}{dt}(\|\Lambda^{s} u\|_{L^{2}}^{2}+\|\Lambda^{s}
w\|_{L^{2}}^{2})+\|\mathcal{L}\Lambda^{s}
u\|_{L^{2}}^{2}\nonumber\\&\leq C(1+\|w\|_{L^{\infty}}^{2}+\|\nabla u\|_{L^{\infty}})
(\|\Lambda^{s}u\|_{L^{2}}^{2}+\|\Lambda^{s}w\|_{L^{2}}^{2}).
\end{align}
Recalling \eqref{csdfnmy} and \eqref{cxdf78nmycv}, we conclude via the Gronwall type inequality
$$\|\Lambda^{s} u(t)\|_{L^{2}}^{2}+\|\Lambda^{s}
w(t)\|_{L^{2}}^{2}+\int_{0}^{t}{\|\mathcal{L}\Lambda^{s}
  u(\tau)\|_{L^{2}}^{2}\,d\tau}\leq C(t,\,u_{0},\,w_{0}).$$
This completes the proof of Theorem \ref{addTh2}.
\end{proof}

\vskip .3in
\section{The proof of Theorem \ref{addTh3}}
\label{sec4}
\setcounter{equation}{0}

This section is devoted to the proof of Theorem \ref{addTh3}.
Similarly it suffices to consider the case $\alpha+\beta=\frac{7}{4}$.
The basic $L^2$-estimate follows from Lemma \ref{loglema1}.

\begin{lemma}
Assume $(u_{0},w_{0})$ satisfies the assumptions stated in Theorem \ref{addTh3}.
Then the corresponding solution $(u, w)$
of the system \eqref{loggnDMP} admits the following bound for any $t>0$
\begin{eqnarray} \label{s4logtt001}
\|u(t)\|^2_{L^2}+\|w(t)\|_{L^2}^2
   + \int_{0}^{t}{(\|\mathcal{L}u(\tau)\|_{L^{2}}^{2}
   +\|\Lambda^{\beta}w(\tau)\|_{L^{2}}^{2})
\,d\tau} \leq
C(t,\,u_{0},\,w_{0}).
\end{eqnarray}
\end{lemma}

\vskip .1in
Our next goal is to improve the regularity of $u$, which can be stated as follows.
\begin{lemma}
Assume $(u_{0},w_{0})$ satisfies the assumptions stated in Theorem \ref{addTh3}.
Then the corresponding solution $(u, w)$
of the system \eqref{loggnDMP} admits the following bound for any $t>0$
\begin{eqnarray} \label{s4logtt002}
\left\|\frac{\Lambda^{\alpha+\beta-1}}{g(\Lambda)}u(t)\right\|^2_{L^2}+\int_{0}^{t}{
\left\|\frac{\Lambda^{2\alpha+\beta-1}}{g^{2}(\Lambda)}u(\tau)\right\|^2_{L^2}\,d\tau} \leq
C(t,\,u_{0},\,w_{0}).
\end{eqnarray}
When  $\alpha+\beta=\frac{7}{4}$,
\begin{eqnarray} \label{s4logtt003}
\left\|\frac{\Lambda^{\frac{3}{4}}}{g(\Lambda)}u(t)\right\|^2_{L^2}+\int_{0}^{t}{
\left\|\frac{\Lambda^{\alpha+\frac{3}{4}}}{g^{2}(\Lambda)}u(\tau)\right\|^2_{L^2}\,d\tau} \leq
C(t,\,u_{0},\,w_{0}).
\end{eqnarray}
In particular, we have for any $\epsilon_{1}\in(0,\,\alpha+\frac{3}{4})$,
\begin{eqnarray} \label{cvfdf0r}
\int_{0}^{t}{
\left\|\Lambda^{\alpha+\frac{3}{4}-\epsilon_{1}}u(\tau)\right\|^2_{L^2}\,d\tau} \leq
C(t,\,u_{0},\,w_{0}).
\end{eqnarray}
\end{lemma}

\begin{proof}
Taking the scalar product of the first equation of \eqref{loggnDMP} with  $\frac{\Lambda^{2\alpha+2\beta-2}}{g^{2}(\Lambda)}u$, we have
\begin{align}
 \frac{1}{2}\frac{d}{dt} \left\|\frac{\Lambda^{\alpha+\beta-1}}{g(\Lambda)}u(t)\right\|^2_{L^2}+
\left\|\frac{\Lambda^{2\alpha+\beta-1}}{g^{2}(\Lambda)}u\right\|^2_{L^2} =&
\int_{\mathbb{R}^{3}}{ \nabla\times w\cdot \frac{\Lambda^{2\alpha+2\beta-2}}{g^{2}(\Lambda)}u\,dx}
\nonumber\\&
-\int_{\mathbb{R}^{3}}{ (u \cdot \nabla u)\cdot \frac{\Lambda^{2\alpha+2\beta-2}}{g^{2}(\Lambda)}u\,dx}
.\nonumber
\end{align}
By the Young inequality,
\begin{align}
\left|\int_{\mathbb{R}^{3}}{ \nabla\times w\cdot \frac{\Lambda^{2\alpha+2\beta-2}}{g^{2}(\Lambda)}u\,dx}\right| \leq& C\|\Lambda^{\beta}w\|_{L^{2}}\left\|\frac{\Lambda^{2\alpha+\beta-1}}
{g^{2}(\Lambda)}u\right\|_{L^2}\nonumber\\ \leq&
\frac{1}{16}\left\|\frac{\Lambda^{2\alpha+\beta-1}}
{g^{2}(\Lambda)}u\right\|_{L^2}^{2}+C\|\Lambda^{\beta}w\|_{L^{2}}^{2}.\nonumber
\end{align}
Following the arguments used in proving \eqref{logtt002}, we deduce for any $r_{1}\in(0,\,\alpha+\beta-1)$ and $r_{2}\in(0,\,2\alpha+\beta-1)$ that
\begin{eqnarray} \label{s4logtt004}
\left\|\frac{\Lambda^{\alpha+\beta-1}}{g(\Lambda)}u\right\|^2_{L^2}\geq
C_{1}\|\Lambda^{r_{1}}u\|_{L^{2}}^2-C_{2}\|u\|_{L^{2}}^2,
\end{eqnarray}
\begin{eqnarray} \label{s4logtt005}
\left\|\frac{\Lambda^{2\alpha+\beta-1}}{g^{2}(\Lambda)}u\right\|^2_{L^2}\geq
C_{3}\|\Lambda^{r_{2}}u\|_{L^{2}}^2-C_{4}\|u\|_{L^{2}}^2.
\end{eqnarray}
By a simple embedding inequality,
\begin{align*}
\left|-\int_{\mathbb{R}^{3}}{ (u \cdot \nabla u)\cdot \frac{\Lambda^{2\alpha+2\beta-2}}{g^{2}(\Lambda)}u\,dx}\right| \leq& C\|\Lambda^{\beta}(uu)\|_{L^{2}}\left\|\frac{\Lambda^{2\alpha+\beta-1}}
{g^{2}(\Lambda)}u\right\|_{L^2}\nonumber\\
 \leq& C\|u\|_{L^{\frac{12}{3-2\beta}}}\|\Lambda^{\beta}u\|_{L^{\frac{12}{3+2\beta}}}\left\|\frac{\Lambda^{2\alpha+\beta-1}}
{g^{2}(\Lambda)}u\right\|_{L^2}\nonumber\\
 \leq& C\|\Lambda^{\beta}u\|_{L^{\frac{12}{3+2\beta}}}^{2}\left\|\frac{\Lambda^{2\alpha+\beta-1}}
{g^{2}(\Lambda)}u\right\|_{L^2}\nonumber\\
\leq& C\|\Lambda^{\frac{3+2\beta}{4}}u\|_{L^{2}}^{2}\left\|\frac{\Lambda^{2\alpha+\beta-1}}
{g^{2}(\Lambda)}u\right\|_{L^2}.
\end{align*}
The high-low frequency technique implies
$$\|\Lambda^{\frac{3+2\beta}{4}}u\|_{L^{2}}\leq \|S_{N}\Lambda^{\frac{3+2\beta}{4}}u\|_{L^{2}}
+\sum_{j=N}^{\infty}\|\Delta_{j}\Lambda^{\frac{3+2\beta}{4}}u\|_{L^{2}}.$$
Thanks to $\alpha\geq\frac{5}{4}$ and the Plancherel theorem, we have
\begin{align}
\|S_{N}\Lambda^{\frac{3+2\beta}{4}}u\|_{L^{2}} \leq& C\|S_{N}\Lambda^{\frac{3+4\alpha}{8}}u\|_{L^{2}}+C\|S_{N}u\|_{L^{2}}\nonumber\\
 \leq& C\|S_{N}\Lambda^{\frac{3}{4}}u\|_{L^{2}}^{\frac{1}{2}}
\|S_{N}\Lambda^{\alpha}u\|_{L^{2}}^{\frac{1}{2}}+C\|u\|_{L^{2}}\nonumber\\
 \leq& C\|\chi(2^{-N}\xi)|\xi|^{\frac{3}{4}}\widehat{u}(\xi)\|_{L^{2}}^{\frac{1}{2}}
\|\chi(2^{-N}\xi)|\xi|^{\alpha}\widehat{u}(\xi)\|_{L^{2}}^{\frac{1}{2}}
+C\|u\|_{L^{2}}
\nonumber\\
 =& C\left\|\chi(2^{-N}\xi)g
(\xi)\frac{|\xi|^{\frac{3}{4}}\widehat{u}(\xi)}{g
(\xi)}\right\|_{L^{2}}^{\frac{1}{2}}
\left\|\chi(2^{-N}\xi)g
(\xi)\frac{|\xi|^{\alpha}\widehat{u}(\xi)}{g
(\xi)}\right\|_{L^{2}}^{\frac{1}{2}}
+C\|u\|_{L^{2}}
\nonumber\\
 \leq& C g(2^{N})\left\|\frac{|\xi|^{\frac{3}{4}}\widehat{u}(\xi)}{g
(\xi)}\right\|_{L^{2}}^{\frac{1}{2}}
\left\|\frac{|\xi|^{\alpha}\widehat{u}(\xi)}{g
(\xi)}\right\|_{L^{2}}^{\frac{1}{2}}
+C\|u\|_{L^{2}}
\nonumber\\
 \leq& C g(2^{N})\left\|\frac{\Lambda^{\alpha+\beta-1}}
{g(\Lambda)}u\right\|_{L^2}^{\frac{1}{2}}
\|\mathcal{L}u\|_{L^{2}}^{\frac{1}{2}}
+C\|u\|_{L^{2}}.\nonumber
\end{align}
The high frequency can be bounded by the Bernstein inequality and \eqref{s4logtt004} along with \eqref{s4logtt005},
\begin{align}
\sum_{j=N}^{\infty}\|\Delta_{j}\Lambda^{\frac{3+2\beta}{4}}u\|_{L^{2}} \leq& C
\sum_{j=N}^{\infty}2^{-\frac{1}{4}j}\|\Delta_{j}\Lambda^{\frac{2+\beta}{2}}
u\|_{L^{2}}\nonumber\\
 \leq& C
\sum_{j=N}^{\infty}2^{-\frac{1}{4}j}\|\Lambda^{\frac{2+\beta}{2}}
u\|_{L^{2}}\nonumber\\
 \leq& C2^{-\frac{N}{4}}\|\Lambda^{r_{1}}
u\|_{L^{2}}^{1-\theta}\|\Lambda^{r_{2}}
u\|_{L^{2}}^{\theta}\nonumber\\
 \leq& C2^{-\frac{N}{4}}\left(\left\|\frac{\Lambda^{\alpha+\beta-1}}{g(\Lambda)}
u\right\|_{L^2}+\|u\|_{L^{2}}\right)^{1-\theta}
\left(\left\|\frac{\Lambda^{2\alpha+\beta-1}}{g^{2}(\Lambda)}
u\right\|_{L^2}+\|u\|_{L^{2}}\right)^{\theta},\nonumber
\end{align}
where $\theta$ is given by
$$\theta=\frac{2+\beta-2r_{1}}{2r_{2}-2r_{1}}<\frac{1}{2}.$$
Therefore,
\begin{align} \|\Lambda^{\frac{3+2\beta}{4}}u\|_{L^{2}} \leq& C g(2^{N})\left\|\frac{\Lambda^{\alpha+\beta-1}}{g(\Lambda)}u\right\|_{L^2}^{\frac{1}{2}}
\|\mathcal{L}u\|_{L^{2}}^{\frac{1}{2}}
+C\|u\|_{L^{2}}\nonumber\\& +C2^{-\frac{N}{4}}\left(\left\|\frac{\Lambda^{\alpha+\beta-1}}{g(\Lambda)}
u\right\|_{L^2}+\|u\|_{L^{2}}\right)^{1-\theta}
\left(\left\|\frac{\Lambda^{2\alpha+\beta-1}}{g^{2}(\Lambda)}
u\right\|_{L^2}+\|u\|_{L^{2}}\right)^{\theta}.\nonumber
\end{align}
Combining all the estimates altogether yields
\begin{align}
 &\frac{d}{dt} \left\|\frac{\Lambda^{\alpha+\beta-1}}{g(\Lambda)}u(t)\right\|^2_{L^2}+
\left\|\frac{\Lambda^{2\alpha+\beta-1}}{g^{2}(\Lambda)}u\right\|^2_{L^2}
\nonumber\\ \leq&C\|\Lambda^{\beta}w\|_{L^{2}}^{2}+C g^{2}(2^{N})\left\|\frac{\Lambda^{\alpha+\beta-1}}{g(\Lambda)}u\right\|_{L^2}
\|\mathcal{L}u\|_{L^{2}}\left\|\frac{\Lambda^{2\alpha+\beta-1}}{g^{2}(\Lambda)}
u\right\|_{L^2}
\nonumber\\&\quad+C\|u\|_{L^{2}}^{2}\left\|\frac{\Lambda^{2\alpha+\beta-1}}{g^{2}(\Lambda)}
u\right\|_{L^2}+C2^{-\frac{N}{2}}
\left(\left\|\frac{\Lambda^{\alpha+\beta-1}}{g(\Lambda)}
u\right\|_{L^2}+\|u\|_{L^{2}}\right)^{2(1-\theta)}
\nonumber\\& \quad\times\left(\left\|\frac{\Lambda^{2\alpha+\beta-1}}{g^{2}(\Lambda)}
u\right\|_{L^2}+\|u\|_{L^{2}}\right)^{2\theta}
\left\|\frac{\Lambda^{2\alpha+\beta-1}}{g^{2}(\Lambda)}
u\right\|_{L^2}.\nonumber
\end{align}
Therefore, if we set
$$A(t):=\left\|\frac{\Lambda^{\alpha+\beta-1}}{g(\Lambda)}u(t)\right\|^2_{L^2},\qquad
B(t):=\left\|\frac{\Lambda^{2\alpha+\beta-1}}{g^{2}(\Lambda)}u(t)
\right\|^2_{L^2},$$
then
\begin{align}
\frac{d}{dt}A(t)+B(t)
 \leq&C\|\Lambda^{\beta}w\|_{L^{2}}^{2}+C g^{2}(2^{N})A^{\frac{1}{2}}(t)
\|\mathcal{L}u\|_{L^{2}}B^{\frac{1}{2}}(t)
\nonumber\\&+C\|u\|_{L^{2}}^{2}B^{\frac{1}{2}}(t)
+C2^{-\frac{N}{2}}\left(A(t)+\|u\|_{L^{2}}^{2}\right)^{1-\theta}
\left(B(t)+\|u\|_{L^{2}}^{2}\right)^{\theta}B^{\frac{1}{2}}(t)\nonumber\\
 \leq&\frac{1}{2}B(t)+C g^{4}(2^{N})A(t)
\|\mathcal{L}u\|_{L^{2}}^{2}
+CA(t)
+C2^{-\frac{N}{1-2\theta}}\left(A(t)\right)^{\frac{2(1-\theta)}{1-2\theta}}
\nonumber\\&+C2^{-N}\|u\|_{L^{2}}^{4\theta}A^{2-2\theta}(t)+\phi(\|u\|_{L^{2}}),
\nonumber
\end{align}
where $\phi(t)$ is given by
$$\phi(\|u(t)\|_{L^{2}})=C(\|u(t)\|_{L^{2}}^{4}+\|u(t)\|_{L^{2}}
^{\frac{4(1-\theta)}{1-2\theta}})\leq C(t,\,u_{0},\,w_{0}).$$
This yields
$$
\frac{d}{dt}A(t)+B(t)
\leq C g^{4}(2^{N})A(t)
\|\mathcal{L}u\|_{L^{2}}^{2}
+CA(t)
+C2^{-\frac{N}{1-2\theta}}\left(A(t)\right)^{\frac{2(1-\theta)}{1-2\theta}}+
\phi(\|u(t)\|_{L^{2}}).
$$
By taking
$$2^{N}\approx e+A(t),$$
we  obtain
\begin{eqnarray}\label{s4logtt007}
\frac{d}{dt}A(t)+B(t)
\leq C g^{4}\left(e+A(t)\right)\left(e+A(t)\right)
(1+\|\mathcal{L}u\|_{L^{2}}^{2})+\phi(\|u(t)\|_{L^{2}}).
\end{eqnarray}
Noticing that
$$g^{4}\big(e+A(t)\big)\big(e+A(t)\big)\geq1,$$
we obtain from \eqref{s4logtt007} that
$$\int_{e+A(0)}^{e+A(t)}\frac{d\tau}{\tau  g^{4}(\tau)}\leq C
\int_{0}^{t}{
\left(1+\phi(\|u(\tau)\|_{L^{2}})+ \|\mathcal{L}u (\tau) \|_{L^{2}}^{2}\right)\,d\tau}.$$
Thanks to
$$\int_{e}^{\infty}\frac{d\tau}{\tau g^{4}(\tau)}=\infty$$
and the following fact due to \eqref{s4logtt001}
$$\int_{0}^{t}{
\left(1+\phi(\|u(\tau)\|_{L^{2}})+\|\mathcal{L}u (\tau) \|_{L^{2}}^{2}\right)\,d\tau}\leq C(t,\,u_{0},\,w_{0}),$$
we deduce that
$$A(t)\leq C(t,\,u_{0},\,w_{0}).$$
In addition,  \eqref{s4logtt007} implies
$$\int_{0}^{t}{B(\tau)\,d\tau}\leq C(t,\,u_{0},\,w_{0}).$$
The desired estimates \eqref{s4logtt002} and \eqref{s4logtt003} follow directly, and \eqref{s4logtt005} immediately implies \eqref{cvfdf0r}.
We thus complete the proof of this lemma.
\end{proof}

\vskip .1in
Motivated by Lemma \ref{AZL302}, we will show the following key lemma.
\begin{lemma}\label{s4AZL302}
Assume $(u_{0},w_{0})$ satisfies the assumptions stated in Theorem \ref{addTh3}.
If $\alpha\geq\frac{5}{4}$, $\beta>0$ and $\alpha+\beta\geq \frac{7}{4}$, then the corresponding solution $(u, w)$
of the system \eqref{loggnDMP} admits the following bounds for any $t>0$,
\begin{gather}
\|\Lambda^{\frac{7}{4}}u(t)\|_{L^{2}}^{2}
+\int_{0}^{t}{\|\Lambda^{\frac{7}{4}}\mathcal{L}u(\tau)\|_{L^{2}}^{2}\,d\tau}\leq
C(t,\,u_{0},\,w_{0}),\label{s4logtt009}\\
\|\Lambda^{\varrho}w(t)\|_{L^{2}}^{2}
+\int_{0}^{t}{\|\Lambda^{\varrho+\beta}w(\tau)\|_{L^{2}}^{2}\,d\tau}\leq
C(t,\,u_{0},\,w_{0}), \label{s4logtt0010}
\end{gather}
for any $\varrho<1+\beta$.
In particular, due to $\alpha\geq\frac{5}{4}$, we have
\begin{eqnarray}\label{s4logtt0011}
\int_{0}^{t}{(\|\nabla u(\tau)\|_{L^{\infty}}+\|\Lambda^{\frac{5}{2}} u(\tau)\|_{L^{2}})\,d\tau}\leq
C(t,\,u_{0},\,w_{0}).
\end{eqnarray}
\end{lemma}
\begin{proof}
Applying $\Lambda^{\frac{7}{4}}$ to the first equation of \eqref{loggnDMP} and dotting by $\Lambda^{\frac{7}{4}}u$, we obtain
\begin{align*}
& \frac{1}{2}\frac{d}{dt}\|\Lambda^{\frac{7}{4}}u(t)\|_{L^{2}}^{2}
+\|\Lambda^{\frac{7}{4}}\mathcal{L}
u\|_{L^{2}}^{2}\nonumber\\
&=  -\int_{\mathbb{R}^{3}}{[\Lambda^{\frac{7}{4}},
u\cdot\nabla]u\cdot \Lambda^{\frac{7}{4}}u\,dx}
+\int_{\mathbb{R}^{3}}{\Lambda^{\frac{7}{4}}(\nabla\times w)
\Lambda^{\frac{7}{4}}u\,dx}\nonumber\\
&:=  \widetilde{N}_{1}+\widetilde{N}_{2}.
\end{align*}
Applying $\Lambda^{\varrho}$ to the second equation of \eqref{loggnDMP} and taking the inner product with $\Lambda^{\varrho}w$ yield
\begin{align*} &
\frac{1}{2}\frac{d}{dt}\|\Lambda^{\varrho}w(t)\|_{L^{2}}^{2}
+\|\Lambda^{\varrho+\beta}w \|_{L^{2}}^{2}+2\|\Lambda^{\varrho}w\|^2_{L^2}+\|\Lambda^{\varrho}\nabla\cdot w\|_{L^{2}}^{2}\nonumber\\&= \int_{\mathbb{R}^{3}}
\Lambda^{\varrho}\big(\nabla\times u\big)\Lambda^{\varrho}w\,dx -\int_{\mathbb{R}^{3}}
[\Lambda^{\varrho}, u \cdot
\nabla] w\Lambda^{\varrho}w\,dx\nonumber\\
&:=  \widetilde{N}_{3}+\widetilde{N}_{4}.
\end{align*}
Similar to the proof of \eqref{logtt002}, we obtain for any $\epsilon_{2}\in(0,\,\alpha)$ that
\begin{eqnarray} \label{s4logtt0014}
\|\Lambda^{\frac{7}{4}}\mathcal{L}u\|_{L^{2}}^{2}\geq
C_{3}\|\Lambda^{\alpha+\frac{7}{4}-\epsilon_{2}}u\|_{L^{2}}^{2}
-C_{4}\|\Lambda^{\frac{7}{4}}u\|_{L^{2}}^2.
\end{eqnarray}
By the Sobolev embedding inequality and the Kato-Ponce inequality,
\begin{align}
\widetilde{N}_{1}
 \leq& \|[\Lambda^{\frac{7}{4}},
u\cdot\nabla]u\|_{L^{2}}\|\Lambda^{\frac{7}{4}}u\|_{L^{2}}
 \nonumber\\
 \leq&C \|\nabla u\|_{L^{k}}
 \|\Lambda^{\frac{7}{4}}
 u\|_{L^{\frac{2k}{k-2}}}\|\Lambda^{\frac{7}{4}}u\|_{L^{2}}
  \nonumber\\
 \leq&C\left(\left\|\frac{\Lambda^{2\alpha+\beta-1}}{g^{2}
 (\Lambda)}u\right\|_{L^2}+ \|\nabla u\|_{L^{2}}\right)
 (\|\Lambda^{\frac{7}{4}}\mathcal{L}
u\|_{L^{2}}+\|\Lambda^{\frac{7}{4}}u\|_{L^{2}})
 \|\Lambda^{\frac{7}{4}}u\|_{L^{2}}
  \nonumber\\
 \leq&
 \frac{1}{16}\|\Lambda^{\frac{7}{4}}\mathcal{L}
u\|_{L^{2}}^{2}+
 C\left(1+\left\|\frac{\Lambda^{2\alpha+\beta-1}}{g^{2}(\Lambda)}u\right\|_{L^2}^{2}+ \|\nabla u\|_{L^{2}}^{2}\right)(1+
 \|\Lambda^{\frac{7}{4}}u\|_{L^{2}}^{2}),\nonumber
\end{align}
where $k>2$ satisfies
$$\frac{7-4\alpha-2\beta}{6}\leq\frac{1}{k}\leq \frac{\alpha}{3}.$$
By \eqref{s4logtt0014} and the interpolation inequality, we can bound $\widetilde{N}_{2}$ by
\begin{align}
\widetilde{N}_{2}
\leq& C\|\Lambda^{\varrho+\beta}w \|_{L^{2}}\|\Lambda^{\frac{7}{4}}u\|_{H^{\frac{11}{4}-\varrho-\beta}}\nonumber\\
 \leq&C\|\Lambda^{\varrho+\beta}w \|_{L^{2}}\|\Lambda^{\frac{7}{4}}u\|_{L^{2}}^{1-\vartheta_{1}}
 \|\Lambda^{\frac{7}{4}}\mathcal{L}u\|_{L^{2}}^{\vartheta_{1}}
  \nonumber\\
 \leq& \frac{1}{16}\|\Lambda^{\frac{7}{4}}\mathcal{L}u\|_{L^{2}}^{2}
  +\frac{1}{16}\|\Lambda^{\varrho+\beta}w \|_{L^{2}}^{2} +C\|\Lambda^{\frac{7}{4}}u\|_{L^{2}}^{2},\nonumber
\end{align}
where $\varrho>1$. Similarly, one has
\begin{align}
\widetilde{N}_{3}
 \leq&\|\Lambda^{\varrho-\beta}\big(\nabla\times u\big)\|_{L^{2}}\|\Lambda^{\varrho+\beta}w\|_{L^{2}}\nonumber\\
 \leq&C\|u\|_{H^{\varrho+1-\beta}}
 \|\Lambda^{\varrho+\beta}w\|_{L^{2}}
 \nonumber\\
 \leq&
  C\|u\|_{L^{2}}^{1-\vartheta_{2}}
 \|\Lambda^{\frac{7}{4}}\mathcal{L}u\|_{L^{2}}^{\vartheta_{2}}
\|\Lambda^{\varrho+\beta}w\|_{L^{2}}
   \nonumber\\
\leq&
  \frac{1}{16}\|\Lambda^{\frac{7}{4}}\mathcal{L}u\|_{L^{2}}^{2}
  +\frac{1}{16}\|\Lambda^{\varrho+\beta}w \|_{L^{2}}^{2}+C\|u\|_{L^{2}}^{2},\nonumber
\end{align}
where $\varrho<\frac{5}{2}$.
As in \eqref{sdsfew09}, we choose $m_{0}>2$ satisfying
$$\frac{7-4\alpha-2\beta}{6}\leq \frac{1}{m_{0}}\leq\frac{\beta}{3}$$
to obtain,  for $\alpha+\beta\geq \frac{7}{4}$ and $\varrho<1+\beta$,
\begin{align}
\widetilde{N}_{4}
 \leq& C\|[\Lambda^{\varrho}, u\cdot\nabla]w\|_{B_{2,2}^{-\beta}}\|\Lambda^{\varrho}w\|_{B_{2,2}^{\beta}}\nonumber\\
 \leq& C(\|\nabla
u\|_{L^{m_{0}}}\|w\|_{B_{\frac{2m_{0}}{m_{0}-2},2}^{\varrho-\beta}}
+\|u\|_{L^{2}}\|w\|_{L^{2}})
(\|\Lambda^{\varrho}w\|_{L^{2}}+\|\Lambda^{\varrho+\beta}w\|_{L^{2}}) \nonumber\\
 \leq&
C(\|u\|_{H^{2\alpha+\beta-1}}
\|w\|_{B_{2,2}^{\varrho-\beta+\frac{3}{m_{0}}}}+\|u\|_{L^{2}}\|w\|_{L^{2}})
(\|\Lambda^{\varrho}w\|_{L^{2}}+\|\Lambda^{\varrho+\beta}w\|_{L^{2}})
\nonumber\\
 \leq&
C(\|u\|_{H^{2\alpha+\beta-1}}
\|w\|_{H^{\varrho}}+\|u\|_{L^{2}}\|w\|_{L^{2}})
(\|\Lambda^{\varrho}w\|_{L^{2}}+\|\Lambda^{\varrho+\beta}w\|_{L^{2}}).\nonumber
\end{align}
By the Plancherel theorem and the Bernstein inequality,
\begin{align} \|u\|_{H^{2\alpha+\beta-1}} \approx& \|u\|_{L^{2}}+\|\Lambda^{2\alpha+\beta-1}u\|_{L^{2}}\nonumber\\
 \leq&\|u\|_{L^{2}}+\|S_{N}\Lambda^{2\alpha+\beta-1}u\|_{L^{2}}+
\sum_{j=N}^{\infty}\|\Delta_{j}\Lambda^{2\alpha+\beta-1}u\|_{L^{2}}
\nonumber\\
 \leq&\|u\|_{L^{2}}+C\left\|\chi(2^{-N}\xi)g^{2}(\xi)
\frac{|\xi|^{2\alpha+\beta-1}\widehat{u}(\xi)}{g^{2}(\xi)}\right\|_{L^{2}}
\nonumber\\& +C
\sum_{j=N}^{\infty}2^{-\sigma j}\|\Delta_{j}\Lambda^{\alpha+\frac{3}{4}
+\sigma}u\|_{L^{2}}\nonumber\\
 \leq&\|u\|_{L^{2}}+Cg^{2}(2^{N})\left\|
\frac{|\xi|^{2\alpha+\beta-1}\widehat{u}(\xi)}{g^{2}(\xi)}\right\|_{L^{2}}+C
2^{-\sigma N}\|\Lambda^{\alpha+\frac{3}{4}
+\sigma}u\|_{L^{2}}\nonumber\\
 \leq&\|u\|_{L^{2}}+Cg^{2}(2^{N})\left\|
\frac{\Lambda^{2\alpha+\beta-1}}{g^{2}(\Lambda)}u\right\|_{L^{2}}\nonumber\\& +C
2^{-\sigma N}\|\Lambda^{\alpha+\frac{3}{4}
-\epsilon_{1}}u\|_{L^{2}}^{1-2\sigma}\|\Lambda^{\frac{7}{4}+\alpha
-\epsilon_{2}}u\|_{L^{2}}^{2\sigma}\nonumber\\
 \leq&\|u\|_{L^{2}}+Cg^{2}(2^{N})\left\|
\frac{\Lambda^{2\alpha+\beta-1}}{g^{2}(\Lambda)}u\right\|_{L^{2}}\nonumber\\&+C
2^{-\sigma N}\|\Lambda^{\alpha+\frac{3}{4}
-\epsilon_{1}}u\|_{L^{2}}^{1-2\sigma}(\|\Lambda^{\frac{7}{4}}\mathcal{L}u(\tau)\|_{L^{2}}+
\|\Lambda^{\frac{7}{4}}u\|_{L^{2}})^{2\sigma},\nonumber
\end{align}
where $\sigma$ is given by
$$\sigma=\frac{\epsilon_{1}}{1+2\epsilon_{1}-2\epsilon_{2}}<\frac{1}{2}.$$
Combining all the estimates above, one has
 \begin{align}
 &  \frac{d}{dt}(\|\Lambda^{\frac{7}{4}}u(t)\|_{L^{2}}^{2}
+\|\Lambda^{\varrho}w(t)\|_{L^{2}}^{2})
+\|\Lambda^{\frac{7}{4}}\mathcal{L}
u\|_{L^{2}}^{2}+\|\Lambda^{\varrho+\beta}w \|_{L^{2}}^{2}\nonumber\\&\leq
C\left(1+\left\|\frac{\Lambda^{2\alpha+\beta-1}}{g^{2}(\Lambda)}u\right\|_{L^2}^{2}+ \|\nabla u\|_{L^{2}}^{2}\right)(1+
 \|\Lambda^{\frac{7}{4}}u\|_{L^{2}}^{2})\nonumber\\&\quad
 +C\left(g^{2}(2^{N})\left\|
\frac{\Lambda^{2\alpha+\beta-1}}{g^{2}(\Lambda)}u\right\|_{L^{2}}
\|w\|_{H^{\varrho}}+\|u\|_{L^{2}}\|w\|_{L^{2}}\right)
(\|\Lambda^{\varrho}w\|_{L^{2}}+\|\Lambda^{\varrho+\beta}w\|_{L^{2}})
\nonumber\\&\quad
 +C\left(2^{-\sigma N}\|\Lambda^{\alpha+\frac{3}{4}
-\epsilon_{1}}u\|_{L^{2}}^{1-2\sigma}(\|\Lambda^{\frac{7}{4}}\mathcal{L}u(\tau)\|_{L^{2}}+
\|\Lambda^{\frac{7}{4}}u\|_{L^{2}})^{2\sigma}
\|w\|_{H^{\varrho}}+\|u\|_{L^{2}}\|w\|_{L^{2}}\right)\nonumber\\&\quad\times
(\|\Lambda^{\varrho}w\|_{L^{2}}+\|\Lambda^{\varrho+\beta}w\|_{L^{2}}).\nonumber
 \end{align}
Writing
$$X(t):=\|\Lambda^{\frac{7}{4}}u(t)\|_{L^{2}}^{2}
+\|\Lambda^{\varrho}w(t)\|_{L^{2}}^{2},\qquad Y(t):=\|\Lambda^{\frac{7}{4}}\mathcal{L}
u(t)\|_{L^{2}}^{2}+\|\Lambda^{\varrho+\beta}w(t)\|_{L^{2}}^{2},$$
and choosing
$$2^{N}\approx e+X(t),$$
 we obtain
\begin{align}  \frac{d}{dt}X(t)
+Y(t)  \leq&
C\left(1+\left\|\frac{\Lambda^{2\alpha+\beta-1}}{g^{2}(\Lambda)}u\right\|_{L^2}^{2}+ \|\nabla u\|_{L^{2}}^{2}\right)\big(e+X(t)\big)\nonumber\\& \quad
 +C\left(g^{2}\big(e+X(t)\big)\left\|
\frac{\Lambda^{2\alpha+\beta-1}}{g^{2}(\Lambda)}u\right\|_{L^{2}}
\big(\|w\|_{L^{2}}+X^{\frac{1}{2}}(t)\big)+\|u\|_{L^{2}}\|w\|_{L^{2}}\right)
\nonumber\\
&\quad\times
\big(X^{\frac{1}{2}}(t)+Y^{\frac{1}{2}}(t)\big)
\nonumber\\&\quad
 +C\big(e+X(t)\big)^{-\sigma}\|\Lambda^{\alpha+\frac{3}{4}
-\epsilon_{1}}u\|_{L^{2}}^{1-2\sigma}(X^{\frac{1}{2}}(t)+Y^{\frac{1}{2}}(t))^{2\sigma}
\big(\|w\|_{L^{2}}+X^{\frac{1}{2}}(t)\big)\nonumber\\&\quad\times
\big(X^{\frac{1}{2}}(t)+Y^{\frac{1}{2}}(t)\big)+C\|u\|_{L^{2}}\|w\|_{L^{2}}
\big(X^{\frac{1}{2}}(t)+Y^{\frac{1}{2}}(t)\big).\nonumber
 \end{align}
Thanks to the H$\rm\ddot{o}$lder inequality,
\begin{align}  \frac{d}{dt}X(t)
+Y(t) \leq&
C\left(H(\|u\|_{L^{2}},\,\|w\|_{L^{2}})+\left\|\frac{\Lambda^{2\alpha+\beta-1}}
{g^{2}(\Lambda)}u\right\|_{L^2}^{2}+ \|\Lambda^{\alpha+\frac{3}{4}
-\epsilon_{1}}u\|_{L^{2}}^{2}\right)\big(e+X(t)\big)\nonumber\\&\quad
 +C\left(H(\|u\|_{L^{2}},\,\|w\|_{L^{2}})+\left\|
\frac{\Lambda^{2\alpha+\beta-1}}{g^{2}(\Lambda)}u\right\|_{L^{2}}^{2}\right)
\big(e+X(t)\big)g^{4}\big(e+X(t)\big)
\nonumber\\&\quad+C H(\|u\|_{L^{2}},\,\|w\|_{L^{2}})+\frac{1}{2}Y(t),\nonumber
 \end{align}
where $H(b_{1},\,b_{2})$ is an increasing smooth function with respective to both $b_{1}$ and $b_{2}$, and thus satisfies
$$\int_{0}^{t}H(\|u(\tau)\|_{L^{2}},\,\|w(\tau)\|_{L^{2}})\,d\tau\leq C(t,\,u_{0},\,w_{0}).$$
Noticing the following fact
$$\big(e+X(t)\big)g^{4}\big(e+X(t)\big)\geq1,$$
we finally get
\begin{eqnarray}  \label{s4logdf15yu}\frac{d}{dt}X(t)
+Y(t)\leq CR(t)\big(e+X(t)\big)g^{4}\big(e+X(t)\big),
 \end{eqnarray}
where $R(t)$ is given by
$$R(t):=H(\|u(t)\|_{L^{2}},\,\|w(t)\|_{L^{2}})
+\left\|\frac{\Lambda^{2\alpha+\beta-1}}{g^{2}(\Lambda)}u(t)\right\|_{L^2}^{2}+ \|\Lambda^{\alpha+\frac{3}{4}
-\epsilon_{1}}u(t)\|_{L^{2}}^{2}.$$
Recalling the assumption on $g$,
$$\int_{e}^{\infty}\frac{d\tau}{\tau g^{4}(\tau)}=\infty$$
and the bound due to \eqref{s4logtt002} and \eqref{cvfdf0r}
$$\int_{0}^{t}{R(\tau)\,d\tau}\leq C(t,\,u_{0},\,w_{0}),$$
we obtain
$$\int_{e
+X(0)}^{e
+X(t)}\frac{d\tau}{\tau g^{4}(\tau)}\leq \int_{0}^{t}{R(\tau)\,d\tau}\leq C(t,\,u_{0},\,w_{0}).$$
That is,  $X(t)$ must be finite for any given $t>0$,
\begin{eqnarray}  \label{s4logdf16yu}
X(t)\leq C(t,\,u_{0},\,w_{0}).
\end{eqnarray}
By \eqref{s4logdf15yu},
\begin{eqnarray}  \label{s4logdf17yu}
\int_{0}^{t}{Y(\tau)\,d\tau}\leq C(t,\,u_{0},\,w_{0}).
\end{eqnarray}
The estimates \eqref{s4logdf16yu} and \eqref{s4logdf17yu} along with \eqref{s4logtt0014} imply \eqref{s4logtt009} and \eqref{s4logtt0010}. Clearly \eqref{s4logtt0011} is  an easy consequence of \eqref{s4logtt009}. In fact, due to
$\alpha\geq \frac{5}{4}$, we invoke the inequality with  $0<\epsilon_{2}<\frac{5}{4}+\alpha$,
 $$\|\Lambda^{\frac{5}{2}} u\|_{L^{2}},\,\|\nabla u\|_{L^{\infty}}\leq C\|u\|_{L^{2}}
 +C\|\Lambda^{\alpha+\frac{7}{4}-\epsilon_{2}}u\|_{L^{2}},$$
which is  \eqref{s4logtt0011}.
This completes the proof of Lemma \ref{s4AZL302}.
\end{proof}

\vskip .1in
The estimate \eqref{s4logtt0011} and Lemma \ref{AZL3dsf8} allow us to obtain the following bound.
\begin{lemma}
Assume $(u_{0},w_{0})$ satisfies the assumptions stated in Theorem \ref{addTh3}.
If $\alpha\geq \frac{5}{4}$, $\beta>0$ and $\alpha+\beta\geq \frac{7}{4}$, then the corresponding solution $(u, w)$
of the system \eqref{loggnDMP} admits the following bound for any $t>0$,
\begin{eqnarray}\label{s4sfdsg65df}
\int_{0}^{t}{\|w(\tau)\|_{L^{\infty}}^{2}\,d\tau}\leq
C(t,\,u_{0},\,w_{0}).
\end{eqnarray}
\end{lemma}

\vskip .1in
Finally we provide the global $H^s$-estimate for $u$ and $w$, and thus finish the proof of Theorem \ref{addTh3}.

\begin{proof}[{Proof of Theorem \ref{addTh3}}]
As in \eqref{tsdfwef}, we have
\begin{align}&
\frac{1}{2}\frac{d}{dt}(\|\Lambda^{s} u\|_{L^{2}}^{2}+\|\Lambda^{s}
w\|_{L^{2}}^{2})+ \|\Lambda^{s}\mathcal{L}
  u\|_{L^{2}}^{2}+\|\Lambda^{s+\beta}  w\|_{L^{2}}^{2}+2\|\Lambda^{s}  w\|_{L^{2}}^{2}+\|\Lambda^{s}\nabla\cdot w\|_{L^{2}}^{2}\nonumber\\ &:=J_{1}+J_{2}+J_{3}.\nonumber
\end{align}
By means of the proof of \eqref{logtt002}, we have, for any $0<\epsilon_{3}<\alpha$,
\begin{eqnarray} \label{sxcf68po}
\|\Lambda^{s}\mathcal{L}u\|_{L^{2}}^{2}\geq
C_{5}\|\Lambda^{\alpha+s-\epsilon_{2}}u\|_{L^{2}}^{2}
-C_{6}\|\Lambda^{s}u\|_{L^{2}}^2.
\end{eqnarray}
By \eqref{sxcf68po},
\begin{align}
J_{1} \leq& 2\|\Lambda^{s+1} u\|_{L^{2}}\|\Lambda^{s} w\|_{L^{2}}\nonumber\\
 \leq& 2(\|\Lambda^{s} u\|_{L^{2}}+\|\Lambda^{s} \mathcal{L}u\|_{L^{2}})\|\Lambda^{s} w\|_{L^{2}}
\nonumber\\
 \leq&
\frac{1}{16}\|\Lambda^{s} \mathcal{L}u\|_{L^{2}}^{2}+C(\|\Lambda^{s}  u\|_{L^{2}}^{2}+\|\Lambda^{s} w\|_{L^{2}}^{2}).\nonumber
\end{align}
We deduce from \eqref{tycvdf789} that
\begin{eqnarray}
J_{2}\leq C\|\nabla u \|_{L^{\infty}}\|\Lambda^{s} u\|_{L^{2}}^{2}.\nonumber
\end{eqnarray}
According to \eqref{tyb3178ht4}, we have
\begin{eqnarray}
J_{3}\leq\frac{1}{16}\|\Lambda^{s} \mathcal{L}u\|_{L^{2}}^{2}+C(\|\nabla u \|_{L^{\infty}}+\| w \|_{L^{\infty}}^{2})
(\|\Lambda^{s}u\|_{L^{2}}^{2}+\|\Lambda^{s}w\|_{L^{2}}^{2}).\nonumber
\end{eqnarray}
Combining all the  estimates yields
\begin{align} &
 \frac{d}{dt}(\|\Lambda^{s} u\|_{L^{2}}^{2}+\|\Lambda^{s}
w\|_{L^{2}}^{2})+\|\Lambda^{s}\mathcal{L}
  u\|_{L^{2}}^{2}+\|\Lambda^{s+\beta}  w\|_{L^{2}}^{2}\nonumber\\&\leq C(1+\|\nabla u \|_{L^{\infty}}+\| w \|_{L^{\infty}}^{2})
(\|\Lambda^{s}u\|_{L^{2}}^{2}+\|\Lambda^{s}w\|_{L^{2}}^{2}).\nonumber
\end{align}
Thanks to \eqref{s4logtt0011} and \eqref{s4sfdsg65df}, one has
$$\int_{0}^{t}{(1+\|\nabla u(\tau)\|_{L^{\infty}}+\| w(\tau)\|_{L^{\infty}}^{2})\,d\tau}\leq
C(t,\,u_{0},\,w_{0}).$$
The Gronwall inequality implies
$$\|\Lambda^{s} u(t)\|_{L^{2}}^{2}+\|\Lambda^{s}
w(t)\|_{L^{2}}^{2}+\int_{0}^{t}{(\|\Lambda^{s}\mathcal{L}
  u(\tau)\|_{L^{2}}^{2}+\|\Lambda^{s+\beta}  w(\tau)\|_{L^{2}}^{2})\,d\tau}\leq C(t,\,u_{0},\,w_{0}).$$
This completes the proof of Theorem \ref{addTh3}.
\end{proof}

\vskip .2in
\appendix
\section{Besov spaces} \label{apset1}

This appendix provides the definition of the Besov spaces and related facts that have been used in the previous sections. Some of the materials are taken from \cite{BCD}.

\vskip .1in
We start with the partition of unity. Let $B(0, r)$ and ${\mathcal C}(0, r_1, r_2)$ denote the standard ball and the annulus, respectively,
$$
B(0, r) = \left\{\xi \in \mathbb R^n: \, |\xi| \le r \right\}, \qquad
\mathcal C (0, r_1, r_2) = \left\{\xi \in \mathbb R^n:\, r_1 \le|\xi|\le r_2 \right\}.
$$
There are two compactly supported smooth radial functions $\phi$ and $\psi$ satisfying
\ben
&& \mbox{supp} \,\phi \subset B(0, 4/3), \quad \mbox{supp} \,\psi \subset \mathcal C(0, 3/4, 8/3),  \notag\\
&& \phi(\xi) + \sum_{j\ge 0}  \psi(2^{-j} \xi) = 1 \qquad \mbox{for all}\,\, \xi \in \mathbb R^n. \label{aa}
\een
We use $\widetilde h$ and $h$ to denote the inverse Fourier transforms of $\phi$ and
$\psi$ respectively,
$$
\widetilde h = \mathcal F^{-1} \phi, \quad  h = \mathcal F^{-1} \psi.
$$
In addition, for notational convenience, we write $\psi_j(\xi) = \psi(2^{-j} \xi)$. By a simple property of the Fourier transform,
$$
h_j(x) :=\mathcal F^{-1} (\psi_j)(x) = 2^{n j} \, h(2^j x).
$$
The inhomogeneous dyadic block operator $\Delta_j$ are defined as follows
\beno
&& \Delta_j f=0 \qquad \mbox{for $j\le -2$},\\
&& \Delta_{-1} f = \widetilde h \ast f = \int_{\mathbb R^n} f(x-y) \, \widetilde h(y)\,dy,\\
&& \Delta_j f = h_j \ast f = 2^{n j} \int_{\mathbb R^n} f(x-y) \,  h(2^j y)\,dy \qquad \mbox{for $j\ge 0$}.
\eeno
The corresponding inhomogeneous low frequency cut-off operator $S_j$ is defined by
$$
S_j f = \sum_{k\le j-1} \Delta_k f.
$$
For any function $f$ in the usual Schwarz class $\mathcal{S}$, (\ref{aa}) implies
\ben
\widehat f (\xi) = \phi(\xi)\, \widehat f (\xi) + \sum_{j\ge 0}  \psi(2^{-j} \xi)\,\widehat f (\xi),  \label{bb}
\een
or, in terms of the inhomogeneous dyadic block operators,
$$
f= \sum_{j\ge -1} \Delta_j f \quad \mbox{or}\quad \mbox{Id} = \sum_{j\ge -1} \Delta_j,
$$
where Id denotes the identity operator. More generally, for any $F$ in the space of tempered distributions, denoted  ${\mathcal S}'$, (\ref{bb}) still holds but in the distributional sense. That is, for $F \in {\mathcal S}'$,
\ben
F = \sum_{j\ge -1} \Delta_j F \quad \mbox{or}\quad \mbox{Id} = \sum_{j\ge -1} \Delta_j \qquad \mbox{in} \quad  {\mathcal S}'.  \label{bb1}
\een
In fact, one can verify that
$$
S_j F := \sum_{k\le j-1} \Delta_k F  \quad \to \quad F \qquad \mbox{in} \quad  {\mathcal S}'.
$$
Note that (\ref{bb1}) is referred to as the Littlewood-Paley decomposition for tempered distributions.

\vskip .1in
The inhomogeneous Besov space can be defined in terms of  $\Delta_j$ specified above.
\begin{define}
	For $1\le p,q \le \infty$ and $s\in {\mathbb R}$, the inhomogeneous Besov space $B^s_{p,q}$   consists of  the functions $f\in {\mathcal S}'$
	satisfying
	$
	\|f\|_{B^s_{p,q}} \equiv \|2^{js} \|\Delta_j f\|_{L^p} \|_{l^q}
	<\infty.
	$
\end{define}

\vskip .1in
Bernstein's inequality is a useful tool on Fourier localized functions and these inequalities trade derivatives for integrability. The following proposition provides Bernstein type inequalities for fractional derivatives.

\begin{lemma}\label{vfgty8xc}
	For $\alpha\ge0$, $1\le p\le q\le \infty$, and $f\in L^p(\mathbb{R}^n)$,
	
	(1)  if there exist  some integer $j$ and a constant $K>0$, such that,
	$\mbox{supp}\, \widehat{f} \subset \{\xi\in \mathbb{R}^n: \,\, |\xi| \le K 2^j \}$,
	  then
		$$
		\|(-\Delta)^\alpha f\|_{L^q(\mathbb{R}^n)} \le C_1\, 2^{2\alpha j +
			j n(\frac{1}{p}-\frac{1}{q})} \|f\|_{L^p(\mathbb{R}^n)};
		$$
	
	(2) if there exist some integer $j$ and constants $0<K_1\le K_2$, such that,
	$\mbox{supp}\, \widehat{f} \subset \{\xi\in \mathbb{R}^n: \,\, K_12^j\le |\xi| \le K_2 2^j \}$,
	 then
		$$
		C_1\, 2^{2\alpha j} \|f\|_{L^q(\mathbb{R}^n)} \le \|(-\Delta)^\alpha
		f\|_{L^q(\mathbb{R}^n)} \le C_2\, 2^{2\alpha j +
			j n(\frac{1}{p}-\frac{1}{q})} \|f\|_{L^p(\mathbb{R}^n)},
		$$
		where $C_1$ and $C_2$ are constants depending only on $\alpha,p$ and $q$.
		
 \end{lemma}

\vskip .3in
\section{A global regularity result when $\nabla \nabla\cdot w$ is eliminated}
\label{apset2}
As we mentioned in the introduction, the term $\nabla \nabla\cdot w$ in the equation of $w$ in the micropolar system is a ``bad" term in the sense that it prevents us from deriving the estimate $\|w\|_{L^{q}}$ with $q>2$ directly by the standard  $L^{q}$-estimate.  This appendix provides a global regularity result
for the micropolar system without this term, namely (\ref{3DMP}) with $\mu=0$. As we shall see in Theorem \ref{SIPTh4}, the requirement on the fractional powers can be reduced to $\alpha\ge \frac54$ and $\beta=0$, which is the best one at this moment.

\begin{thm}\label{SIPTh4}
	Consider the following 3D incompressible micropolar
	equations, namely,
	\begin{equation}\label{simpd8}
	\left\{\begin{array}{l}
	\partial_t u+(u\cdot\nabla) u
	+(-\Delta)^{\alpha} u +\nabla p =  \nabla \times w ,\vspace{2mm} \qquad x\in \mathbb{R}^{3},\,t>0,\\
	\partial_tw + (u\cdot\nabla) w +2 w
	= \nabla\times u,
	\vspace{2mm}\\
	\nabla\cdot u=0,\vspace{2mm}\\
	u(x,0)=u_{0}(x),\quad w(x,0)=w_{0}(x).
	\end{array}\right.
	\end{equation}
	Let $(u_{0},\,w_{0})\in H^{s}(\mathbb{R}^{3})$ with $s>\frac{5}{2}$ and $\nabla\cdot u_{0}=0$. If $\alpha\geq
	\frac{5}{4}$, then the system \eqref{simpd8} admits a unique global solution $(u,\,w)$ such that for any given $T>0$,
	$$(u,\,w)\in L^{\infty}([0, T]; H^{s}(\mathbb{R}^{3})),\quad \Lambda^{\alpha}u\in L^{2}([0, T]; H^{s}(\mathbb{R}^{3})).
	$$
\end{thm}

\vskip .2in
It suffices to consider the endpoint case $\alpha=\frac{5}{4}$ since $\alpha>\frac54$ is even simpler. Combining Lemma \ref{L301} and Lemma \ref{addt01}, we still have
\begin{eqnarray}\label{5yettr01}
\|u(t)\|_{H^{\frac{1}{4}}}^{2}
+\int_{0}^{t}{\|u(\tau)\|_{H^{\frac{3}{2}}}^{2}\,d\tau}\leq
C(t,\,u_{0},\,w_{0}).
\end{eqnarray}

\vskip .1in
With \eqref{5yettr01} at our disposal, we are in the position to establish the following key estimates.
\begin{lemma}\label{lemma123}
	Assume $(u_{0},w_{0})$ satisfies the assumptions stated in Theorem \ref{SIPTh4}, then the smooth solution $(u, w)$ of \eqref{simpd8} admits the following bounds
	\begin{align}\label{5yettr02}
	\int_{0}^{t}{
		\|\nabla u(\tau)\|_{L^{\infty}}\,d\tau}\leq C(t,\,u_{0},\,w_{0}),
	\end{align}
	\begin{align}\label{5yettr03}
	\|w(t)\|_{L^{\infty}} \leq C(t,\,u_{0},\,w_{0}).
	\end{align}
\end{lemma}

\begin{proof}
	By $\nabla\cdot u=0$, we rewrite $(\ref{simpd8})_{1}$ as follows
	\begin{eqnarray}
	\partial_{t}u+\Lambda^{\frac{5}{2}} u=-\left(\mathbb{I}_{3}+(-\Delta)^{-1}\nabla\nabla\cdot\right)\Big[\nabla\cdot(u\otimes u)-\nabla\times w\Big],\nonumber
	\end{eqnarray}
	where we have eliminated the pressure term by $\na\cdot u=0$. Applying $\Lambda^{-1}$ yields
	\begin{eqnarray}\label{5yettr04}
	\partial_{t}\Lambda^{-1}u+\Lambda^{\frac{5}{2}} \Lambda^{-1} u=-\Lambda^{-1}\left(\mathbb{I}_{3}+(-\Delta)^{-1}\nabla\nabla\cdot\right)\Big[\nabla\cdot(u\otimes u)-\nabla\times w\Big].
	\end{eqnarray}
	Applying Lemma 3.1 of \cite{Ye18jmaa} to \eqref{5yettr04} leads to
	\begin{align} \label{5yettr05}
	\|\Lambda^{\frac{3}{2}-\epsilon}u\|_{L_{t}^{1}L^{8}}
	=&\|\Lambda^{\frac{5}{2}-\epsilon}\Lambda^{-1}u\|_{L_{t}^{1}L^{8}} \nonumber\\ \leq & C(t,u_{0})+ C(t)\left\|\Lambda^{-1}\left(\mathbb{I}+(-\Delta)^{-1}\nabla\nabla\cdot\right)\Big[\nabla\cdot(u\otimes u)-\nabla\times w\Big]\right\|_{L_{t}^{1}L^{8}}\nonumber\\
	\leq & C(t,u_{0}) + C(t)\left\|\Lambda^{-1} \Big[\nabla\cdot(u\otimes u)-\nabla\times w\Big]\right\|_{L_{t}^{1}L^{8}}\nonumber\\
	\leq & C(t,u_{0}) + C(t)\|uu\|_{L_{t}^{1}L^{8}}
	+ C(t)\|w\|_{L_{t}^{1}L^{8}}\nonumber\\
	\leq & C(t,u_{0}) + C(t)\|u\|_{L_{t}^{2}L^{16}}^{2}
	+ C(t)\|w\|_{L_{t}^{1}L^{8}}
	\nonumber\\
	\leq & C(t,u_{0}) + C(t)\|u\|_{L_{t}^{2}H^{\frac{3}{2}}}^{2}
	+ C(t)\|w\|_{L_{t}^{1}L^{8}}
	\nonumber\\
	\leq & C(t,\,u_{0},\,w_{0}) + C(t)\|w\|_{L_{t}^{1}L^{8}},
	\end{align}
	where in the last line we have used \eqref{5yettr01}. By the equation of $w$
	in (\ref{simpd8}),
	\begin{align*}
	\frac{d}{dt}\|w(t)\|_{L^{8}}\leq \|\nabla u\|_{L^{8}}.
	\end{align*}
	By an interpolation inequality, one derives
	\begin{align}
	\|w(t)\|_{L^{8}}&\leq \|w_{0}\|_{L^{8}}+\int_{0}^{t}\|\nabla u(\tau)\|_{L^{8}}\,d\tau\nonumber\\
	&\leq \|w_{0}\|_{L^{8}}+C\int_{0}^{t}\|u(\tau)\|_{L^{2}}^{1-\frac{17}{21-8\epsilon}}
	\|\Lambda^{\frac{3}{2}-\epsilon}u(\tau)\|_{L^{8}}^{\frac{17}{21-8\epsilon}}\,d\tau\nonumber\\
	&\leq \|w_{0}\|_{L^{8}}
	+C\left(\int_{0}^{t}\|u(\tau)\|_{L^{2}}\,d\tau\right)^{1-\frac{17}{21-8\epsilon}}
	\left(\int_{0}^{t}
	\|\Lambda^{\frac{3}{2}-\epsilon}u(\tau)\|_{L^{8}}
	\,d\tau\right)^{\frac{17}{21-8\epsilon}},\nonumber
	\end{align}
	where $0<\epsilon<\frac{1}{2}$. Therefore, we conclude
	\begin{align}\label{5yettr06}
	\|w\|_{L_{t}^{1}L^{8}}&\leq  t\|w_{0}\|_{L^{8}}
	+Ct\left(\int_{0}^{t}\|u(\tau)\|_{L^{2}}\,d\tau\right)^{1-\frac{17}{21-8\epsilon}}
	\left(\int_{0}^{t}
	\|\Lambda^{\frac{3}{2}-\epsilon}u(\tau)\|_{L^{8}}
	\,d\tau\right)^{\frac{17}{21-8\epsilon}}
	\nonumber\\&\leq
	t\|w_{0}\|_{L^{8}}+C(t,\,u_{0},\,w_{0})
	\|\Lambda^{\frac{3}{2}-\epsilon}u\|_{L_{t}^{1}L^{8}}^{\frac{17}{21-8\epsilon}}.
	\end{align}
	Combining \eqref{5yettr05} and \eqref{5yettr06}, we have
	\begin{align}
	\|\Lambda^{\frac{3}{2}-\epsilon}u\|_{L_{t}^{1}L^{8}}&\leq C(t,\,u_{0},\,w_{0})+C(t,\,u_{0},\,w_{0})
	\|\Lambda^{\frac{3}{2}-\epsilon}u\|_{L_{t}^{1}L^{8}}^{\frac{17}{21-8\epsilon}}\nonumber\\
	&\leq C(t,\,u_{0},\,w_{0})+\frac{1}{2}\|\Lambda^{\frac{3}{2}-\epsilon}u\|_{L_{t}^{1}L^{8}},
	\nonumber
	\end{align}
	which yields
	\begin{align} \label{5yettr07}
	\|\Lambda^{\frac{3}{2}-\epsilon}u\|_{L_{t}^{1}L^{8}} \leq C(t,\,u_{0},\,w_{0}).
	\end{align}
	By further taking $0<\epsilon<\frac{1}{8}$, we obtain from \eqref{5yettr07} that
	\begin{align} \label{5yettr08}
	\int_{0}^{t}\|\nabla u(\tau)\|_{L^{\infty}}\,d\tau \leq& C\int_{0}^{t}\| u(\tau)\|_{L^{2}}\,d\tau+C\int_{0}^{t}
	\|\Lambda^{\frac{3}{2}-\epsilon}u(\tau)\|_{L^{8}}\,d\tau\nonumber\\
	\leq& C(t,\,u_{0},\,w_{0}).
	\end{align}
	By the equation of $w$ in (\ref{simpd8}), we again have, for any $2\le q<\infty$,
	\begin{align}
	\frac{d}{dt}\|w(t)\|_{L^{q}}\leq \|\nabla u\|_{L^{q}} \quad \mbox{or}\quad
	\|w(t)\|_{L^{q}} \le \|w_0\|_{L^{q}} + \int_0^t \|\nabla u\|_{L^{q}}\, d\tau.
	\nonumber
	\end{align}
	Letting $q\rightarrow\infty$ and invoking (\ref{5yettr08}), we find
	$$\|w(t)\|_{L^{\infty}} \leq C(t,\,u_{0},\,w_{0}).$$
	Thus, we complete the proof of Lemma \ref{lemma123}.
\end{proof}

\vskip .1in
By \eqref{5yettr02} and \eqref{5yettr03}, we can obtain our ultimate global $H^s$-estimate for $u$ and $w$.
\begin{proof}[{Proof of Theorem \ref{SIPTh4}}]
	Similar to \eqref{dfgdfgytr16}, we have
	\begin{align*}  &
	\frac{d}{dt}(\|\Lambda^{s} u\|_{L^{2}}^{2}+\|\Lambda^{s}
	w\|_{L^{2}}^{2})+\|\Lambda^{s+\frac{5}{4}}
	u\|_{L^{2}}^{2}\nonumber\\&\leq C(1+\|w\|_{L^{\infty}}^{2}+\|\nabla u\|_{L^{\infty}})
	(\|\Lambda^{s}u\|_{L^{2}}^{2}+\|\Lambda^{s}w\|_{L^{2}}^{2}),
	\end{align*}
	which along with the Gronwall inequality, \eqref{5yettr02} and \eqref{5yettr03} yield
	$$\|\Lambda^{s} u(t)\|_{L^{2}}^{2}+\|\Lambda^{s}
	w(t)\|_{L^{2}}^{2}+\int_{0}^{t}{ \|\Lambda^{s+\frac{5}{4}}
		u(\tau)\|_{L^{2}}^{2} \,d\tau}\leq C(t,\,u_{0},\,w_{0}).$$
	This finish the proof of Theorem \ref{SIPTh4}.
\end{proof}

\vskip .3in
\section{Local well-posedness result on (\ref{3DMP})}
\label{apset3}
For the sake of completeness,  we present in this appendix the local
well-posedness result of (\ref{3DMP}) with initial
data $(u_{0}, w_{0}) \in H^{s}(\mathbb{R}^{3})$ with $s>\frac{5}{2}$.
\begin{Pros}\label{P1}
Let $(u_{0}, w_{0}) \in H^{s}(\mathbb{R}^{3})$ with $s>\frac{5}{2}$ and $\nabla\cdot u_{0}=0$. If $\alpha+\beta>1$, then there exists a
positive time $T$ depending on $\|u_{0}\|_{H^{s}}$ and
$\|w_{0}\|_{H^{s}}$ such that
(\ref{3DMP}) admits a unique solution $(u, w)\in C([0, T];
H^{s}(\mathbb{R}^{3}))$.
\end{Pros}


We remark that the same local well-posedness result also holds true for \eqref{logorga11}.
Similarly to \cite{MB2011,ConstantinF} (also see \cite{Dwye2019}), the main ingredient of the proof of
the  Proposition \ref{P1}  is to approximate  (\ref{3DMP}) by the Friedrichs method   to obtain a family of global
smooth solutions.

For $N>0$,  set $B(0,N)=\{\xi \in \mathbb{R}^{3}|\,|\xi|\leq N\}$ and denote by
$\chi_{B(0,N)}$  the characteristic function on $B(0,N)$.
Define the functional space
$$L^{2}_{N}:=\{f\in L^{2}(\mathbb{R}^{3})|\, \mbox{ supp }\,\widehat{f}\subset B(0,N)\},$$
and  the spectral cut-off
$$\widehat{\mathcal {J}_{N}f}(\xi)=\chi_{B(0,N)}(\xi)\widehat{f}(\xi).$$

\begin{proof}[{\textbf{Proof of Proposition \ref{P1}}}]
We first consider the following approximate system
of (\ref{3DMP}),
\begin{equation}\label{AMHD}
\left\{\aligned
&\partial_{t}u^{N}+\mathcal {P}\mathcal
{J}_{N}((\mathcal {J}_{N}u^{N}\cdot\nabla) \mathcal {J}_{N}u^{N})+ \Lambda^{2\alpha}\mathcal {J}_{N}u^{N} =\mathcal {P}\mathcal\nabla \times {J}_{N}w^{N}, \\
&\partial_{t}w^{N}+\mathcal {J}_{N} ((\mathcal {J}_{N}u^{N} \cdot \nabla)\mathcal {J}_{N}w^{N})+2\mathcal {J}_{N}w^{N}+\Lambda^{2\beta}\mathcal {J}_{N}w^{N}=\nabla\times  {J}_{N}u^{N}+  \nabla \nabla\cdot  {J}_{N}w^{N},\\
&\nabla\cdot u^{N}=0,\\
&u^{N}(x, 0)=\mathcal {J}_{N}u_{0}(x),  \quad w^{N}(x,0)=\mathcal {J}_{N}w_{0}(x) ,
\endaligned\right.
\end{equation}
where $\mathcal {P}$ denotes the standard projection onto divergence-free
vector fields. Thanks to the Cauchy-Lipschitz theorem (Picard's Theorem, see \cite{MB2011}), we can find that for
any fixed $N$, there exists a unique local solution $(u^{N},w^{N})$ on $[0,\,T_{N})$ in  the functional setting $L^{2}_{N}$ with $T_{N}=T(N, u_{0}, w_{0})$. By $\mathcal {J}_{N}^{2}=\mathcal {J}_{N},\,\mathcal
{P}^{2}=\mathcal {P}$ and $\mathcal {P}\mathcal {J}_{N}=\mathcal
{J}_{N}\mathcal {P}$, we can check that $(\mathcal {J}_{N}u^{N},\,\mathcal
 {J}_{N}w^{N})$ is also a solution to
(\ref{AMHD}) with the same initial datum. Based on the uniqueness, it yields
$$\mathcal {J}_{N}u^{N}=u^{N},\ \ \ \mathcal {J}_{N}w^{N}=w^{N}.$$
Consequently, the approximate system (\ref{AMHD}) reduces to
\begin{equation}\label{AMHD1}
\left\{\aligned
&\partial_{t}u^{N}+\mathcal {P}\mathcal
{J}_{N}((u^{N}\cdot\nabla)u^{N})+ \Lambda^{2\alpha} u^{N} =\mathcal {P}\mathcal\nabla \times  w^{N}, \\
&\partial_{t}w^{N}+\mathcal {J}_{N} ((u^{N} \cdot \nabla)w^{N})+2w^{N}+\Lambda^{2\beta}w^{N}=\nabla\times u^{N}+  \nabla \nabla\cdot  w^{N},\\
&\nabla\cdot u^{N}=0,\\
&u^{N}(x, 0)=\mathcal {J}_{N}u_{0}(x),  \quad w^{N}(x,0)=\mathcal {J}_{N}w_{0}(x).
\endaligned\right.
\end{equation}
A basic energy estimate implies $(u^{N},w^{N})$ of \eqref{AMHD1} satisfies
\begin{align}
&\|u^{N}(t)\|_{L^{2}}^{2}+\|w^{N}(t)\|_{L^{2}}^{2}+  \int_{0}^{t}{
(\|\Lambda^{\alpha}
u^{N}\|_{L^{2}}^{2}+\|w^{N}\|_{H^{\beta}}^{2})(\tau)\,d\tau} \leq C(\|u_{0}\|_{L^{2}}^{2}+\|w_{0}\|_{L^{2}}^{2},t).\nonumber
\end{align}
As a result, the local solution can be extended into a global
one, via the classical Picard Extension Theorem (see, e.g., \cite{MB2011}).
By the direct $H^s$-estimates (see for example \eqref{dfgdfgytr16}), we deduce from \eqref{AMHD1} that
\begin{align}\label{A12}
& \frac{d}{dt}(\|u^{N}(t)\|_{H^{s}}^{2}
+\|w^{N}(t)\|_{H^{s}}^{2})+\|\Lambda^{\alpha}u^{N}\|_{H^{s}}^{2}
+\|\Lambda^{\beta}w^{N}\|_{H^{s}}^{2}\nonumber\\
&\leq C(1+\|\nabla u^{N}\|_{L^{\infty}}+\|\nabla w^{N}\|_{L^{\infty}})
(\| u^{N}\|_{H^{s}}^{2}+\|w^{N}\|_{H^{s}}^{2})\nonumber\\
&\leq C(1+\|u^{N}\|_{H^{s}}+\|w^{N}\|_{H^{s}}) (\| u^{N}\|_{H^{s}}^{2}+\|w^{N}\|_{H^{s}}^{2}),
\end{align}
where we use the fact that
$$\|\nabla f\|_{L^{\infty}(\mathbb{R}^{3})}\leq C\|f\|_{H^{s}(\mathbb{R}^{3})},\quad s>\frac{5}{2}.$$
We assume in \eqref{A12} that
$\|u^{N}\|_{H^{s}}+\|w^{N}\|_{H^{s}}\geq1$ since, otherwise, we replace $\|u^{N}\|_{H^{s}}+\|w^{N}\|_{H^{s}}$ by $1+\|u^{N}\|_{H^{s}}+\|w^{N}\|_{H^{s}}$.
Denoting
$$X(t):={\| u^{N}(t)\|_{H^{s}}^{2}+\|w^{N}(t)\|_{H^{s}}^{2}},$$
we get from (\ref{A12}) that
$$\frac{d}{dt}X(t) \leq \kappa X(t)^{\frac{3}{2}},$$
where $\kappa>0$ is an absolute constant.
By direct calculations, we show that for all $N$
$$
\sup_{0\leq t\leq T}( {\| u^{N}(t)\|_{H^{s}}^{2}+\|w^{N}(t)\|_{H^{s}}^{2}})
 \leq
\frac{ {4\|u_{0}\|_{H^{s}}^{2}+4\|w_{0}\|_{H^{s}}^{2}}}
{\big(2-\kappa T \sqrt{{\|u_{0}\|_{H^{s}}^{2}
+\|w_{0}\|_{H^{s}}^{2} }}\big)^{2}},
$$
where $T>0$ satisfies
$$T<\frac{2}{\kappa\sqrt{\|u_{0}\|_{{H}^{s}}^{2}+\|w_{0}\|_{H^{s}}^{2}}}.$$
As a result, the family $(u^{N},w^{N})$ is uniformly bounded in $C([0, T]; H^{s})$ with $s>\frac{5}{2}$.
We can also show that
$$\partial_{t}u^{N},\,\,\partial_{t}w^{N}\in L_{t}^{\infty}
([0, T]);\,H_{x}^{-\vartheta}(\mathbb{R}^{3})\quad \mbox{for some} \,\,
\vartheta\geq 2.$$
As the embedding $L^{2}\hookrightarrow H^{-\vartheta}$ is locally compact, by the Aubin-Lions argument
we conclude that a subsequence
$(u^{N},w^{N})_{N\in\mathbb{N}}$ satisfies, on any compact subset of $\mathbb R^3$,
$$\|u^{N}-u^{N'}\|_{L^{2}}\rightarrow0,\quad \|w^{N}-w^{N'}\|_{L^{2}}\rightarrow0,\quad as\quad N,\,\,N'\rightarrow\infty.$$
Noticing that $\|f\|_{H^{s'}}\leq C
\|f\|_{L^{2}}^{1-\frac{s'}{s}}\|f\|_{H^{s}}^{\frac{s'}{s}}$ for $s>s'$, we have
$$\|u^{N}-u^{N'}\|_{H^{s'}}\rightarrow0,\quad \|w^{N}-w^{N'}\|_{H^{s'}} \rightarrow0,\quad
as\quad N,\,\,N'\rightarrow\infty.$$
Then we have strong convergence limit $(u, w)\in C([0, T];
H^{s'}(\mathbb{R}^{3}))$ for any $s'<s$;
hence, up to extraction, the sequence $(u^{N},w^{N})_{N\in\mathbb{N}}$ has a limit
$(u,\,w)$ satisfying
\begin{equation}\label{AMHD11}
\left\{\begin{array}{l}
\partial_t u+\mathcal {P}(u\cdot\nabla) u
       +\Lambda^{2\alpha} u =  \mathcal {P}\nabla \times w,\vspace{2mm}\\
\partial_tw + (u\cdot\nabla) w +2 w+\Lambda^{2\beta}w
             = \nabla\times u+  \nabla \nabla\cdot w,
             \vspace{2mm}\\
\nabla\cdot u=0,\vspace{2mm}\\
u(x,0)=u_{0}(x),\quad w(x,0)=w_{0}(x).
\end{array}\right.
\end{equation}
Furthermore, it is not hard to check that $(u, w)\in L^{\infty}([0, T];
H^{s}(\mathbb{R}^{3}))$. Finally, we claim that
$(u, w)\in C([0, T]; H^{s}(\mathbb{R}^{3}))$. It
suffices to consider $u\in C([0, T]; H^{s}(\mathbb{R}^{3})$ since the
same procedure can be applied to $w$ to obtain the desired result.
First, one has
 \begin{align}
 \sup_{0\leq t\leq T}(\|u\|_{H^{s}}+\|w\|_{H^{s}})\leq C(T)<\infty.\nonumber
 \end{align}
By the equivalent norm, we get
 \begin{align}\label{t2.03}
 \|u(t_{1})-u(t_{2})\|_{H^{s}}=\Big\{(\sum_{k<N}+ \sum_{k\geq N})
 (2^{ks}\|\Delta_{k}u(t_{1})-\Delta_{k}u(t_{2})\|_{L^{2}})^{2}
 \Big\}^{\frac{1}{2}}.
 \end{align}
Let $\varepsilon>0$ be arbitrarily small. Thanks to $u\in
L^{\infty}([0, T]; H^{s}(\mathbb{R}^{3}))$, there exists an integer
$N=N(\varepsilon)>0$ such that
 \begin{align}\label{t2.04}
 \Big\{\sum_{k\geq N}
 (2^{ks}\|\Delta_{k}u(t_{1})-\Delta_{k}u(t_{2})\|_{L^{2}})^{2}
 \Big\}^{\frac{1}{2}}<\frac{\varepsilon}{2}.
 \end{align}
Appealing to $(\ref{AMHD11})_{1}$ implies
\begin{align}
\Delta_{k}u(t_{1})-\Delta_{k}u(t_{2}) =&\int_{t_{1}}^{t_{2}}{\frac{d}{d\tau}
\Delta_{k}u(\tau)\,d\tau}\nonumber\\
 =&-\int_{t_{1}}^{t_{2}}{ \Delta_{k}\mathcal {P}[\nabla \times w+(u\cdot\nabla) u+ \Lambda^{2\alpha} u](\tau)\,d\tau}.\nonumber
\end{align}
This allows us to derive
\begin{align}
& \sum_{k<N}
 2^{2ks}\|\Delta_{k}u(t_{1})-\Delta_{k}u(t_{2})\|_{L^{2}}^{2}\nonumber\\
 &= \sum_{k<N}
 2^{2ks}\Big(\Big\|\int_{t_{1}}^{t_{2}}{ \Delta_{k}\mathcal {P}[\nabla \times w+(u\cdot\nabla) u+\Lambda^{2\alpha} u](\tau)\,d\tau}\Big\|_{L^{2}}\Big)^{2}\nonumber\\
&\leq \sum_{k<N}
 2^{2ks}\Big(\int_{t_{1}}^{t_{2}}{ \|\Delta_{k}[\nabla \times w+(u\cdot\nabla) u+\Lambda^{2\alpha} u]\|_{L^{2}}(\tau)\,d\tau}\Big)^{2}
\nonumber\\
&\leq \sum_{k<N}
 2^{2ks}\Big(\int_{t_{1}}^{t_{2}}{ [\|\Delta_{k}\nabla \times w
\|_{L^{2}}+ \|\|\Delta_{k}(u\cdot\nabla u)\|_{L^{2}}+
\|\|\Delta_{k}\Lambda^{2\alpha}u\|_{L^{2}}](\tau)\,d\tau}\Big)^{2} \nonumber\\
&= \sum_{k<N}
 2^{2k}\Big(\int_{t_{1}}^{t_{2}}{ 2^{k(s-1)}\|\Delta_{k}\nabla \times w
(\tau)\|_{L^{2}}\,d\tau}\Big)^{2}
\nonumber\\
&\quad+ \sum_{k<N}
 2^{2k}\Big(\int_{t_{1}}^{t_{2}}{  2^{k(s-1)}\|\Delta_{k}\nabla\cdot(u\otimes u)(\tau)\|_{L^{2}} \,d\tau}\Big)^{2}
\nonumber\\
&\quad+
\sum_{k<N}
 2^{4\alpha k}\Big(\int_{t_{1}}^{t_{2}}{
2^{ks}\|\Delta_{k}u(\tau)\|_{L^{2}}\,d\tau}\Big)^{2}
\nonumber\\
&\leq C\sum_{k<N}
 2^{2k}\Big(\|w\|_{L_{t}^{\infty}H^{s}}^{2}|t_{1}-t_{2}|^{2}
 +\|uu\|_{L_{t}^{\infty}H^{s}}^{2}|t_{1}-t_{2}|^{2} \Big)\nonumber\\&
 \quad +C\sum_{k<N}
 2^{4\alpha k} |t_{1}-t_{2}|^{2}
 \|u\|_{L_{t}^{\infty}H^{s}}^{2} \nonumber\\
&\leq C\sum_{k<N}
 2^{2k}|t_{1}-t_{2}|^{2}\Big(\|w\|_{L_{t}^{\infty}H^{s}}^{2}+
\|u\|_{L_{t}^{\infty}L^{\infty}}^{2}\| u\|_{L_{t}^{\infty}H^{s}}^{2}\Big)\nonumber\\&
 \quad +C\sum_{k<N}
 2^{4\alpha k} |t_{1}-t_{2}|^{2}
 \|u\|_{L_{t}^{\infty}H^{s}}^{2} \nonumber\\
&\leq  C
 2^{2N}|t_{1}-t_{2}|^{2}\Big(\|w\|_{L_{t}^{\infty}H^{s}}^{2}+\|u\|_{L_{t}^{\infty}
 H^{s}}^{4} \Big)+C2^{4\alpha N} |t_{1}-t_{2}|^{2}
 \|u\|_{L_{t}^{\infty}H^{s}}^{2},\nonumber
\end{align}
which implies
 \begin{align}\label{t2.07}
 \Big\{\sum_{k<N}
 (2^{ks}\|\Delta_{k}u(t_{1})-\Delta_{k}u(t_{2})\|_{L^{2}})^{2}
 \Big\}^{\frac{1}{2}}<\frac{\varepsilon}{2}
 \end{align}
provided that $|t_{1}-t_{2}|$ is small enough. The desired $u\in C([0, T]; H^{s}(\mathbb{R}^{3})$ follows from (\ref{t2.03}), (\ref{t2.04}) and
(\ref{t2.07}). Since $(u, w)$ are all in Lipschitz space, the uniqueness follows directly (see the end of Section \ref{sec2}). This completes the proof of Proposition \ref{P1}.
\end{proof}

\vskip .2in
\section*{Acknowledgements}
The research  of D. Wang is partially supported by the National Science Foundation under grants DMS-1613213 and DMS-1907519.
The research  of J. Wu is partially supported by the National Science Foundation under grant DMS 1624146 and the AT\&T Foundation at Oklahoma State University.
Z. Ye is supported by the National Natural Science Foundation of China (No. 11701232), the Natural Science Foundation of Jiangsu Province (No. BK20170224) and the Qing Lan Project of Jiangsu Province.
Part of this work was done when Z. Ye visited the Department of Mathematics, University of Pittsburgh; and he appreciates the hospitality of Prof. Dehua Wang and Prof. Ming Chen.
The authors would like to thank the anonymous referee for profitable suggestions and valuable comments.


\vskip .3in

\begin{thebibliography}{00} \frenchspacing

\bibitem{Applebaum}
D. Applebaum, {\em L$\rm\acute{e}$vy Processes and Stochastic Calculus}, vol. 116 (Cambridge University
Press, Cambridge/New York, 2009), pp. xxx+460.

\bibitem{BCD}
H. Bahouri, J.-Y. Chemin, R. Danchin, {\em Fourier Analysis and Nonlinear Partial Differential Equations},
in: Grundlehren der Mathematischen Wissenschaften, vol. 343, Springer, Heidelberg, 2011.

\bibitem{Nicki}
	N. Boardman, R. Ji, H. Qiu, J. Wu, \emph{Global existence and uniqueness of weak solutions to the Boussinesq equations without thermal diffusion},  Commun. Math. Sci.  17 No.6, (2019), 1595--1624.

\bibitem{Barbatomr}
D. Barbato, F. Morandin, M. Romito, \emph{Global regularity for a slightly supercritical hyperdissipative Navier-Stokes system}, Anal. PDE 7(8) (2014), 2009--2027.

\bibitem{Bertoinl}
J. Bertoin, {\em L$\rm\acute{e}$vy Processes}, Volume 121 of Cambridge Tracts in Mathematics (Cambridge University Press, Cambridge, 1996)

\bibitem{Boldrinidr}
J. Boldrini, M. Dur$\rm\acute{a}$n, M. Rojas-Medar, \emph{Existence and uniqueness of strong solution for the incompressible micropolar fluid equations in domains of $\mathbb{R}^{3}$}, Ann. Univ. Ferrara, Sez. VII, Sci. Mat. 56 (2010), 37--51.


\bibitem{BValdinoci}
C. Bucur, E. Valdinoci, {\em Nonlocal Diffusion and Applications}, Lecture Notes of the Unione Matematica
Italiana, vol. 20, Springer, International Publishing, ISBN 978-3-319-28738-6, 2016, xii, 155 pp.


\bibitem{chenmm}
Q. Chen, C. Miao, \emph{Global well-posedness for the micropolar
fluid system in critical
Besov spaces}, J. Differential Equations 252 (2012), 2698--2724.

\bibitem{Clavind}
P. Clavin, \emph{Instabilities and nonlinear patterns of overdriven detonations in gases}, in: H. Berestycki, Y. Pomeau (Eds.), Nonlinear PDEs in Condensed Matter and Reactive
Flows, Kluwer, 2002, pp. 49--97.

\bibitem{ConstantinF}
P. Constantin, C. Foias, {\em Navier-Stokes equations}, Chicago lectures in mathematics. Chicago (IL): University of
Chicago Press; 1989.

\bibitem{ConTankov}
R. Cont, P. Tankov, {\em Financial Modeling with Jump Processes}, Chapman Hall/CRC Financial Mathematics Series,
2004, Boca Raton.

\bibitem{P00} P. Constantin, {\em Energy spectrum of quasigeostrophic turbulence}, Phys. Rev. Lett. 89 (2002), 184501

\bibitem{Cowin}
S.C. Cowin, {\em Polar fluids}, Phys. Fluids,  11 (1968),
1919--1927.

\bibitem{Dai}
Y. Dai, W. Hu, J. Wu, B. Xiao,
 {\em The Littlewood-Paley decomposition for periodic functions and applications to the Boussinesq equations},  Analysis and Applications, doi: 10.1142/S0219530519500234.


\bibitem{DLW}
B. Dong, J. Li, J. Wu, \emph{Global well-posedness and large-time decay for the 2D micropolar equations}, J. Differential Equations  262, (2017), 3488--3523.

\bibitem{Dwxye}
B. Dong, J. Wu, X. Xu, Z. Ye, \emph{Global regularity for the 2D micropolar equations with fractional dissipation}, Discrete Contin. Dyn. Syst. 38(8) (2018), 4133--4162.

\bibitem{Dwye2019}
B. Dong, J. Wu,   Z. Ye,
\emph{Global regularity for a 2D tropical climate model with fractional dissipation}.
J. Nonlinear Sci. 29 (2019), no. 2, 511--550.


\bibitem{DZ}
B. Dong, Z. Zhang, \emph{Global regularity of the 2D micropolar fluid flows with zero angular viscosity}, J. Differential Equations  249 (2010),  200--213.

\bibitem{Erdo} M.E. Erdogan, \emph{Polar effects in the
apparent viscosity of suspension},  Rheol. Acta 9 (1970),
434--438.


\bibitem{Eringen1}
A. Eringen, \emph{Theory of micropolar fluids}, J. Math. Mech. 16 (1966), 1--18.

\bibitem{Er2} A.C. Eringen, \emph{ Micropolar fluids with stretch},  Int. J. Engng. Eci. 7 (1969), 115--127.

\bibitem{UU} U. Frisch, S. Kurien, R. Pandit, W. Pauls, S. Ray, A. Wirth and J. Zhu, \emph{ Hyperviscosity, Galerkin Truncation, and Bottlenecks in Turbulence}, Phys. Rev. Lett.,  101 (2008), 264502.


\bibitem{EGaldir}
G. Galdi, S. Rionero, \emph{A note on the existence and uniqueness of solutions of
the micropolar fluid equations}, Internat. J. Engrg. Sci. 15 (1977), 105--108.


\bibitem{Gill} A.E. Gill, {\em Atmosphere–Ocean Dynamics}, Academic Press, London, 1982.

\bibitem{JiuLiu} Q. Jiu, J. Liu, J. Wu, H. Yu, \emph{On the initial-and boundary-value problem for 2D micropolar equations with only angular velocity dissipation}, Z. Angew. Math. Phys.
68 (2017),  68:107.

\bibitem{jiuYu}
Q. Jiu, H. Yu, \emph{Global well-posedness for 3D generalized Navier-Stokes-Boussinesq equations}, Acta Math. Appl. Sin. Engl. Ser. 32 (2016), 1--16.

\bibitem{KPPonce}
T. Kato,  G. Ponce, \emph{Commutator estimates and the Euler and
Navier-Stokes equations}, Comm. Pure Appl. Math. 41 (1988), 891--907.


\bibitem{Katz} N. Katz, N. Pavlovi\'{c}, \emph{A cheap Caffarelli-Kohn-Nirenberg inequality for the Navier-Stokes equation with hyper-dissipation}, Geom. Funct. Anal. 12 (2002),  355--379.

\bibitem{Laskin}
N. Laskin, \emph{Fractional quantum mechanics and L$\rm\acute{e}$vy path integrals}, Phys. Lett. A 268 (2000), 298--305.

\bibitem{Lions}
J L. Lions, Quelques m$\rm\acute{e}$thodes de r$\rm\acute{e}$solution
  de probl$\rm\grave{e}$mes aux limites non lin$\rm\acute{e}$aires,
Dunod, Gauthier-Villars, 1969.

\bibitem{LiuWang}
J. Liu, S. Wang, \emph{Initial-boundary value problem for 2D micropolar
	equations without angular viscosity}, Commun. Math. Sci. 16 (2018), 2147--2165.

\bibitem{Lukaszewicz1}
G. Lukaszewicz, \emph{On nonstationary flows of asymmetric fluids}, Rend. Accad. Naz. Sci. XL Mem. Mat. (5) 12(1988), 83--97.

\bibitem{Lukaszewicz2}
G. Lukaszewicz, \emph{On the existence, uniqueness and asymptotic properties for solutions of flows of asymmetric fluids}, Rend. Accad. Naz. Sci. XL Mem. Mat. (5) 13(1989), 105--120.

\bibitem{Lukaszewicz3}
G. Lukaszewicz, {\em Micropolar Fluids. Theory and Applications, Modeling and Simulation in Science, Engineering
and Technology}, Birkh$\rm\ddot{a}$user, Boston, 1999.

\bibitem{MB2011}
A. Majda, A. Bertozzi, {\em Vorticity and Incompressible
Flow}, Cambridge University Press, Cambridge, 2001.

\bibitem{Metzlerk}
R. Metzler, J. Klafter, \emph{The random walks guide to anomalous diffusion: a fractional dynamics approach}, Phys. Rep.
339 (2000), 1--77.

\bibitem{Nowakowski}
B. Nowakowski, \emph{Large time existence of strong solutions to micropolar equations in cylindrical domains}, Nonlinear Analysis: RWA 14 (2013), 635--660.

\bibitem{VVa} A. Pablo, F. Quiros, A. Rodriguez, J.L. Vazquez, {\em A fractional porous medium equation}, Adv. Math. 226 (2011) 1378-1409.

\bibitem{Pozrikidis}
C. Pozrikidis, {\em The Fractional Laplacian}, CRC Press, London, 2016.

\bibitem{Rojaszammmc}
M. Rojas-Medar, E. Ortega-Torres, \emph{The equations of a viscous asymmetric fluid: An interactive approach}, Z. Angew. Math. Mech. 85 (2005), 471--489.

\bibitem{ttTao}
T. Tao, \emph{Global regularity for a logarithmically supercritical hyperdissipative Navier-Stokes equation},
Anal. PDE (2009) 2(3), 361--366,

\bibitem{TYZ}
C. Tran, X. Yu, Z. Zhai, \emph{Note on solution regularity of the
generalized magnetohydrodynamic equations with partial dissipation},
Nonlinear Anal. 85 (2013), 43--51.

\bibitem{Woyczyh}
W. Woyczy$\rm\acute{n}$ski, \emph{L$\rm\acute{e}$vy processes in the physical sciences, L$\rm\acute{e}$vy processes}, Birkh$\rm\ddot{a}$user Boston, Boston, MA, 2001, pp. 241--266.

\bibitem{wujde3}
J. Wu, \emph{Generalized MHD equations}, J. Differential Equations 195 (2003), 284--312.

\bibitem{wujmfm}
J. Wu, \emph{Global regularity for a class of generalized magnetohydrodynamic equations}, J. Math. Fluid Mech. 13 (2011), 295--305.

\bibitem{Xiang}
Z. Xiang, W. Yan,   \emph{Global regularity of solutions to the Boussinesq equations with fractional diffusion}, Adv. Differential Equations 18 (2013),  1105--1128.

\bibitem{Xue}
L.  Xue, \emph{Well posedness and zero microrotation viscosity limit of
the 2D micropolar fluid equations}, Math. Methods Appl. Sci. 34 (2011), 1760--1777.

\bibitem{Yamaguchi}
N. Yamaguchi, \emph{Existence of global strong solution to the micropolar fluid systemin a bounded domain}, Math. Methods Appl. Sci. 28 (2005), 1507--1526.

\bibitem{Yamazaki2}
K. Yamazaki, \emph{Global regularity of logarithmically supercritical MHD system with zero diffusivity}, Appl. Math. Lett., 29 (2014), 46--51.

\bibitem{Yamazaki}
K. Yamazaki, \emph{On the global regularity of N-dimensional generalized Boussinesq system}, Appl. Math. 60 (2015), 109--133.

\bibitem{Yama}
K. Yamazaki, \emph{ Global regularity of logarithmically supercritical MHD system with improved logarithmic powers}, Dyn. Partial Differ. Equ.,  15 (2018),  147--173.

\bibitem{Yang1} W. Yang, Q. Jiu and J. Wu, \emph{The 3D incompressible Navier-Stokes equations with  partial hyperdissipation},  Math. Nachr. 292 (2019), No.8, 1823--1836.

\bibitem{Yeat15}
Z. Ye, \emph{A note on global well-posedness of solutions to Boussinesq equations with fractional dissipation}, Acta Math. Sci. Ser. B Engl. Ed. 35(1) (2015), 112--120.

\bibitem{Ye18jmaa}
Z. Ye, \emph{On regularity criteria of the 2D generalized MHD equations}, J. Math. Anal. Appl. 463 (2018), 989--1005.


\bibitem{Yeacap18}
Z. Ye, \emph{Some new regularity criteria for the 2D Euler-Boussinesq
equations via the temperature}, Acta Appl. Math. 157 (2018), 141--169.

\bibitem{ye2017good}
Z. Ye, \emph{Global regularity of the regularized Boussinesq
equations with zero diffusion}, submitted for publication (2017).

\end{thebibliography}
\end{document}